\documentclass[11pt]{article}
\usepackage{lmodern}
\usepackage{xurl}

\usepackage{hyperref}

\usepackage[toc,page]{appendix}
\usepackage[utf8]{inputenc}
\usepackage{amsmath, latexsym, stmaryrd, amsthm, caption, amssymb, bbm,
	,xcolor,subfig,bbm, multirow, multicol, makecell,graphics,pgfplots,bm,tabularx,a4wide,xurl,breakcites,enumitem,mathtools, comment,libertineRoman,authblk}
\PassOptionsToPackage{}{stmaryrd}
\SetSymbolFont{stmry}{bold}{U}{stmry}{m}{n}

\hypersetup{
	colorlinks=true,
	linkcolor=blue,         
	citecolor=blue,         
	urlcolor=blue,            
	pdfstartview=FitV
}
\numberwithin{equation}{section}
\usepackage[T1]{fontenc}
\usepackage[normalem]{ulem}

\newcommand{\tonde}[1]{\left( #1 \right)}

\newcommand{\di}{\mathrm d }

\newcommand{\ee}{\, \wedge \,} 
\newcommand{\ses}{\Leftrightarrow} 

\newcommand{\fre}{\ar@{-}} 

\newcommand{\eps}{\varepsilon}

\newcommand{\e}{{\rm e}}

\newcommand{\N}{\mathbb{N}}
\newcommand{\R}{\mathbb{R}}

\newcommand{\E}{\mathbb{E}}
\newcommand{\Z}{\mathbb{Z}}

\newcommand{\C}{\mathbb{C}}

\renewcommand{\S}{\mathbb S}
\renewcommand{\P}{\mathbb P}

\renewcommand{\di }{\, \mathrm{d}}

\renewcommand{\e}{\mathrm e}

\newcommand{\Wick}[1]{\mathopen{\vcentcolon} #1\mathclose{\vcentcolon}}

\DeclareMathOperator{\Cov}{Cov}

\DeclareMathOperator{\Var}{Var}

\DeclareMathOperator{\dW}{dist_W}
\DeclareMathOperator{\Vol}{Vol}
\DeclareMathOperator{\Span}{span}

\def\be{\begin{eqnarray}}
	\def\ee{\end{eqnarray}}
\def\ben{\begin{eqnarray*}}
	\def\een{\end{eqnarray*}}

\usepackage[nameinlink,capitalise,noabbrev]{cleveref}
\usepackage{aliascnt}

\newtheorem{theorem}{Theorem}[section]
\crefname{theorem}{theorem}{theorems}
\Crefname{theorem}{Theorem}{Theorems}

\newaliascnt{question}{theorem}
\newtheorem{question}[question]{Question}
\aliascntresetthe{question}
\crefname{question}{question}{questions}
\Crefname{question}{Question}{Questions}

\newaliascnt{lemma}{theorem}
\newtheorem{lemma}[lemma]{Lemma}
\aliascntresetthe{lemma}
\crefname{lemma}{lemma}{lemmas}
\Crefname{lemma}{Lemma}{Lemmas}

\newaliascnt{proposition}{theorem}
\newtheorem{proposition}[proposition]{Proposition}
\aliascntresetthe{proposition}
\crefname{proposition}{proposition}{propositions}
\Crefname{proposition}{Proposition}{Propositions}

\newaliascnt{corollary}{theorem}
\newtheorem{corollary}[corollary]{Corollary}
\aliascntresetthe{corollary}
\crefname{corollary}{corollary}{corollaries}
\Crefname{corollary}{Corollary}{Corollaries}

\newtheorem{assumption}{Assumption}
\crefname{assumption}{assumption}{assumptions}
\Crefname{assumption}{Assumption}{Assumptions}

\crefname{claim}{claim}{claims}
\Crefname{claim}{Claim}{Claims}

\theoremstyle{definition}

\newaliascnt{definition}{theorem}
\newtheorem{definition}[definition]{Definition}
\aliascntresetthe{definition}
\crefname{definition}{definition}{definitions}
\Crefname{definition}{Definition}{Definitions}

\newaliascnt{example}{theorem}

\aliascntresetthe{example}
\crefname{example}{example}{examples}
\Crefname{example}{Example}{Examples}

\newaliascnt{remark}{theorem}
\newtheorem{remark}[remark]{Remark}
\aliascntresetthe{remark}
\crefname{remark}{remark}{remarks}
\Crefname{remark}{Remark}{Remarks}

\pgfplotsset{compat=1.18}
\usepackage{subfig}

\title{Phase Transitions in the Fluctuations of Functionals of Random Neural Networks \\}
\author{Simmaco Di Lillo}
\author{Leonardo Maini}
\author{Domenico Marinucci}
\affil{RoMaDS - Department of Mathematics, University of Rome Tor Vergata, 	Rome, Italy \newline
	{\tt{{\{dilillo, maini, marinucc\}@mat.uniroma2.it}}}}

\date{}                     

\usepackage{titlesec}

\begin{document}

\maketitle

\begin{abstract}
	We establish central and non-central limit theorems for sequences of functionals of the Gaussian output of an infinitely-wide random neural network on the sphere $\mathbb{S}^d$. We show that the asymptotic behaviour of these functionals as the depth of the network increases depends crucially on the fixed points of the covariance function, resulting in three distinct limiting regimes: convergence to the same functional of a limiting Gaussian field, convergence to a Gaussian distribution, convergence to a distribution in the $Q$th Wiener chaos, $Q\ge2$. 
    Our proofs exploit tools that are now classical  (Hermite expansions, Diagram Formula, Stein-Malliavin techniques), but also ideas which have never been used in similar contexts: in particular, the asymptotic behaviour is determined by the fixed-point structure of the iterative operator associated with the covariance, whose nature and stability governs the different limiting regimes.

\end{abstract}
\noindent \textbf{Keywords}: Deep Random Neural Networks, Asymptotic Theory, Hermite Expansions, Malliavin-Stein Method.

\noindent \textbf{AMS Classification}:   60G60, 68T07
\renewcommand{\contentsname}{Structure of the work}
\cleardoublepage
\tableofcontents
\clearpage

\section{Introduction}

Neural networks are ubiquitous in the revolution that is taking place with machine learning and artificial intelligence techniques. A rapidly growing number of mathematical tools and techniques have been devoted to probe their main features; among these, a lot of interest has been drawn by the idea of investigating the properties of neural networks with random coefficients in various asymptotic limits. This approach goes back at least to \cite{neal1996priors},
who was the first to note that, in the infinite width regime, neural networks converge (in the sense of finite-dimensional distributions) to Gaussian random fields. This ansatz was later explored rigorously by \cite{g.2018gaussian,lee2018deep} and others; more recently, the result was extended to functional convergence by \cite{hanin2}, whereas a few other authors have strengthened convergence with quantitative central limit theorems, see, for instance, \cite{Goldstein2024,basteri2024quantitative,Torrisi2025, favaro2025quantitative,celli1,celli2} and the references therein. Neural networks with random coefficients are of interest for a number of reasons: they represent the model at initialization, before training has taken place; they can be viewed as a Bayesian prior, to be updated to a posterior by means of a Gaussian likelihood; and they can be viewed as a "generic" realization of a neural network with a given architecture, in order to explore its functional properties.

To understand the characteristics of different architectures, a natural approach is to investigate the spectral and geometric properties of the resulting fields, in particular in the regime where the depth (i.e., the number of layers, denoted by $L$) gets larger and larger (clearly the most relevant for current applications). In a recent contribution \cite{nostro}  it has been shown that the behaviour of the different architectures as the depth $L$ increases can be partitioned into three classes, according to the value of a single parameter that represents the derivative of the covariance function at the origin. In particular, when this parameter is smaller than unity (labelled the low-disorder case), the spectral mass of the neural network Gaussian field (NNGF, in short) concentrates at the origin; when this parameter is larger than unity (the high-disorder case), the spectral mass is not confined to any compact set, meaning that the behaviour of the field is dominated by components corresponding to higher and higher frequencies. When the parameter is exactly equal to unity (the sparse case, which includes the ReLU activation function, with a suitable normalization), there is a form of boundary behaviour, where the realizations of the NNGF converge to a constant function in the norm $L^2$, but do not exhibit any form of convergence in the stronger Sobolev sense. In a further contribution \cite{nostro2}, the expected geometry of the NNGF excursion sets was investigated; in particular, the expected boundary volume of the excursion regions was studied, showing that it could exhibit fractal, diverging, constant or vanishing values, again depending on the derivative of the covariance function of the field at the origin.

The present paper takes one further step in this direction, investigating the fluctuations of nonlinear functionals of the NNGF, in the limit where the depth goes to infinity. Our results cover, for instance, excursion volumes, but in general they allow to study the asymptotics for arbitrary local functionals of the NNGF. In particular, we establish central and non-central limit theorems, by means of techniques which can be viewed as a broad generalization (of some independent interest, we believe) of the classical techniques by Dobrushin and Major, Breuer and Mayor, and Taqqu (see \cite{BreuerMajor1983,DobrushinMajor1979,Taqqu1979,NourdinPeccati2011,MainiNourdin2024}, and the references therein). Anticipating some of the results below, we stress that our limiting distribution is typically non-Gaussian in the low disorder case: in fact, in this regime (under a suitable normalization) we establish a form of $L^2$ convergence to a non-degenerate, nonlinear transform of a limiting random field with bounded covariance function. On the other hand, in the high-disorder case after suitable normalization the limit can be Gaussian or non Gaussian, with a phase transition depending on the activation function and the dimension of the input space (the higher the dimension, the more difficult it becomes to have Gaussian behaviour). More often, the limiting law is of the second type, a peculiar distribution which we introduce here for the first time. The sparse case is similar to the high-disorder one in the sense that it can show both Gaussian and non-Gaussian behaviour, with a different normalization and phase transition. It is important to stress that our proofs exploit tools that are well-known in this literature (Hermite expansions, Diagram Formula, Stein-Malliavin techniques), but also ideas which apparently have never been used in these contexts. In particular, the asymptotic behaviour is determined by the fixed-point structure of the iterative operator associated with the covariance, whose nature and stability governs the different limiting regimes. In addition, the real-analytic structure of the covariance function plays a role in some technical parts of the proofs; see \Cref{sec: proofs} and \Cref{sec: auxiliary} for more details. Quantitative versions of our results are given in Wasserstein distance, using the now classical Stein-Malliavin approach discussed in \cite{NourdinPeccati2011}.

\subsection{Informal statements of the main results}\label{sec:informalmainresults}
We now introduce the notation and main results, at a high level of generality and omitting many details; this will allow us to explain our findings and compare them with the existing literature. Our definitions and formal statements will be given in full detail starting from \Cref{sec::background}.

In the sequel, $T_L$ denotes a NNGF of depth $L$ associated to a certain underlying activation function $\sigma$. More precisely, $T_L=(T_L(x))_{x\in\S^d}$ is a centred, continuous, isotropic Gaussian field with unit variance and covariance function defined recursively by composition (denoted by $\circ$)
\[
\E[T_L(x) T_L(y)]=:\kappa_L(\langle x,y \rangle)\,\,,\qquad \kappa_L=\underbrace{\kappa \circ \dots\circ \kappa}_{L \text{ times}}\,,
\]
where $\kappa:[-1,1]\rightarrow[-1,1]$, which satisfies $\kappa(1)=1$, is completely determined by the choice of the activation function $\sigma$.
In this paper, we will focus specifically on functionals of $T_L$ of the form
\begin{equation}
    \label{eq:first-def-functional}
F_L:=\int_{\S^d}\varphi(\widehat T_L(x))\,\di x\,,
\end{equation}
where $\varphi:\R\rightarrow\R$ is in $L^2(\R,e^{-x^2/2}\,dx)$ and $\widehat T_L$, whose covariance function is $\widehat \kappa_L$, is a rescaling of $T_L$ (the reason why we consider such a rescaling will become clearer later). Local functionals such as \eqref{eq:first-def-functional}  have been widely studied in the literature on the geometry of random fields, see the next subsections for precise references; for instance, choosing $\varphi=\mathbbm{1}_{[u,\infty)}$, $F_L$ becomes the so-called excursion volume (at level $u\in\R$) of $\hat T_L$, namely
\[
F_L:=\Vol\left(\{x\in\S^d\,:\,\widehat T_L(x)\ge u\}\right)\,.
\]


We will focus on the following question:
\begin{question}

    What is the limit in distribution for the standardized  fluctuations of the sequence $F_L$? Namely, as $L\rightarrow\infty$,
    \[
\widetilde F_L:=\frac{F_L-\E[F_L]}{\sqrt{\Var(F_L)}}\overset{d}{\longrightarrow}\,\,?\,
\]
\label{question1-1}
\end{question}

Before stating our main results, we need to introduce some more notation. Under the assumptions above, $F_L$ has an orthogonal expansion in $L^2(\Omega)$ (the so-called chaotic decomposition) of the form
\[
F_L=\E[F_L]+\sum_{q=Q}^\infty F_L[q]\qquad
\]
where $F_L[q]$ denotes the $q$th chaotic component of $F_L$ and $Q$ is the \textbf{Hermite rank} of $F_1$, defined as
\[
Q:=\inf\{\,q\ge 1\,:\,\Var(F_1[q])\neq 0\,\}\,.
\]
In our context, $Q$ is always $\ge2$.
As anticipated, the answer to \Cref{question1-1} is fully determined by the fixed-point structure of $\kappa$.  A first distinction based on the nature of fixed points was studied in \cite{nostro}, where three regimes were considered: the low-disorder regime $k'(1)<1$; the high-disorder regime $k'(1)>1$ and the (critical) sparse regime $k'(1)=1$. Then, it was shown that $\kappa'(1)$ governs an important connection between the geometric properties of the excursion sets of a random neural network and the rate of decay of its angular power spectrum. In this paper, we will further explore the nature of the fixed points of $\kappa$, and we will show how the answer to \Cref{question1-1} depends on higher order features that go beyond the derivative at $1$.

Let us now summarize and discuss our main results in the three regimes.

\subsubsection{Low-disorder: limiting fields and limiting functionals}
In the low-disorder regime, we will introduce some regularity conditions on the behaviour of the covariance function at the origin (see \Cref{ass::CRI} below). Under these conditions, we shall prove that $\widehat \kappa_L\rightarrow\widehat \kappa_\infty$ uniformly in the compact sets of $[-1,1]$, where $\hat \kappa_\infty$ is the covariance function of a well-defined limiting Gaussian field $\widehat T_\infty$. Moreover, we have (see \Cref{thm:low-disorder})
    \[
     F_L \overset{d}{\rightarrow} F_\infty:=\int_{\S^d}\varphi(\widehat T_\infty(x))\,\di x\,.
    \]
    The distribution of $F_\infty$ is generally not Gaussian, see \Cref{rem:distr of Finfty}.

\subsubsection{Sparse/high-disorder: sharp phase transition to non-Gaussian fluctuations}\label{informalHDeS}

In the other two regimes, we shall consider different technical assumptions: \Cref{ass::CRI} in the high-disorder case and \Cref{ass:cri-pot} in the sparse case. Here we prove that $\widehat\kappa_L$, once suitably rescaled, converges (pointwise) on $(-1,1)$ to a limit function $\widehat \kappa_\infty$, which is not the covariance function of a limit Gaussian field on $\S^d$ as in the low-disorder case. Indeed, we have
\begin{equation}\label{singularity-kinfty}
    \widehat \kappa_\infty(\langle x,y \rangle)\,\asymp \,d_{\S^d}(x,y)^{-\alpha}\,,
\end{equation}
where $d_{\S^d}$ denotes the geodesic distance in $\S^d$ and the order $\alpha>0$ of the singularity in the limit where $x\approx y$ depends
on the attraction/repulsion properties of the fixed points of $\kappa$. Concerning the asymptotic behaviour of the functionals, we prove in \Cref{thm:high-disorder} and \Cref{thm:sparse} (where $\alpha=2\nu$) that
\begin{equation}\label{eq:summaryourresults}
    \widetilde F_L\overset{d}{\longrightarrow}\begin{cases}
    Z\qquad&\text{ if }\alpha\ge d/Q\,,\\
    Z_Q\qquad&\text{ if }\alpha<d/Q\,;
\end{cases}\,
\end{equation}
here, $Z$ is a standard Gaussian variable, while $Z_Q$ is a non-Gaussian random variable whose law depends on the Hermite rank $Q$ and $\widehat \kappa_\infty$; these random variables are introduced in this paper for the first time, to the best of our knowledge. Moreover, we also provide precise variance rates and quantitative bounds in Wasserstein distance for these forms of convergence.

We stress that our results entail a sharp phase transition for the fluctuations of the sequence of functionals, depending entirely on $Q$ and on $\alpha$, the index of attractiveness/repulsiveness of the fixed points of $\kappa$. Our proofs hence require ideas based on the theory of fixed points and properties of real-analytic functions, see \Cref{sec: proofs,sec: auxiliary} for more details.

\subsection{Discussion: a comparison with previous results}

The determination of limit theorems for local functionals of Gaussian random processes and fields is a classical topic in probability, dating back at least to the seminal papers by Breuer, Dobrushin, Major, Rosenblatt, Taqqu, see \cite{BreuerMajor1983,DobrushinMajor1979,Rosenblatt1961,Taqqu1975,Taqqu1979}. We shall now briefly explain in what sense our focus here is novel and to what extent a comparison can be drawn with the existing literature.

We note first that the classical papers on nonlinear functionals of Gaussian fields are based on a very different notion of asymptotics rather than the one we are using here: namely large-domain or high-frequency, as we shall discuss in the two subsections below.

\subsubsection{Large-domain asymptotics on \texorpdfstring{$\Z^d$ and $\R^d$}{Z\^{}d and R\^{}d}}\label{sec:largedomain}
 The investigation of large-domain asymptotics for functionals of the form \eqref{eq:first-def-functional}
 started in the discrete setting of stationary Gaussian fields $(B(x))_{x\in \Z^d}$, with the seminal papers \cite{BreuerMajor1983,DobrushinMajor1979,Rosenblatt1961,Taqqu1975} that studied the fluctuations of
 \[
F_L:=\sum_{x\in\Z^d\cap LD}\varphi\left(B(x)\right)\,,\qquad\text{ as }L\rightarrow\infty\,,
\]
where $D$ is a compact set and $LD:=\{Lx\,:\,x\in D\}$. The continuous version of the problem consists of considering stationary Gaussian fields $(B(x))_{x\in \R^d}$ and functionals of the form
\[
F_L=\int_{LD}\varphi(B(x))\,dx\,, \qquad\text{ as }L\rightarrow\infty\,.
\]
Fluctuations in the continuous case were first investigated in \cite{Taqqu1979} and by following authors, to the point of becoming today a classical topic in probability; see, e.g., \cite{Arcones1994,LO14,MainiNourdin2024} and  \cite{LMNP24,MRZ25,MRV21,NT20} for many generalizations and extensions.

Central and non-central limit theorems have been proved under different conditions: integrability of the covariance function (introduced in the seminal paper by Breuer and Major \cite{BreuerMajor1983}); precise asymptotic conditions for the covariance function (see \cite{DobrushinMajor1979,GMT24,LMNP24,Maini2024}); or, more recently, conditions on the singularities of the spectral measure of $B$, see, e.g. \cite{Gass25,ILRS13,MainiNourdin2024}.

To simplify the exposition, assume that the underlying Gaussian field $B$ is centred, with unit variance and polynomially decaying, positive covariance function, namely
\begin{equation}
    \label{regvarassumption}
0<\Cov(B_x,B_y)=\rho(\|x-y\|)\sim \|x-y\|^{-\alpha}\,,\qquad\alpha>0\,.
\end{equation}
Under this assumption, we have a complete picture:
\begin{itemize}
    \item \textbf{(Short-range dependence).} If $z\mapsto \rho(\|z\|)\,\in\,L^Q$ ($\alpha>d/Q$), the celebrated Breuer-Major theorem (see \cite{BreuerMajor1983} and the many generalizations \cite{ADP24,CNN20,NPY19,NZ20}) yields Gaussian fluctuations.
    \item \textbf{(Critical case).} If $\alpha=d/Q$, then $z\mapsto \rho(\|z\|)\,\notin\,L^Q$, but one usually can prove Gaussian fluctuations with different techniques.
    \item \textbf{(Long-range dependence).} If $\alpha<d/Q$, then a series of seminal contributions by \cite{DobrushinMajor1979,Taqqu1975,Taqqu1979} yield non-Gaussian fluctuations for $Q\ge2$.  In particular, the limit distribution belongs to the $Q$th chaos and is of Hermite-type, see, e.g. \cite{DobrushinMajor1979,LMNP24,LO14,Taqqu1979} for different formulations and generalizations. Hermite distributions and processes have been widely studied and applied in the literature; see, e.g., \cite{ADT23,GNNS21,LNNT26,RST12}.
\end{itemize}
In a very concise way, under Assumption \eqref{eq:first-def-functional} we have
\begin{equation}\label{eq:summaryclassicalresults}
    \widetilde F_L \overset{d}{\longrightarrow}\begin{cases}
    Z\qquad&\text{ if }\alpha\ge d/Q\,,\\\mathcal{H}_{Q,\alpha}\qquad&\text{ if }\alpha<d/Q\,,
\end{cases}
\end{equation}
where $Z$ is a standard Gaussian and $\mathcal{H}_{Q,\alpha}$ has Hermite distribution depending only on $D,Q,\alpha$.

\subsubsection{High-frequency asymptotics}\label{sec:highfreq}
In the previous subsection, we reviewed some classical results in the framework of large-domain asymptotics. This framework is clearly meaningless when dealing with random fields whose domain is a fixed manifold, such as $\mathbb{S}^d$. Here, over the last 15 years, a number of different papers have considered geometric functionals under a different framework, namely, high-frequency/high-energy asymptotics. In these circumstances, one typically considers sequences of random fields corresponding to larger and larger values of some diverging parameters, typically the eigenvalues for sequences of random eigenfunctions; that is, the sequence of zero-mean, unit variance Gaussian random fields $T_{\ell}:\Omega \times \mathbb{S}^d \rightarrow \mathbb{R}$ that satisfy the following:
\begin{equation}\label{eigenfunctions}
\Delta_{\mathbb{S}^d}T_{\ell}=-\ell(\ell+d-1)T_{\ell}\ , \qquad  \, \ell \in \N
\end{equation}
where $\Delta_{\mathbb{S}^d}$ denotes the spherical Laplacian in dimension $d$, see \cite{AtkinsonHan} for more details. A number of variations of this model have been considered, for instance, by focusing on different manifolds (such as the torus) or considering weighted averages of the eigenfunctions, typically around small energy windows; for our purposes we do not need to go into details and focus only on model \eqref{eigenfunctions}. In this framework, many papers have considered the asymptotic behaviour of integrals of non-linear statistics for $T_{\ell}$, in analogy with what we do here but in the asymptotic regime (``high-energy'' or ``high-frequences'') where $\ell \rightarrow \infty$: see for instance \cite{MW_CMP2014,MR_JFA2015,Rossi_JTP2019,Todino_JMP2019} and the references therein. It is important to stress that in this framework one can obtain limiting behaviours where the asymptotic distribution is dominated by components belonging to a single chaotic components (typically the second or the fourth, $q=2$ or $q=4$), or situations where a countably infinite number of chaotic terms appears. In either circumstance, the asymptotic distribution of non-linear statistics was derived, and it was shown to be Gaussian in all cases; quantitative central limit theorems were also established. To sum up, to the best of our knowledge, there has been no natural model leading to non-Gaussian asymptotic distributions in the case of the sphere; it is of course possible to construct very easily artificial sequences which will lead to non-Gaussianity, but none which has emerged from principled motivations.

\subsubsection{Our framework: deep asymptotics}
The asymptotic framework considered in this paper does not fall into either of the two categories described above. Indeed, we study sequences of random fields defined over a $\mathbb{S}^d$, as in the high-energy asymptotic regime; however, the governing parameter is the number $L$ of iterations appearing in the covariance function $\kappa_{L}$.
This novel setting leads to both central and non-central limit theorems for functionals of Gaussian fields on the sphere. The non-Gaussian behaviour under this form of deep asymptotics differs significantly from that observed in the large-domain case; as noted earlier, non-Gaussian limits have not been observed so far in the high-frequency asymptotic setting.

Given the different nature of the asymptotics involved, it may seem meaningless to compare the results presented in \Cref{sec:informalmainresults} with the previous literature discussed in \Cref{sec:largedomain,sec:highfreq}. Nevertheless, in the following, we argue that a surprising and subtle parallelism can be identified between our results in the high-disorder and sparse regimes \eqref{eq:summaryourresults} and the classical results obtained in the large-domain setting \eqref{eq:summaryclassicalresults}.\\

Let us fix a centred, stationary Gaussian field $B=(B(x))_{x\in\R^d}$ with unit variance and assume \eqref{regvarassumption}.
Observe that, after the change of variable $x = Lx'$, addressing \Cref{question1-1} for functionals defined on the increasing domain $LD$ is equivalent to considering a fixed domain $D \subseteq \mathbb{R}^d$ and a sequence of Gaussian fields $(B_{Lx})_{x \in \mathbb{R}^d}$ with covariance $\rho(L\cdot)$. The corresponding sequence of functionals to be studied is as follows.
\[
F_L(D)=\int_D \varphi(B_{Lx})\,dx\qquad L\rightarrow\infty\,.
\]
Under Assumption \eqref{regvarassumption}, we have the following.
\[
\frac{\rho(L\|x-y\|)}{\rho(L)}\rightarrow \|x-y\|^{-\alpha}=:\rho_\infty(\|x-y\|)\,.
\]
On the other hand, if $(\widehat T_L(x))_{x\in\S^d}$ is the sequence of Gaussian fields considered in  \Cref{sec:informalmainresults}, we saw in~\Cref{informalHDeS} that its covariance function $\widehat \kappa_L$, once properly rescaled, converges to a limit function $\widehat \kappa_\infty$ satisfying
\eqref{singularity-kinfty}. This suggests an interesting analogy: in both cases, when we have a limit function ($\rho_\infty$ for large domains or $\widehat \kappa_\infty$ for deep NNGFs), the order $\alpha$
of the singularity of the limit function in $x\approx y$ determines the asymptotic distribution of fluctuations. Note, however, that the Hermite limit in \eqref{eq:summaryclassicalresults} and the non-Gaussian limit \eqref{eq:summaryourresults} are in general different, as their construction is based on different asymptotic aspects of the covariance function: the polynomial decay in the first case, and the fixed-point asymptotic dynamics in the second case.

To conclude, we observe another interesting parallelism, related to the short/long-range dependence, see again Subsection \ref{sec:largedomain}. In both cases, taking two disjoint compact sets $D_1,D_2$ (in $\R^d$ or $\S^d$), one could prove (see, e.g. \cite{Maini2024} for large-domain asymptotics and arguing as in \Cref{thm:variance chaos-iperuranio} for deep asymptotics) that the functionals $F_L(D_1)$, $F_L(D_2)$ are asymptotically uncorrelated (short-range dependence) if $\alpha\ge d/Q$  and asymptotically correlated (long-range dependence) if $\alpha<d/Q$. These analogies suggest the presence of a general structure that may emerge when limiting functions and their singularities are involved. However, the analysis of this general framework is beyond the scope of the present paper.


\subsection{Plan of the paper.}
The remainder of the paper is organized as follows.
In \Cref{sec::background} we collect the necessary background on neural networks, spectral representations in $\S^d$, chaotic decompositions and Malliavin-Stein method. Our main results are presented in \Cref{sec:main results} and proved in \Cref{sec: proofs}.
All remaining proofs, as well as auxiliary and technical results, are gathered in \Cref{sec: auxiliary}.


\section{Background}\label{sec::background}

In the sequel, all random variables are defined on a common probability space $(\Omega,\mathcal F,\mathbb P)$. We denote by $Z$ a standard Gaussian random variable, and by $\mathcal N(0,1)$ its law. The notation $\overset{d}{\rightarrow}$ stands for convergence in distribution, while $\sim$ denotes equality in distribution.
Throughout the paper, $\S^d$ denotes the $d$-dimensional unit sphere in $\mathbb R^{d+1}$, endowed with the uniform measure $\mathrm{d}x$.
Moreover, given two functions $f_n$ ad $g_n$ defined on $I\subseteq R$, we write $f_n(x)\lesssim g_n(x)$ for $x\in I$ if there exist constants $C>0$ and $N\in\mathbb N$ such that $f_n(x) \leq C g_n(x)$ for all $n \geq N$ and for all $x\in I$. Finally, we write $f_n(x) \asymp g_n(x)$ if both $f_n(x) \lesssim g_n(x)$ and $g_n(x)\lesssim f_n(x)$ hold.

We begin by recalling basic notions on deep neural networks and the Gaussian random fields that arise in the infinite-width limit.
We then review the spectral representation of isotropic Gaussian fields on $\S^d$ and introduce chaotic decompositions and Wiener-It\^o integrals, together with Malliavin-Stein bounds that will be used throughout the paper.

\subsection{Random Neural Networks}\label{sec:NN}

\noindent
Fix a depth $L \in \mathbb{N}$. Consider a fully connected random neural network
with layer widths $n_1,\dots,n_L$, activation function $\sigma:\mathbb{R}\to\mathbb{R}$. We associate with this architecture the random field $$T_L:\S^d \to \R^{n_{L+1}}\,,$$
defined recursively as follows. For every $x\in\S^d$, set
\[
T_s(x)=
\begin{cases}
W^{(0)}x + b^{(1)}, & s=0\,,\\[4pt]
W^{(s)}\,\sigma\!\big(T_{s-1}(x)\big) + b^{(s+1)}, & s=1,\dots,L\,,
\end{cases}
\]
where the activation function $\sigma$ is applied component-wise.  The weights and biases $\{(W^{(s)},b^{(s)})\}_{s\ge0}$ are assumed to be Gaussian independent random elements with suitable variances\footnote{The variances normalization ensures that the variance of the pre-activations remains constant across layers, which is crucial for the large-width asymptotics} (see \cite{nostro} for more details). In the remainder of the paper, we drop the bias terms $b^{(s)}$ for simplicity.

It is by now classical (see~\cite{pmlr-v97-hanin19a,hanin2,basteri2024quantitative,favaro2025quantitative,celli1,celli2} and references therein) that, as $n_1,\dots,n_L\to\infty$, the network output $T_L$ converges in distribution to an isotropic centred Gaussian random field whose $n_{L+1}$ components are independent and identically distributed. In particular, by~\cite{nostro}, if $K_L$ denotes the covariance function of one component of the limiting field, then
\begin{equation*}
    K_L(x,y)
    = \kappa_L(\langle x,y\rangle)
    := \underbrace{\kappa\circ \cdots \circ \kappa}_{L \text{ times}}
    (\langle x,y\rangle),
\end{equation*}
where the kernel $\kappa:[-1,1]\to\mathbb{R}$ is given by
\begin{equation}\label{eq:KL1}
    \kappa(u)
    := C\,
    \mathbb{E}\!\left[
        \sigma(Z_1)\sigma\!\big(uZ_1+\sqrt{1-u^2}\,Z_2\big)
    \right],
\end{equation}
 $Z_1,Z_2\overset{\mathrm{i.i.d.}}\sim \mathcal{N}(0,1)$ and $C$ is a constant that ensures $\kappa(1) =1$. \eqref{eq:KL1} yields important properties for $\kappa$, see \ref{p0}--\ref{p3} at the beginning of \Cref{sec: proofs}.\\
Hence, from now on $T_L=(T_L(x))_{x\in\S^d}$ will denote a centred, isotropic real-valued Gaussian field with covariance function $\kappa_L$. We refer to $T_L$ as a \emph{neural network Gaussian field} (NNGF) of depth $L$.

\subsection{Spectral representations on \texorpdfstring{$\S^d$}{S\^{}d}}\label{sec: spec rep}
Let $T=(T(x))_{x\in \S^d}$ be a centred isotropic Gaussian field on $\S^d$ with unit variance; as usual, we assume that $T:\Omega \times \S^d  \rightarrow \R$, is a jointly measurable map. We now briefly introduce spectral representations on $\S^d$ for $T$; for more details, see, e.g. \cite{AtkinsonHan,marinucci2011random,yadrenko1983spectral}.

\medskip\noindent
It is well-known that $T$ admits the following spectral representation:
\begin{equation}
	\label{eq: spherical spec}
	T(x,\omega) = \sqrt{C_0} \zeta_{00}(\omega)+ \sum_{(\ell,m)\in E } \sqrt{C_\ell}\,\zeta_{\ell m }(\omega) \,Y_{\ell m}(x), \qquad x\in \S^d, \; \omega\in \Omega
\end{equation}
where the above decomposition holds in  $L^2(\Omega \times \S^d, \mathbb P\otimes \di x )$, with:
\begin{itemize}
\item $E = \{ (\ell,m)\; | \; \ell \in \N\setminus\{0\}, \; m = 1, \dots, n_{\ell,d}\}$.
\item $\{ Y_{\ell m}\}_{(\ell,m)}$ orthonormal basis of spherical harmonics (Laplace-Beltrami eigenfunctions on $\S^d$).
\item $n_{\ell,d}$ the
	cardinalities
	of the eigenspaces spanned by $\{ Y_{\ell m}\}_m$, satisfying $n_{\ell,d}\asymp \ell^{d-1}$.
	\item $(\zeta_{\ell m})$  $\overset{\mathrm{ i.i.d.}}{\sim}\mathcal N(0,1)$ and $(C_{\ell})_\ell$ the angular power spectrum of $T$\footnote{To simplify the proofs later, we change here the standard notation, where $\sqrt{C_{\ell}}\zeta_{\ell m}$ is replaced by $a_{\ell m}\sim \mathcal{N}(0,C_{\ell})$.}.

\end{itemize}
\medskip
\noindent
It is well-known that the power spectrum $(C_\ell)_\ell$ uniquely determines the covariance function $\kappa$ of $T$ through the following expansion
\begin{equation}\label{scho}
\kappa(\langle x,y \rangle):=\E[ T(x) T(y) ]
=C_0 +
\sum_{\ell=1}^\infty
C_\ell \frac{n_{\ell,d}}{\omega_d}
G_{\ell,d}(\langle x ,y\rangle),
\end{equation}
where $\omega_d$ denotes the Lebesgue measure of $\S^d$ and $(G_{\ell,d})_\ell$ is the sequence of normalized Gegenbauer polynomials, characterized as the unique family of polynomials with degree of $G_{\ell,d}$ equal to $\ell$, satisfying $G_{\ell,d}(1) = 1$\footnote{We normalize Gegenbauer polynomials so that $G_{\ell,d}(1)=1$ rather than using orthonormal normalization.} and orthogonal with respect to the probability measure
\begin{equation}
\label{defmu}
    \di \mu(u)=\di \mu_d(u)
=\frac{\omega_{d-1}}{\omega_d}
\mathbf{1}_{[-1,1]}(u)(1-u^2)^{d/2-1}
\,\di u\,.
\end{equation}
In particular, with such normalization, one has the addition formula   (see e.g. \cite[Theorem 2.9]{AtkinsonHan})
\begin{equation}
    \label{defGell}
    G_{\ell,d}(\langle x,y \rangle)=\frac{\omega_d}{n_{\ell,d}}\sum_{m=1}^{n_{\ell,d}}Y_{\ell m}(x)Y_{\ell m}(y)\,\qquad x,y\in\S^d.
\end{equation}
{A useful tool that will be used in the sequel is the Gegenbauer linearization formula, see e.g. \cite[(4.1)]{Koornwinder2018}, that for $\ell,\ell'\ge 1$ ensures the existence of strictly positive coefficients $a_k(\ell,\ell')$ such that
\begin{equation}
    \label{eq:linearizationformula}
    G_{\ell,d}(u)G_{\ell',d}(u)=\sum_{k=0}^{\ell\wedge\ell'}a_k(\ell,\ell')\,G_{\ell+\ell'-2k,d}(u)\,.
\end{equation}
By this result and an induction argument on $r$, we get that $\forall$ $r\ge1$, there exists a strictly positive array of deterministic weights $b_{k,r}$ such that
\[
G_{\ell,d}(u)^{2r}=\sum_{k=0}^{r\ell}b_{k,r}G_{2r\ell-2k,d}(u)\,.
\]
Since $G_{\ell,d}(u)^{2r}\in L^2(\mu)$ and all the positive coefficients in its expansion $L^2$ correspond to Gegenbauer polynomials $G_s$ with $s=2rl-2k$ even, we get for every $r\ge1$
\begin{equation}
\label{positivity of powers of Gell}
    \int_{-1}^1 G_{\ell,d}(u)^{2r+1}\di\mu(u)>0\qquad  \iff \qquad\ell \,\,\text{ even }\,.
\end{equation}
}

\smallskip
\noindent
In the sequel, we will denote the power spectrum of the NNGF $T_L$, with covariance function $\kappa_L$, as $(C_\ell(L))_{\ell}$. Then, using ~\eqref{scho} and orthogonality, noting that $\|G_{\ell,d}\|^2_{L^2(\mu)}=1/n_{\ell,d}$, we have
\begin{equation}\label{C-orto}
\begin{aligned}
C_0(L) &= \int_{-1}^1 \kappa_L(u) \di \mu(u)\,\,, \\
C_\ell(L) &= \omega_d  \int_{-1}^1 \kappa_L(u) G_{\ell,d}(u) \di \mu(u)\,\,,\qquad \ell\ge 1\,.
\end{aligned}
\end{equation}

\smallskip
\noindent

\subsection{Chaotic decompositions, Wiener-It\^o integrals and Malliavin-Stein bounds.}

Here we collect some background material on isonormal Gaussian processes,
Wiener chaos, and Malliavin--Stein bounds; we refer to
\cite{NourdinPeccati2011} for more discussion and details.

\paragraph{Isonormal Gaussian framework and representation of the field.}
Given a real separable Hilbert space $\mathfrak H$, an isonormal Gaussian process $X=(X(h))_{h\in\mathfrak H}$ is a collection of jointly Gaussian, centred random variables
satisfying $\E[X(h)X(g)]
=
\langle h,g\rangle_{\mathfrak H}$.
In particular, consider $\mathfrak H=\ell^2(E)$, where $E$ and $(\zeta_{\ell m})_{(\ell,m)\in E}$ are as above; we will focus on the isonormal Gaussian process
\[
X(h)=\sum_{(\ell,m)\in E}h(\ell,m)\,\zeta_{\ell m}\,.
\]
An isotropic Gaussian random field $T$ with angular power spectrum $(C_\ell)_\ell$ can be naturally studied within this Gaussian framework.
For each $x\in \S^d$, we denote by $Y_\cdot(x)\sqrt{C_\cdot}\in\mathfrak H$
the element defined by $(\ell,m)\longmapsto \sqrt{C_\ell}\,Y_{\ell m}(x).$
With this notation and the decomposition~\eqref{eq: spherical spec}, $T$ satisfies
\begin{equation}
    \label{repTinX}T=\left(X\!\big(Y_\cdot(x)\sqrt{C_\cdot}\big)\right)_{{x}\in\S^d}\,.
\end{equation}
\paragraph{Hermite polynomials and chaotic decomposition.}Consider $L^2(\mathbb R,e^{-x^2/2})$ and its orthogonal basis of Hermite polynomials $(H_q)_{q\ge0}$, characterized as the unique family of monic polynomials with a degree of $H_q$ equal to $q$, satisfying
\[
\E[H_q(Z)H_p(Z)] = q!\,\delta_{qp}\,.
\]
For $q\ge1$, the $q$th Wiener chaos $\mathcal H_q$ is defined as the closed linear subspace of $L^2(\Omega)$ generated by $\{H_q(X(h)):\;h\in\mathfrak H,\ \|h\|_{\mathfrak H}=1\}$.
It is a standard fact that $L^2(\Omega,\sigma(X))$ is the direct sum of all the Wiener chaoses; in other words, every square integrable random variable $F$ which is measurable with respect to $X$ has a chaotic decomposition
\begin{equation}
    \label{eq:chaodeco background}
    F=\E[F]+\sum_{q=1}^\infty F[q]\,\,,
\end{equation}
where $F[q]$ is the $q$th chaotic component  of $F$, i.e. its projection on $\mathcal{H}_q$.
In particular, in our context, a natural orthogonal basis of $\mathcal{H}_q$ is the collection of all Wick products of order $q$
\[
\left\{\Wick{\zeta_{e_1}^{a_1}\dotsb \zeta_{e_k}^{a_k}}\,\Big|\,k\in\N\,,\,\sum_{i=1}^k a_i=q\,,\,e_i\neq e_j\,,\,e_i\in E\right\}
\]
where for any $\zeta_{e_1}, \dotsb, \zeta_{e_q}$, setting $s_e:=|\{ i \; | \; e_i = e \}|\,,$ the Wick product is defined as follows
$$\Wick{\zeta_{e_1}\dotsb \zeta_{e_k}}\; \; \; := \prod_{ e\in \{ e_1, \dots, e_q\}} H_{s_e}(\zeta_e)\quad.   $$
\paragraph{Wiener--It\^o integrals.}
We denote by $\mathfrak H^{\otimes q}=\ell^2(E^q)$ the tensor Hilbert space, generated by the simple tensors $h_1\otimes\dotsb \otimes h_q(e_1,\dots,e_q):=h_1(e_1)\dotsb  h_q(e_q)$ and endowed with the standard inner product, first defined on simple tensors
\[
\langle f_1\otimes\dotsb\otimes f_q,g_1\otimes\dotsb\otimes g_q\rangle_{\mathfrak H^{\otimes q}}=\prod_{i=1}^q\langle f_i,g_i\rangle_{\mathfrak H}\,\,\qquad f_i,g_i\in\mathfrak H\,,
\]
and then extended by linearity and closure. We also denote by $\mathfrak H^{\odot q}=\ell_s^2(E^q)$ the subspace of $\mathfrak H^{\otimes q}$ made of symmetric tensors, i.e. functions in $\ell^2(E^q)$ invariant under permutations of the $q$ variables. For $h_1,\dots,h_q\in\mathfrak H$, we denote their symmetric tensor product (well defined by associativity)
\[
h_1\odot\dots\odot h_q:=\frac{1}{q!}\sum_{\tau\in S_q}h_{\tau(1)}\odot\dotsb\odot h_{\tau(q)}\,.
\]
For $h\in\mathfrak H$ with $\|h\|_{\mathfrak H}=1$, the $q$th Wiener--It\^o integral is defined by
\begin{equation}\label{defIq}
I_q(h^{\otimes q})=H_q(X(h))\,\qquad .
\end{equation}
By density\footnote{Using polarization identities, the set $\Span\{h^{\otimes q}\,,\,h\in\mathfrak H\,,\,\|h\|_{\mathfrak H}=1\}$ is dense in $\mathfrak H^{\odot q}$.}
$I_q$ extends uniquely to a linear isometry $I_q:\mathfrak H^{\odot q}\to\mathcal H_q$, satisfying
\[
\E[I_q(f)\,I_p(g)]=q!\langle f,g\rangle_{\mathfrak H^{\otimes q}}\,\delta_{pq}\,.
\]
By product formula (see e.g. \cite[(2.7.9)]{NourdinPeccati2011}), denoting by $\delta_e$ the Dirac function at $e$, we also have
\begin{equation}
\label{Wick Wiener-Ito}
    I_q(\delta_{e_1}\odot \dotsb \odot \delta_{e_q})=\Wick{\zeta_{e_1} \dotsb  \zeta_{e_q}}\,\qquad.
\end{equation}
\paragraph{Representation of the chaotic components of the functionals.}
In what follows, we will focus on the chaotic decomposition of functionals of $T$ of the form
\begin{equation}\label{chaotic-noL}
F=\int_{\S^d}\varphi(T(x))\,\di x\,\,,
\end{equation}
with $\varphi\in L^2(\R,e^{-x^2/2})$ that can be expanded in the orthogonal basis of Hermite polynomials
\[
\varphi(x)=\sum_{q=0}^\infty \varphi_q H_q(x)\,,\,\,\qquad\varphi_q:=\frac{1}{q!}\E[\varphi(Z)H_q(Z)]\,.
\]
By standard arguments (see e.g. \cite{NourdinPeccati2011}), $F$ has the following chaotic decomposition
\begin{equation}\label{214star}
F=\E[F]+\sum_{q=1}^\infty F[q]\,\,,\qquad F[q]=\varphi_q\,\int_{\S^d}H_q(T(x))\,dx\,.
\end{equation}
Moreover, by \eqref{repTinX}--\eqref{defIq} and the properties of $I_q$, $F[q]\in\mathcal{H}_q$ can be written as a Wiener-It\^o integral
\begin{equation}\label{Iq-Hq}
F[q]=\varphi_q\int_{\S^d}H_q(T(x))\,\di x
=\varphi_q\int_{\S^d}I_q(\sqrt{C_\cdot}^{\otimes q} {Y_\cdot(x)}^{\otimes q})\,\di x=
I_q\!\left(\varphi_q\,\mathcal G_q\,\sqrt {C}^{\otimes q}\right)\,,
\end{equation}
where $\mathcal G_q$ denotes the generalized Gaunt integral (see \cite[Section 3.5.2]{marinucci2011random})
\begin{equation}
    \label{defGaunt}
\mathcal G_q(e_1,\dots,e_q)
:=
\int_{\S^d}Y_{e_1}(x)\cdots Y_{e_q}(x)\,\di x\,.
\end{equation}
\paragraph{Malliavin-Stein bounds.}
In the proofs of our quantitative central limit theorems, we use the following result, which was proved in \cite{MRZ25} using the Malliavin-Stein method introduced in the seminal paper~\cite{Nourdin2009}.
In the sequel, $\dW$ will denote the Wasserstein distance, defined as follows
\[
\dW(F,G) = \sup\{|\E[h(F)]-\E[h(G)]|\,: \,h:\R\rightarrow\R\,\,,\,\, \|h\|_{\textrm{ Lipschitz}}\le 1 \}\,\,.
\]
\begin{proposition}[Proposition 2.7 on \cite{MRZ25}]\label{prop:malliavinstein}
	Let $Y$ be in $L^2(\Omega)$ with  finite chaotic decomposition
    \begin{equation}
    \label{finite Y deco}
    Y=\sum_{q=0}^N I_q(f_q)\,,\quad \quad N\in\N\,.
\end{equation}
where $f_q=c_q\mathcal{G}_q\sqrt{C_\cdot}^{\otimes q}\in\mathfrak{H}^{\odot q}$ for some coefficients $c_q$ such that $\Var(Y)=1$. Then, we have
	\begin{align}
    &\dW(Y, Z) \leq
	4N\sum_{q=1}^N
	3^{2q} q! \mathcal{M}_{q}\,\,,\nonumber\\
    \label{def_MP}
		\mathcal{M}_{q}^2 = \max_{1\leq r \leq q-1}  c_q^4\int_{(\S^d)^4}\kappa^r&(\langle x,y \rangle)\,\kappa^r(\langle z,w \rangle)\,\kappa^{q-r}(\langle x,z \rangle)\,\kappa^{q-r}(\langle y,w \rangle)\,dx\,dy\,dz\,dw\,.
	\end{align}
\end{proposition}
\begin{remark}
    The previous proposition is given in \cite{MRZ25} for general $f_q$ and replacing $\mathcal{M}_{q}$ with the maximum over all the contractions' norms. For the sake of brevity, we preferred here to avoid the definition of contractions and replace their norms with their explicit expressions when $f_q=c_q\mathcal{G}_q\sqrt{C^{\otimes q}}$, which has already been computed several times in the literature, see e.g. \cite{MR_JFA2015} and the references therein.
\end{remark}
\section{Main results}
\label{sec:main results}
Recall the definition of the NNGF $T_L$ with covariance function $\kappa_L$ in  \Cref{sec:NN} and of its power spectrum $(C_\ell(L))$ in  \Cref{sec: spec rep}. Our objects of interest are the functionals $F_L$ and their normalizations $\widetilde F_L$, defined as
\begin{equation}
    \label{ourfunc}
    \widetilde F_L := \frac{F_L - \E[ F_L]}{\sqrt{\Var(F_L)}},
\qquad
F_L:=\int_{\S^d}\varphi(\widehat T_L(x))\, \di x\,,
\end{equation}
where $\varphi\in L^2(\R,e^{-x^2/2})$ is non-linear and
\begin{equation}\label{eq:THat}
\widehat T_L(x)
:=
\frac{T_L(x)-\sqrt{C_0(L)}\zeta_{00}}
{\sqrt{1-C_0(L))}}\,\,
\end{equation}
is a centred, isotropic Gaussian field with unit variance and covariance function
\begin{equation}
    \label{defkappahatL}
    \widehat\kappa_L(u):=\frac{\kappa_L(u)-C_0(L)}
{1-C_0(L))}\,\qquad u\in[-1,1]\,\,.
\end{equation}
Our goal is to answer \Cref{question1-1}, i.e., understand the limit in distribution of $\widetilde F_L$ as $L\rightarrow\infty$, in the three regimes discovered in \cite{nostro}:
\begin{itemize}
\item \textbf{Low-disorder:} $\kappa'(1)<1$, see \Cref{thm:low-disorder};
\item \textbf{High-disorder:} $\kappa'(1)>1$, see \Cref{thm:high-disorder};
\item \textbf{Sparse:} $\kappa'(1)=1$, see \Cref{thm:sparse}.
\end{itemize}
\noindent
These regimes are reminiscent of the dynamics appearing in the study of attractive hyperbolic (low-disorder), non-hyperbolic (sparse), and repulsive hyperbolic (high-disorder) fixed points, see~\cite{fixed-point3,fixed-point1,fixed-point2} for more details on fixed point. As we shall see, different types of fixed points of $\kappa$ lead to different asymptotic behaviours of the $L$-fold iterated covariance $\kappa_L$, and hence to distinct limiting distributions for the sequence of functionals $\widetilde F_L$.
\begin{remark}
 The normalization \eqref{eq:THat} removes the integral mean
    \begin{equation*}
        \sqrt{C_0(L)}\,\zeta_{00}:= \int_{\S^d} T_L(x) \,\di x\,,
    \end{equation*}
    namely, the (random) sample average over the sphere, which plays an important role in the  behaviour described by the three regimes. In the low-disorder and sparse regimes, \(T_L-\sqrt{C_0(L)}\,\zeta_{00}\) was shown in \cite{nostro} to converge in \(L^2\) to $0$. \(\widehat T_L\) is the natural object to consider if one wants to study the spatially varying component of the field; this is natural in all three regimes: in the low-disorder and sparse cases, it removes the asymptotically dominant constant mode, while in the high-disorder case it singles out the proper spatial fluctuations.
    Note that normalization ensures that $\widehat T_L$ is centred, isotropic, and with unit variance.
\end{remark}

\medskip
\noindent
From now on, we will always consider the following assumption.

\begin{assumption}\label{ass::CRI} There exists $\gamma>1$ and a constant $c$ such that, as $u\to 1^-$,
    \begin{equation}\label{CRI}
1- \kappa(u)
=
(1-u)
-
c (1-u)^\gamma
+
o((1-u)^\gamma).
\end{equation}
\end{assumption}

\begin{remark}
The condition on $\kappa$ is fairly standard (see, e.g.,~\cite{bietti2021deep,nostro2,dilillo2025criticalpointsrandomneural,dilillo2026largedeviationprinciplesfunctional}) and is implied by suitable spectral decay properties.
In particular, it ensures that the field is almost surely of class $C^1$ (see, for example,~\cite[Proposition 3.16]{nostro2}). This assumption is of course related to the activation function $\sigma$ from which the NNGF $T_L$ is derived; basically for all activations
$\sigma$ commonly adopted in applications it is satisfied, see e.g. \cite[Remark 3.8]{nostro2}.
\end{remark}

\smallskip
\noindent

We can now state the first main result of the paper.

\begin{theorem}[Low-disorder regime]\label{thm:low-disorder}
Under \Cref{ass::CRI} and $\kappa'(1)<1$, we find that $\widehat T_L$ converges to a Gaussian field $\widehat{T}_\infty$ in $L^2(\Omega\times \S^d)$ and
\[
F_L
\overset{d}{\longrightarrow}
F_\infty
:=
\int_{\S^d}\varphi\big(\widehat T_\infty(x)\big)\,\di x.
\]
\end{theorem}
\noindent
In other words, \Cref{thm:low-disorder} highlights a conservative asymptotic behaviour: the limit in distribution of the sequence of functionals has the same functional structure, but evaluated in the limit field $\widehat{T}_\infty$, whose properties will be illustrated below. Let us give a sketch of the proof (again, the full argument will be discussed in~\Cref{low-dis-proof}).
\begin{proof}[Sketch of the proof of \Cref{thm:low-disorder}]
In the low-disorder regime the kernel is strictly contractive at its unique fixed point $u=1$ and the iterates $\kappa_L$ converge smoothly without developing singular behaviour.
As a consequence, the angular power spectrum stabilizes and no spectral mass accumulates at high frequencies. For every $L$, we have
\begin{equation}\label{gL}
\widehat T_L(x)
=
\sum_{(\ell,m) \in E} \sqrt{g_L(\ell,m)}\, \zeta_{\ell m} Y_{\ell m}(x)\,,\qquad
g_L(\ell,m)
:=
\mathbf{1}_{\ell\geq 1} \,
\frac{C_\ell(L)}{1-C_0(L)}\,,
\end{equation}
and in this regime the power spectrum $g_L$ of $\widehat T_L$ converges to a limiting power spectrum $g_\infty$ on $E$, which defines the Gaussian field $\widehat T_\infty$. Moreover, a dominated convergence of $g_L$ implies the convergence of the corresponding Gaussian fields.
The convergence in distribution of $F_L$ then follows directly from the chaotic expansion.
\end{proof}

\begin{remark}[On the distribution of $F_\infty$]\label{rem:distr of Finfty}
The distribution of $F_\infty$ is unknown and typically non-Gaussian. For instance, if $\varphi$ is bounded, then $F_\infty$ is itself bounded; if $\varphi=H_q$, $q\ge2$, then $F_\infty$ belongs to the $q$th Wiener chaos and hence is not Gaussian, see e.g. \cite[Corollary 5.2.11]{NourdinPeccati2011}.


\end{remark}

We now turn to the two remaining regimes. Since we assumed  $\varphi\in L^2(\R,e^{-x^2/2})$,  using the Wiener–Itô chaotic decomposition of $L^2(\Omega)$ , we obtain
\begin{equation}
\label{deco Y}
F_L-\E[F_L]
=
\sum_{q=1}^\infty F_L[q]\,\qquad F_L[q]=\varphi_q\int_{\S^d}H_q(\widehat T_L(x))\,dx
\end{equation}
and we define the Hermite rank by
\begin{equation*}\label{HH-rank}
Q
=
\inf\big\{
q\ge1 \;:\;
\
\Var(F_1[q])\neq0
\big\}.
\end{equation*}
\begin{remark}[Well-posedness of $Q$]\label{simmetria}
By symmetry arguments it is simple to see that the functionals we consider can be identically null in the high-disorder regime if and only if $\varphi$ and $\kappa$ are both odd (in the other regimes $\kappa$ cannot be  odd\footnote{We will prove that when $\kappa'(1)\le1$, $\kappa_L(x)\to 1$. If $\kappa$ is odd, $\kappa_L(-1)=-\kappa_L(1)=-1$, which yields a contradiction.}). For the sake of simplicity, from now on we rule out this case, which could be addressed by focusing on functionals computed on one hemisphere rather than the whole $\mathbb{S}^d$.
We shall show that $Q \geq 2$, 
and in particular $\Var(F_L[q]) = 0$ for every $q < Q$ and $L \geq 1$.
\end{remark}
\medskip
\noindent
In the sparse and high-disorder regimes, the limiting behaviour of the functionals $F_L$ depends on the Hermite rank $Q$ and on a parameter $\nu$ associated with the fixed-point structure of the kernel $\kappa$. In particular, a sharp phase transition between Gaussian and non-Gaussian fluctuations occurs at
$Q=d/(2\nu)$. More precisely, when $Q\ge d/(2\nu)$ the normalized functional $\widetilde F_L$ converges to a Gaussian limit, whereas for $Q<d/(2\nu)$ the limit is non-Gaussian and depends on the function arising in the asymptotics of the kernel iterates. The corresponding family of limit distributions is introduced in the following definition, for $Q\in \N$ with $Q\geq 2$.

\begin{definition}\label{Q-spherical}
Let $\beta$ be a function in $L^2([-1,1],\di\mu)$ and define
$$ g_\beta(\ell) = \int_{-1}^1 \beta(u) G_{\ell,d}(u) \di \mu(u) \,. $$
We say that $\beta$ is $Q$-admissible if
\begin{equation}\label{bbb}
\sum_{\ell_1, \dots, \ell_Q\in \N} \prod_{i=1}^Q  |g_\beta(\ell_i)| n_{\ell_i}  \int_{-1}^1 G_{\ell_1}(u) \dotsb  G_{\ell_Q} \di \mu(u)<\infty \,.
\end{equation}
Under this condition, we define $Z_Q(\beta)\in L^2(\Omega)$ as follows
\begin{equation}
    \label{defZQ}
    Z_Q(\beta)
:=
\sum_{\substack{e_1, \dots, e_Q\in E \\ e_i = (\ell_i, m_i)}}
\prod_{i=1}^Q
\sqrt{|g_\beta(\ell_i)|}  \mathcal{G}_q(e_1,\dots,e_Q)
:\zeta_{e_1} \dotsb \zeta_{e_Q}: \,.
\end{equation}
\end{definition}
\noindent
The fact that $Z_Q(\beta)$ is well defined in $L^2(\Omega)$ when \eqref{bbb} is satisfied is proved in \Cref{welldefinitionZQ}. We also denote $\widetilde Z_Q(\beta):=Z_Q(\beta)\,\Var(Z_Q(\beta))^{-1/2}$ if $\Var(Z_Q(\beta))\neq 0$.

\begin{remark} The condition~\eqref{bbb} holds if $\beta$ is the covariance function of a Gaussian  field
\[
T_\beta(x)
=
\sum_{ \ell=0}^\infty\sum_{m=1}^{n_{\ell,d}}
\sqrt{g_\beta(\ell)}\,\zeta_{\ell m}\,Y_{\ell m}(x)\,.
\]
Indeed, since $|G_{\ell,d}|\le1$ and $g_\beta\ge0$ if $\beta$ is a covariance function, condition~\eqref{bbb} follows by
\begin{equation}
\label{easycondforbb}
    \Var(T_\beta(x))=\sum_{\ell=0}^\infty g_\beta(\ell) n_{\ell,d}<\infty\,.
\end{equation}
In this case, one can show that (see \eqref{Iq-Hq} and \Cref{welldefinitionZQ}) $$Z_Q(\beta)\overset{d}{=}\int_{\S^d}H_q(T_\beta(x))\,\di x\,.$$
If \eqref{easycondforbb} does not hold, the previous representation does not work, but $Z_Q(\beta)$ is still in the $Q$-th chaos and equal to $I_q\Big(\mathcal{G}_q\sqrt{|g_\beta|}^{\otimes q}\Big)$ if \eqref{bbb} holds, see \Cref{welldefinitionZQ}.
\end{remark}


We are now ready to state the main result in the high-disorder regime. Recall that here we assume $\varphi$ non-odd when $\kappa$ is odd (see \Cref{simmetria}).
\begin{theorem}[High-disorder regime]\label{thm:high-disorder} Assume \Cref{ass::CRI}  and $\kappa'(1)>1$. Then $\kappa$ admits a unique fixed point $b\in(-1,1)$, satisfying $b\ge0$ and $\kappa'(b)\in[0,1)$. Assume that $\kappa'(b)\neq 0$ and set
\[
\nu
=
-
\frac{\log(\kappa'(b))}
{\log(\kappa'(1))}>0,
\]
Then $$\beta_L(x) : =  \frac{\kappa_L(x) - b}{(\kappa'(b))^L} \to \beta(x) \qquad x\in (-1,1) $$
where $\beta\in L^Q(\mu)$  if and only if $Q<d/(2\nu)$. Moreover,  the following conclusions hold:
\begin{enumerate}
\item[(i)] \textbf{Gaussian case.} If $Q>\frac{d}{2\nu}$, then
\begin{align*}
\Var(F_L) &\asymp \kappa'(b)^{\frac{d}{2\nu}L},
&\dW(\widetilde F_L,Z) &\lesssim  L^{-d/4}.
\end{align*}

\item[(ii)]\textbf{Critical case.} If $Q=\frac{d}{2\nu}$, then
\begin{align*}
\Var(F_L) &\asymp \kappa'(b)^{\frac{d}{2\nu}L} L,
&\dW(\widetilde F_L,Z) &\lesssim L^{-1/2}.
\end{align*}

\item[(iii)] \textbf{Non-central case.} If $Q<\frac{d}{2\nu}$ then $\beta$ is $Q$-admissible and
\begin{align*}
\Var(F_L) &\asymp \kappa'(b)^{QL},
&\widetilde F_L &\xrightarrow{d}  {\rm sgn}(\varphi_Q)\,\widetilde Z_Q(\widehat \beta)\,
\end{align*}
where $\widehat \beta: = \beta - \int_{-1}^1 \beta(u) \di \mu(u)$.
\end{enumerate}
\end{theorem}

\begin{remark}
The condition that rules out $ \kappa'(b)=0$ (possible only if $b=0$) is technical: indeed, in the latter case, the convergence to the fixed point $b=0$  is super-exponential~\cite[Chapter II.4]{fixed-point2}. A detailed analysis of this degenerate situation is left for future work. To the best of our knowledge, this situation does not arise for standard neural network kernels.
\end{remark}

\medskip
\noindent

Let us move to the sparse regime. In this case, the iterates of the kernel converge to a non-hyperbolic fixed point. The absence of hyperbolicity prevents a first-order analysis; consequently, stronger regularity assumptions are required at this point, beyond~\Cref{ass::CRI}.

\begin{assumption}\label{ass:cri-pot}
Assume $\kappa'(1)=1$ and one of the following two conditions holds.

\begin{enumerate}

\item[(A)] There exist $1<\gamma_1<\gamma_2$ and constants $c_1,c_2$ such that $\gamma_2>2\gamma_1-1$ and as $u\to 1^-$
\[
1-\kappa(u)
=
(1-u)
-
c_1(1-u)^{\gamma_1}
-
c_2(1-u)^{\gamma_2}
+
o((1-u)^{\gamma_2}).
\]

\item[(B)]There exist $1<\gamma_1<\gamma_2<\gamma_3$  and constants $c_1,c_2,c_3$ such that $\gamma_2=2\gamma_1-1$ and as $u\to 1^-$
\[
1-\kappa(u)
=
(1-u)
-
c_1(1-u)^{\gamma_1}
-
c_2(1-u)^{\gamma_2}
-
c_3(1-u)^{\gamma_3}
+
o((1-u)^{\gamma_3})\,.
\]
\end{enumerate}
\end{assumption}
\begin{remark}
This assumption is not very restrictive for the activation functions that are commonly adopted. It is satisfied, for instance, by the covariance kernels associated with ReLU and ReLU-like activations (see \Cref{rem::relu}); moreover, it covers all ``regular'' activations  (see \Cref{rem::sparse-}).
\end{remark}

In the sparse regime, the asymptotic behaviour is analogous to that of the high-disorder case, but with different variance asymptotics and convergence rates.

\begin{theorem}[Sparse regime]\label{thm:sparse}Under \Cref{ass:cri-pot}, let $\nu=\gamma_1-1$. Then, there exist $\beta_0$ and $\beta_1$ (depending only on $\kappa$) such that
$$ \beta_L(x)
:=
L^{\rho+1}
\Big(
\kappa_L(x)
-1
+\beta_0 L^{-\rho}
-\beta_1 L^{-(\rho+1)} \log L
\Big) \to \beta(x), \qquad  x\in[-1,1) $$

where $\beta\in L^Q(\mu)$  if and only if $Q<d/(2\nu)$. Moreover,  the following conclusions hold:
\begin{enumerate}
\item[(i)] \textbf{Gaussian case.} If $Q>\frac{d}{2\nu}$, then
\begin{align*}
\Var(F_L) &\asymp L^{-\frac{d}{2\nu}},
&\dW(\widetilde F_L,Z) &\lesssim (\log L)^{-d/4}.
\end{align*}

\item[(ii)]\textbf{Critical case.} If $Q=\frac{d}{2\nu}$, then
\begin{align*}
\Var(F_L) &\asymp L^{-\frac{d}{2\nu}}\log L,
&\dW(\widetilde F_L,Z) &\lesssim (\log L)^{-1/2}.
\end{align*}

\item[(iii)] \textbf{Non-central case.} If $Q<\frac{d}{2\nu}$, then $\beta$ is $Q$-admissible and
\begin{align*}
\Var(F_L) &\asymp L^{-Q},
&\widetilde F_L &\xrightarrow{d} {\rm sgn}(\varphi_Q)\,\widetilde Z_Q(\widehat \beta)\,\,.
\end{align*}
where $\widehat \beta: = \beta - \int_{-1}^1 \beta(u) \di \mu(u)$.
\end{enumerate}
\end{theorem}
We provide a sketch of the proof for both the sparse and the high-disorder regimes. The complete proof is given from~\Cref{asymp-sparse-high,central-critical,central-noncritical,non-central}
\begin{proof}[Sketch of the proofs of~\Cref{thm:high-disorder,thm:sparse}]Since the behaviour of the field is governed by the dynamics of the iterated kernel near the attractive fixed point, one needs to prove the existence of sequences $b_L$ and $w_L$ such that the renormalized iterates
\begin{equation}\label{betaL-iperuranio}
\beta_L(u) := \frac{\kappa_L(u) - b_L}{w_L}
\end{equation}
converge to a non-trivial limiting function $\beta(u)$ on $(-1,1)$. See~\Cref{fixed-point:high} and~\Cref{fixed-point-sparse}.

We also prove that the limiting function $\beta$ exhibits a power-type singularity at $1$ of order $\nu$. This singularity drives the rate of the integrals
\[
\int_{-1}^1 \widehat\kappa_L(u)^q \,\di \mu(u)\,,\qquad L\rightarrow\infty\,,
\]
and this fact yields a phase transition occurring at the critical threshold $q = \frac{d}{2\nu}$, which corresponds to the borderline between integrability and non-integrability. This control determines which chaotic components contribute to the limit.

\smallskip
\noindent
$\boldsymbol{Q < \frac{d}{2\nu}.}$  In this regime,
\[
\Var(F_L - F_L[Q]) = o\!\left(\Var(F_L[Q])\right), \quad L \to \infty,
\]
so that the $Q$-th chaos dominates. After normalization, it converges in $L^2(\Omega)$ to a non-Gaussian limit. The proof relies on tools from the theory of real-analytic functions.

\smallskip
\noindent
$\boldsymbol{Q = \frac{d}{2\nu}.}$ In this regime, the $Q$-th chaos still dominates, thanks to a logarithmic correction in the normalization. After normalization, it converges in distribution to a Gaussian limit. The result follows from a quantitative Fourth Moment Theorem.

\smallskip
\noindent
$\boldsymbol{Q > \frac{d}{2\nu}.}$ In this regime, no single chaos dominates: infinitely many chaotic components contribute at the same order. 
In this case, we control the truncated chaotic decomposition via a Malliavin--Stein bound. This, combined with suitable tail estimates, yields a central limit theorem.
\end{proof}

\begin{remark}
Because neural networks are typically applied for high-dimensional data ($d$ is large), a CLT requires either a very large $\nu$ (i.e. we need to choose $\sigma$ with $\nu$ high) or a very large Hermite rank $Q$ (unlikely for functionals commonly studied in the literature). For instance, in the case of ReLU activations, since $\gamma_1=3/2$, $\nu=1/2$,  a CLT holds if and only if $d\leq Q$.
\end{remark}

\section{Proofs}\label{sec: proofs}
In what follows, we prove the convergence theorems in the three different regimes.

\medskip\noindent
From now on, we always denote by $\kappa$ a kernel of the form~\eqref{eq:KL1}, with $\sigma \in \mathbb D^{1,2}$, that is:
\[
\sum_{q=0}^\infty q!\, q\, a_q(\sigma)^2 < \infty\,\qquad a_q(\sigma)=\frac{1}{q!}\E[H_q(Z)\sigma(Z)]=\frac{\kappa^{(q)}(0)}{q!}\,\,.
\]
Under this assumption, we have the following properties.

\begin{enumerate}[label=(P\arabic*), start=0]
\item\label{p0} $\kappa$ can be written in the following way
\begin{equation}\label{power-series}
\kappa(u) = \sum_{q=0}^\infty q!\,a_q^2\, u^q,
\end{equation}
where the above power series converges uniformly on $(-1,1)$.
    \item\label{p1} $\kappa(u) \leq |\kappa(u)|\leq \kappa(|u|) <1$ for every $u\in (-1,1)$.
    \item\label{p2} $\kappa\in C^1([-1,1])$ and $\kappa'(u) \leq \kappa'(1)$ for every $u\in [-1,1]$, and thus $\kappa$ is $(\kappa'(1))$--Lipschitz.
    \item \label{p3}  $\kappa$  analytic in $(-1,1)$ and $\kappa^{(s)}(u)>0$ for $u>0$, $s\in \N$. In particular $\kappa$ increasing in $(0,1]$.
\end{enumerate}
Moreover, if $\kappa\in C^1([-1,1])$ and admits the representation~\eqref{power-series}, then it can also be written in the form~\eqref{eq:KL1} with $\sigma\in\mathbb D^{1,2}$. In particular, for a NNGF $T_L$ with activation $\sigma$ and kernel $\kappa$ of the form ~\eqref{eq:KL1}, satisfying
\Cref{ass::CRI} or \Cref{ass:cri-pot} as in \Cref{thm:low-disorder}, \Cref{thm:high-disorder} and \Cref{thm:sparse}, we have that $\kappa\in C^1$, $\sigma\in \mathbb D^{1,2}$ and the above properties are all fulfilled.

\smallskip
\noindent
The proofs of~\ref{p0}--\ref{p3} and the proof of the equivalence are straightforward and are included for completeness in~\Cref{sec::prop}.

\subsection{Low-disorder regime}\label{low-dis-proof}
The next lemma describes the linearized behaviour of iterates near the fixed point $u=1$. After renormalization, they converge to a non-trivial limit profile, in a way reminiscent of Koenigs’ linearization theorem~\cite{MR1508749}; see also~\cite[Theorem 2.1]{fixed-point2} for a modern proof.
\begin{lemma}\label{lem:tecnical} Under \Cref{ass::CRI} and $\kappa'(1)<1$, for every $x\in [-1,1]$, we have
\begin{equation}\label{conve} \beta_L(x) = \frac{\kappa_L(x) - 1}{\kappa'(1)^L} \to \beta(x)\,,
\end{equation}
where $\beta$ is continuous  function with $\beta(x)<0$ for $x\in(-1,1)$.
\end{lemma}
\begin{proof}
The convergence and the fact that $\beta$ is negative in $(-1,1)$ is established in~\cite[Lemma~3.2]{dilillo2026largedeviationprinciplesfunctional}. Moreover, in ~\cite[Proof of Lemma~3.2, p.15]{dilillo2026largedeviationprinciplesfunctional} it is shown that $\beta_L(x)\le \beta_{L-1}(x)$, so by Dini's lemma  the convergence is uniform in $[-1,1]$ and $\beta$ is continuous.
\end{proof}
Using the previous lemma, we can now prove~\Cref{thm:low-disorder}

\begin{proof}[Proof of~\Cref{thm:low-disorder}] Let $e=(\ell,m)\in E$ with $\ell\geq1$. Using the definition of $C_{\ell}(L)$ (cf.~\eqref{C-orto}) and the orthogonality of Gegenbauer polynomials (recall $G_{0,d}\equiv  1$)
we have for $\ell\ge 1$
\begin{align} \frac{C_\ell(L)}{\kappa'(1)^L} &=  \omega_d\int_{-1}^1 \frac{\kappa_L(u)}{\kappa'(1)^L}  G_{\ell,d}(u) \di \mu (u) \nonumber\\
&=\omega_d \int_{-1}^1 \frac{\kappa_L(u)-1}{\kappa'(1)^L}  G_{\ell,d}(u) \di \mu (u) \to \omega_d \int_{-1}^1 \beta(u) G_{\ell,d}(u) \di \mu (u)\, \label{subtraction}
\end{align}
where the convergence follows by Lemma \ref{lem:tecnical}. Since $ \mu$ is a probability measure, following  the same argument,  we obtain
\begin{equation}\label{starr} \frac{1-C_0(L)}{\kappa'(1)^L} \to -\int_{-1}^1 \beta(u) \di \mu(u)>0\,,  \end{equation}
where the limit is strictly positive because $\beta$ is negative in $(-1,1)$ by \Cref{lem:tecnical}. Thus  we have
$$ g_L(\ell,m): =
\mathbf{1}_{\ell\geq 1} \,
\frac{C_\ell(L)}{1-C_0(L)}  \xrightarrow[]{L \to + \infty}   \mathbf{1}_{\ell\geq 1} \frac{\omega_d \int_{-1}^1 \beta(u) G_{\ell,d}(u) \di \mu(u)}{ -\int_{-1}^1 \beta(u) \di \mu(u)} = : g_\infty(\ell,m) \,.$$
Moreover, in~\cite{nostro} the authors proved that, denoting by $X_L$ the integer valued random variable defined by
$$ \mathbb P(X_L = \ell) := C_{\ell}(L) \frac{n_{\ell,d}}{\omega_d},$$
we have $\E[ X_L^2] = O(\kappa'(1)^L)$ and hence by Markov's inequality $$C_{\ell}(L) \frac{n_{\ell,d}}{\omega_d} \le \P\tonde{X_L\ge \ell}\le \frac{\E[X_L^2]}{\ell^2}\lesssim \frac{k'(1)^L}{\ell^2}\,.$$ Therefore, using~\eqref{starr}, we have $1-C_0(L)\asymp k'(1)^L$ and hence one can bound $g_{L}$ as follows
\begin{equation}
    \label{crucial bound}
    g_L(\ell,m) \lesssim \frac{\kappa'(1)^L}{\ell^2 n_{\ell,d}(1-C_0(L))}\lesssim \frac{1}{n_{\ell,d}\,\ell^2}\,.
\end{equation}
The previous bound implies that
\[
\widehat  T_\infty(x):= \sum_{(\ell,m)\in E} \sqrt{g_\infty(\ell,m)}\,\zeta_{\ell,m}\,Y_{\ell m}(x)\,.
\]
is a well-defined isotropic, unit-variance random field on $\S^d$.

\noindent
Now,  using the independence of $(\zeta_e)$, we have
\begin{align*}  \E\Big[ |\widehat T_L(x) - \widehat T_\infty(x)|^2\Big]  &= \sum_{\ell=0}^\infty \sum_{m =1}^{n_{\ell,d}} \Big|  \sqrt{g_L(\ell,m)} - \sqrt{g_\infty(\ell,m)}\Big|^2  Y_{\ell m}(x)^2    \\
& = \sum_{\ell=1}^\infty \Big|  \sqrt{g_L(\ell,1)} - \sqrt{g_\infty(\ell,1)}\Big|^2  \frac{n_{\ell,d}}{\omega_d} \to 0
\end{align*}
where the last equality follows by \eqref{defGell} and since $g_L(\ell,m), g_\infty(\ell,m)$ are constant with respect to $m$, whence the convergence follows by the dominated convergence $g_L \to g_\infty$. Thus, $\widehat T_L \overset{L^2(\Omega\times \S^d)}{\rightarrow} \widehat T_{\infty}$ follows by the fact that the above $L^2$-distance is independent of $x$ and
\begin{equation}
\label{convL2}
\widehat T_L(x) \overset{L^2(\Omega)}{\longrightarrow} \widehat T_{\infty}(x), \qquad  x \in \S^d\, .\end{equation}
Now, using the Wiener chaos decompositions, Jensen's inequality and Tonelli's theorem we obtain
\begin{align*}
		\E[ |F_L - F_\infty|^2]  &= \E\Bigg[
        \Bigg| \int_{\S^d} \varphi \Big(\widehat T_L(x)\Big)  - \varphi\Big( \widehat T_\infty(x)\Big) \di x  \Bigg|^2 \Bigg]
		\\
& \lesssim
         \int_{\S^d} \sum_{q=0}^\infty  \varphi_q^2\E\Big[ \big| H_q \big(\widehat T_L(x)\big)  - H_q\big(\widehat T_\infty(x)\big) \big|^2 \Big] \di x
\\
& =
         2 \int_{\S^d} \sum_{q=0}^\infty  q! \varphi_q^2 \Bigg( 1 - \Cov\Big(\widehat T_L(x), \widehat T_\infty(x)\Big)^q\Bigg) \di x
	\end{align*}
where the last identity follows using the Diagram formula~\cite[Prop. 4.15]{marinucci2011random}.
We conclude by dominated convergence. Indeed, from~\eqref{convL2} we have  $$\Cov\Big(\widehat T_L(x), \widehat T_\infty(x)\Big)\to 1\,. $$
To dominate, we observe that
$-1\leq \Cov\Big(\widehat T_L(x), \widehat T_\infty(x)\Big)\leq 1 $ and since $\varphi\in L^2(e^{-x^2/2})$, we have   $\sum q! \varphi_q^2< \infty$.
\end{proof}

\subsection{Asymptotic analysis in the sparse and high-disorder regimes}\label{asymp-sparse-high}

To prove the theorem in the remaining two cases, we need to analyse the decay of the integrals $
\int_{-1}^1 \widehat{\kappa}_L(u)^q \,\di \mu(u).$
The proofs differ, but the results can be unified using the following notation.

\medskip
\noindent  {\bfseries High-disorder.} One can prove that $\kappa$ admits a unique fixed point $b\in (-1,1)$ and that $b\ge0$ with $\kappa'(b)\in[0,1)$ (see~\Cref{fixed-point:high} for more details). The case $\kappa'(b)=0$, possible only if $b=0$, will be always excluded in our theorems, see \Cref{rem:superattractive}. We use the following notation
$$ v_L =  (\kappa'(b))^L, \qquad \nu = -\frac{\log(\kappa'(b))}{\log (\kappa'(1))}, \qquad \rho = \nu^{-1},  \qquad \beta_0 = 1-b\, .$$

\medskip
\noindent  {\bfseries Sparse.} Let $\gamma_1$ as in~\Cref{ass:cri-pot} then we denote with
$$ v_L = L^{-1}, \qquad \nu = \gamma_1-1,\qquad   \rho = \nu^{-1}, \qquad \beta_0 = ((\gamma_1-1)c_1)^{-\rho}\, . $$

Moreover, we will prove that  for every kernel $\kappa$ satisfying our assumptions, there exists a function $\beta$ (see \eqref{betaL-iperuranio}) such that $\beta\in L^q(\mu)$ $\iff$ $q< d/(2\nu)$ and, if $\beta\in L^2(\mu)$, we have for $\ell\ge 1$
\begin{equation}\label{puntuale} \frac{g_L(\ell)}{v_L} \to \frac{1}{\beta_0}\int_{-1}^1  \beta(u)G_{\ell,d}(u)\di \mu(u) = : g_\infty(\ell)\, .
\end{equation}
See~\Cref{fixed-point:high,fixed-point-sparse} and the corresponding proofs in~\Cref{fixed-point:high-tecn,fixed-point-sparse-tecn} for more details.

Using the unified notation introduced above, we will prove the following result in both the high-disorder regime (see~\Cref{cor:boun-hatK-high}) and the sparse regime (see~\Cref{cor:boun-hatK-sparse})
\begin{proposition} \label{cor:boun-hatK-iper}
Let $\widehat \kappa_L$ be as in \eqref{defkappahatL} and assume either \Cref{ass::CRI}, $\kappa'(1)>1$ and $\kappa'(b)\neq 0$, or \Cref{ass:cri-pot}. Then the following asymptotics holds for fixed $q\in\N$, $q\ge2$.

\smallskip
\noindent
If $q<d/(2\nu)$, as $L\to\infty$ we have
	\begin{equation*}
		\int_{-1}^1 \widehat \kappa_L(u)^q\,\di \mu(u)
		= v_L^{q}\Big(\beth_q+o(1)\Big)\,
	\end{equation*}
where
    $$\beth_q = \frac{1}{\beta_0^q} \int_{-1}^1 \widehat \beta(u)^q  \di \mu(u)\,,\,\,\,\qquad \widehat \beta:=\beta-\int_{-1}^1\beta(u)\di\mu(u)\,.
    $$
Moreover,  $\beth_q=0 \iff \kappa$ and $q$ are both odd.

\smallskip
\noindent
If $q\ge d/(2\nu)$ and $q$ is even whenever $\kappa$ is odd we have
$$ \int_{-1}^1 \widehat \kappa_L(u)^q \di \mu (u)\asymp
\begin{cases}
			v_L^q \log(v_L^{-1})&\textnormal{if } q=\frac{d}{2\nu} \\
			v_L^{d/(2\nu)}
            &\textnormal{if } q>\frac{d}{2\nu}
		\end{cases}\,\,\,\,\,\,.
        $$

\smallskip
\noindent
Moreover, we always have (with no restrictions on $q,\kappa$)
$$ \int_{-1}^1  |\widehat\kappa_L(u)|^q \di \mu (u)\lesssim
\begin{cases}
v_L^q &\textnormal{if } q<\frac{d}{(2\nu)} \\
			v_L^q \log(v_L^{-1})&\textnormal{if } q=\frac{d}{2\nu} \\
			v_L^{d/(2\nu)}
            &\textnormal{if } q>\frac{d}{2\nu}
		\end{cases}\,\,\,\,\,\,.
        $$

\end{proposition}
As a consequence of the previous results,  we obtain the following
\begin{corollary}
\label{thm:variance chaos-iperuranio} Let $F_L[q]$ be the $q$th chaotic component of $F_L$ (cf.~\eqref{deco Y}). Under the same assumptions of \Cref{cor:boun-hatK-iper}, we have the following asymptotics for fixed $q\in \N$, $q\geq 2$.

\smallskip
\noindent
If $q<d/(2\nu)$, setting $\gimel_q = \beth_q \omega_d^2 q! \varphi_q^2$, as $L\rightarrow\infty$ we have
	\begin{equation*}
		\Var(F_L[q])= v_L^{q} \Big( \gimel_q +o(1) \Big)\,\,\,.
	\end{equation*}

\smallskip
\noindent
If $q\ge d/(2\nu)$, $\varphi_q\neq 0$ and $q$ is even whenever $\kappa$ is odd, we have
\begin{align*}\label{eq: variance chaos non int case}
		\Var(F_L[q])\,\asymp \,
		\begin{cases}
			v_L^q \log(v_L^{-1})\, &\textnormal{if } q=d/(2\nu) \\
			v_L^{d/(2\nu)}\,\, &\textnormal{if }q>d/(2\nu)
		\end{cases}
		\qquad\,.
	\end{align*}
\end{corollary}

\begin{proof} By the definition of $F_L[q]$, the Diagram Formula and a standard spherical change of coordinate we obtain
$$ \Var(F_L[q]) = q! \varphi_q^2 \int_{(\S^d)^2} \widehat\kappa_L(\langle x, y\rangle)^q \di x \di y  = q! \varphi_q^2 \omega_d^2\int_{-1}^1 \widehat \kappa_L(u)^q \di \mu(u)\,. $$
    Now, the claim follows by~\Cref{cor:boun-hatK-iper}.
\end{proof}

\subsection{Central limit theorem in the critical case}\label{central-critical}
In this section, we prove the central limit theorem in the case $Q = d/(2\nu)$. To do this, we need the following lemma,
whose proof is deferred to~\Cref{sec: classic for contractions-tecn} (see \cite[Section 4]{MR_JFA2015} for some related results).

\begin{lemma}[Bound for contractions]
	\label{lem: classic for contractions} 	 Let $r_1, r_2\geq 0$ and let $f:(-1,1)\to \R$. Then
	\begin{align*}
		&\int_{(\S^d)^4}f(\langle x,y \rangle)^{r_1}f(\langle z,w \rangle)^{r_1}f(\langle x,z\rangle)^{r_2}f(\langle y,w\rangle)^{r_2}\di x \di y \di z \di w
		\\
		&\lesssim \int|f(u)|^{r_1+r_2}\di  \mu(u)\int|f(u)|^{r_1}  \di \mu(u)\int|f(u)|^{r_2} \di \mu(u)\,.
	\end{align*}
\end{lemma}

We prove now our critical CLTs, i.e. the cases (ii) in \Cref{thm:high-disorder,thm:sparse}.
\begin{theorem}\label{thm::CLT-high-critical}
    Let $F_L=F_L(\varphi)$ be as in \eqref{ourfunc}, with Hermite rank $Q$, $F_L[Q]$ as in \eqref{deco Y} and  $\varphi$ non-odd if $\kappa$ is odd. Let $\widehat \kappa_L$ be as in \eqref{defkappahatL} and assume either \Cref{ass::CRI}, $\kappa'(1)>1$ and $\kappa'(b)\neq 0$, or \Cref{ass:cri-pot}.  If $Q= d/(2\nu)$, then
\begin{align*}
		\Var(F_L)=\Var(F_L[Q])(1+o(1))\asymp v_L^{\frac{d}{2\nu}} \log(v_L^{-1})
	\end{align*}
	and we have the following quantitative CLT in Wasserstein distance
	\begin{equation*}
		\dW\left( \widetilde F_L , Z\right)\,\lesssim \,
	\frac{1}{\sqrt{\log(v_L^{-1})}}
		\qquad\,.
	\end{equation*}
\end{theorem}
\begin{proof}
We first analyze the variance. From the assumption on the symmetry of $\kappa$ and $\varphi$ and since $Q = d/(2\nu)$, one can use~\Cref{thm:variance chaos-iperuranio} and hence
\begin{equation*}
\Var(F_L[Q])
\asymp
v_L^{d/(2\nu)} \log(v_L^{-1})\,.
\end{equation*}
	Moreover, since $|\widehat\kappa_L|\le 1$ $\big( \widehat \kappa_L$ is the covariance of $\widehat T_L$: a unit-variance field$\big)$, we have
	\begin{align*}
\Var(F_L-F_L[Q])&= \sum_{q=Q+1}^\infty q!\varphi_q^2 \omega_d^2\int_{-1}^1 \widehat\kappa_L(u)^{q}\,\di\mu(u)
	\lesssim \int_{-1}^1 \widehat |\kappa_L(u)|^{Q+1}\,\di \mu(u) \lesssim  v_L^{d/(2\nu)}\,,
	\end{align*}
    where the second-last inequality follows by $\sum_{q}q!\varphi_q^2<\infty$ (since $\varphi\in L^2(\R,e^{-x^2/2})$) and the last ones by ~\Cref{cor:boun-hatK-iper}.
In particular,
\[
\dW\left(
\widetilde F_L ,
\frac{F_L[Q]}{\sqrt{\Var(F_L[Q])}}
\right)^2
\lesssim
\E\left[
\left(
\widetilde F_L-
\frac{F_L[Q]}{\sqrt{\Var(F_L[Q])}}
\right)^2
\right]
\lesssim
\frac{\Var(F_L-F_L[Q])}{\Var(F_L[Q])}
\asymp\frac{1}{\log(v_L^{-1})}\, .
\]
To conclude the proof,
it remains to prove that
\begin{equation*}
\dW\left(
\frac{F_L[Q]}{\sqrt{\Var(F_L[Q])}},
Z
\right)
\lesssim \frac{1}{\sqrt{\log(v_L^{-1})}}\,.
\end{equation*}
Using~\Cref{prop:malliavinstein} we have
\[
\dW\left(
\frac{F_L[Q]}{\sqrt{\Var(F_L[Q])}},
Z
\right)
\lesssim
\max_{1\leq r \leq Q-1}
\frac{\sqrt{h_L(r,Q-r)}}{\Var(F_L[Q])},
\]
where
\begin{align*}
h_L(r,p)&=
\int_{(\S^d)^4}
\widehat\kappa_L(\langle x,y \rangle)^r\;
\widehat\kappa_L(\langle z,w \rangle)^r\;
\widehat\kappa_L(\langle x,z\rangle)^{p}\;
\widehat\kappa_L(\langle y,w\rangle)^{p}\; \di x \di y\di z \di w\,.
\end{align*}
Now, using~\Cref{lem: classic for contractions} we obtain
\begin{align*} h_L(r,Q-r)
&\lesssim  \int_{-1}^1 |\widehat\kappa_L(u)|^r \di \mu(u)
\int_{-1}^1 |\widehat\kappa_L(u)|^{Q-r} \di \mu(u)
\int_{-1}^1 |\widehat\kappa_L(u)|^{Q} \di \mu(u)
\\
&\lesssim v_L^r v_L^{Q-r} v_L^{Q}\log(v_L^{-1})  = \Big( v_L^{d/(2\nu)}\Big)^2 \log(v_L^{-1})
\end{align*}
where the last inequality follows from~\Cref{cor:boun-hatK-iper}. Using
$\Var(F_L[Q])\asymp v_L^{d/(2\nu)}\log(v_L^{-1})$,
we obtain
\[
\frac{\sqrt{h_L(r,Q-r)}}{\Var(F_L[Q])}
\lesssim
\frac{1}{\sqrt{\log(v_L^{-1})}},
\]
which yields the claim.
\end{proof}

\subsection{Central limit theorem in the non-critical case}\label{central-noncritical}
In this section, we prove the central limit theorem in the case $Q > d/(2\nu)$, i.e. the cases (i) in \Cref{thm:high-disorder,thm:sparse}.
\begin{theorem}\label{thm::CLT-high-noncritical}
Let $F_L=F_L(\varphi)$ be as in \eqref{ourfunc}, with Hermite rank $Q$, $F_L[Q]$ as in \eqref{deco Y} and  $\varphi$ non-odd if $\kappa$ is odd. Let $\widehat \kappa_L$ be as in \eqref{defkappahatL} and assume either \Cref{ass::CRI}, $\kappa'(1)>1$ and $\kappa'(b)\neq 0$, or \Cref{ass:cri-pot}.  If $Q> d/(2\nu)$, then
    \begin{align*}
		\Var(F_L)\asymp
			v_L^{d/(2\nu)}
	\end{align*}
	and we have the following quantitative CLT in Wasserstein distance
	\begin{equation*}
		\dW\left(\widetilde F_L , Z\right)\,\lesssim  \frac{1}{\Big(\log(v_L^{-1})\Big)^{d/4}}\,.
	\end{equation*}
\end{theorem}
The proof of this result relies on a quantitative truncation argument.
The following lemma  (whose proof is deferred to~\Cref{sec:tailcontrol}) provides a uniform estimate on the tail of the chaotic expansion;
its proof is inspired by the strategy developed in~\cite[Lemma 3.1]{MRZ25}.
\begin{lemma}\label{lem_tailcontrol}
Under the same assumptions of \Cref{thm::CLT-high-noncritical}, there exist $M,L_0,c\ge1$ such that for all $N\ge M$ and $L\ge L_0$ we have
\begin{equation*}
    \E\left[\left(\widetilde F_L-\widetilde F_{L,N}\right)^2\right]\le \,c\, N^{-d/2}\,.
\end{equation*}
where $F_{L,N}:=\sum_{q=Q}^N F_L[q]$ is the truncated decomposition of $F_L$ and $\widetilde F_{L,N}:=F_{L,N}\,\Var(F_{L,N})^{-1/2}$.
\end{lemma}

We prove now our non-critical CLTs.

\begin{proof}[Proof of~\Cref{thm::CLT-high-noncritical}] We first analyze the variance.  Using the Wiener-Ito decomposition, we have
$$ \Var(F_L) = \sum_{q=Q}^\infty q! \varphi_q^2 \int_{(\S^d)^2} \widehat \kappa_L(\langle x, y\rangle)^q \di x \di y \,.$$
In particular, since $|\widehat\kappa_L|\leq 1$ and $\sum q! \varphi_q^2< \infty$, we have
$$ \Var(F_L[Q]) \leq \Var(F_L) \lesssim \int_{(\S^d)^2} | \widehat \kappa_L(\langle x, y\rangle )|^Q \di x \di y\,.$$
Hence, by~\Cref{cor:boun-hatK-iper} and ~\Cref{thm:variance chaos-iperuranio}  we have $\Var(F_L)\asymp v_L^{d/(2\nu)}$.

\medskip\noindent
Let us move to the quantitative CLT. Let $(N_L)_L$ be a diverging sequence of natural numbers with $N_L\geq Q$. By triangular inequality,
\begin{equation}\label{ttriang} \dW(\widetilde F_L, Z) \leq \dW(\widetilde F_L, \widetilde F_{L,N_L}) + \dW(\widetilde F_{L,N_L}, Z)
\end{equation}
where $\widetilde F_{L,N_L}$ is as in~\Cref{lem_tailcontrol} (replacing $N$ with $N_L$). To bound the first term in~\eqref{ttriang}, one can use~\Cref{lem_tailcontrol}. Indeed,
\[
\dW\left(\widetilde F_L\,,\,\widetilde
F_{L,N_L}\right)\lesssim \E\left[\left(\widetilde F_L-\widetilde F_{L,N_L}\right)^2\right]^{1/2}\lesssim N_L^{-d/4}\,.
\]
Let us bound the other term. Using~\eqref{Iq-Hq} we can rewrite $\widetilde F_{L,N_L}$ as
$$ \widetilde F_{L,N_L} = \sum_{q=0}^{N_L} I_q\Big(c_{q,L}\mathcal{G}_q\sqrt{g_L}^{\otimes q})$$ with
$$c_{q,L}:=\varphi_q\,\Var(F_{L,N_L})^{-1/2}\leq |\varphi_q| \,\Var(F_L[Q])^{-1/2} \leq  c\,|\varphi_q| v_L^{-d/(2\nu)}$$
where the second-last inequality follows since $N_L\geq Q$ and the last follows by~\Cref{cor:boun-hatK-iper} (with $c>0$ independent of $Q$ and $L$). Using~\Cref{prop:malliavinstein} we have
\[
\dW\left(\widetilde F_{L,N_L}\,,\,Z\right)\le 4N_L\sum_{q=1}^{N_L}
	3^{2q} q! \mathcal{M}_{q}\,.
\]
with
\[
\mathcal{M}_q\le c\, \frac{\varphi_q^4}{v_L^{d/(2\nu)}}\max_{r=1,\dots,q-1}\int_{-1}^1 |\widehat \kappa_L(u)|^r\di \mu(u)\int_{-1}^1 |\widehat \kappa_L(u)|^{q-r}\di \mu(u)
\]
Now, using~\Cref{cor:boun-hatK-iper} and since $|\widehat \kappa_L|\le 1$, we obtain, for some $c>0$ independent of $q$ and $L$ (WLOG $r\leq Q-r$)
\begin{equation*}\label{explicit bounds for contractions}
\int_{-1}^1 |\widehat \kappa_L(u)|^r\di \mu(u)\int_{-1}^1 |\widehat \kappa_L(u)|^{q-r}\di \mu(u)  \le c \, v_L^{d/(2\nu)}
		\begin{cases}
			v_L^{d/(2\nu)}\,\, &\textnormal{if }d/(2\nu)<r\le {Q-r}\\
			v_L^{d/(2\nu)}\Big( \log(v_L^{-1})\Big)^2\,\, &\textnormal{if }d/(2\nu)=r=Q-r\\
			v_L^{d/(2\nu)} \log(v_L^{-1})\,\, &\textnormal{if }d/(2\nu)=r<Q-r\\
			v_L^{r}\,\, &\textnormal{if }r<d/(2\nu)<Q-r\\
			v_L^{ r} \log(v_L^{-1}) \,\, &\textnormal{if }r<d/(2\nu)=Q-r\\
			v_L^{ q -d/(2\nu)}\,\, &\textnormal{if } r\le Q-r <d/(2\nu)
		\end{cases}\,.
	\end{equation*}
    Hence, for every $q\geq Q$ we have
\[
\mathcal{M}_{q}^2\le c\,\varphi_q^4 \, v_L^{\delta}
\]
where $\delta<\min\{1,\frac{d}{2\nu},q-\frac{d}{2\nu}\}$. Thus,
\[
\dW\left(\widetilde F_{L,N_L}\,,\,N\right)\lesssim N_L\,9^{N_L}v_L^{\delta}\,,
\]
where the  bound follows from $\sum q! \varphi_q^2<\infty$. Thus, setting $N_L=\lfloor \eps (\log v_L^{-1}) \rfloor$, choosing $\eps>0$ sufficiently small, we get, for $L$ sufficiently large
$$N_L 9^{N_L}v_L^{\delta}<N_L^{-d/4}$$
that concludes the proof.
\end{proof}

\subsection{Non-central limit theorem}\label{non-central}
In this section, we prove the non-central limit theorem in the regime $Q<d/(2\nu)$. We begin by showing that~\Cref{Q-spherical} is well posed.

\begin{lemma}\label{welldefinitionZQ} If $\beta$ is $Q$-admissible (cf. \eqref{bbb}),  then $Z_Q(\beta)$ in \eqref{defZQ} is well defined in $L^2(\Omega)$. Moreover, we have $Z_Q(\beta)=I_q(\mathcal{G}_q\sqrt{|g_\beta|}^{\otimes q})$.
\end{lemma}
\begin{proof}
First of all, by the addition formula (cf.~\eqref{defGell}) and the definition of generalized Gaunt integral (cf.~\eqref{defGaunt})  we have for some positive constant $c$
\[
\sum_{e_1,\dots,e_q\in E}\prod_{i=1}^q |g_\beta(\ell_i)|\,\mathcal{G}_q^2(e_1,\dots,e_q)=c\sum_{\ell_1,\dots,\ell_q\in\N}\prod_{i=1}^q |g_\beta(\ell_i)|n_{\ell_i,d}\int_{-1}^1 G_{\ell_1}(u) \dotsb  G_{\ell_q} \di \mu(u)<\infty\,.
\]
Thus, $\beta$ is $q$-admissible if and only if $\mathcal{G}_q\sqrt{|g_\beta|}^{\otimes q}\in \mathfrak H^{\odot q}$. Let $E_N=\{(\ell,m)\in E\,:\,\ell\le N\}$ and let  $h_N$ be the element of $ H^{\odot q}$ defined by
\[
h_N(e_1,\dots,e_q):=\sum_{e_1,\dots,e_q\in E_N}\mathcal{G}_q(e_1,\dots,e_q)\prod_{i=1}^q\sqrt{g_\beta}(\ell_i)\,\delta_{e_1}\otimes\dotsb\otimes\delta_{e_q}\,.
\]
 Using the linearity of $I_q$ and the definition of Wick product (cf.~\eqref{Wick Wiener-Ito}) we obtain
 \[
I_q(h_N)= \sum_{e_1,\dots,e_q\in E_N}\mathcal{G}_q(e_1,\dots,e_q)\prod_{i=1}^q\sqrt{|g_\beta(\ell_i)|}n_{\ell_i,d}\,\Wick{\zeta_{e_1}\dotsb\zeta_{e_q}}\,.
\]
Since $h_N \to \mathcal G_q \sqrt{g_\beta}$ in $\mathfrak H^{\odot q}$ and  $I_q/\sqrt{q!}$ is a linear isometry between $\mathfrak{H}^{\odot q}$ and $\mathcal{H}_q$, we have $Z_q(\beta)= I_q(\mathcal G_q\sqrt{g_\beta})$ is the $L^2(\Omega)$-limit of a sequence of $L^2(\Omega)$ variables $I_q(h_N)$ and hence the claim.
\end{proof}

In the proof of the non-central limit theorem, we use the following lemma, whose proof is deferred to~\Cref{lem:analytic conservation-tecn}.

\begin{lemma}\label{lem:analytic conservation} Let $f:(-1,1)\rightarrow\R$ be real analytic with radius of convergence at $0$ at least $1$, with $f^{(k)}(0)\ge0$ for $k\ge \bar{k}$ and $f\in L^p(\mu)$, $p\ge1$ integer.   Then
\[
\sum_{\ell_1, \dots, \ell_p \in \N} \prod_{i=1}^p   n_{\ell_i,d }\,\left(\int_{-1}^1f(u)G_{\ell_i}(u)\di \mu(u)\right)\widehat{\mathcal G_p}(\ell_1, \dots,\ell_p)=\int_{-1}^1f^p(u)\di \mu(u)\,.
\]
where \begin{equation}
    \label{eq:Ghat}
    \widehat{\mathcal G_p}(\ell_1, \dots,\ell_p) =  \int_{-1}^1 \prod_{i=1}^p  G_{\ell_i,d} (u) \di \mu(u)
\end{equation}
\end{lemma}
We prove now our non-central limit theorems, i.e. the cases (iii) in \Cref{thm:high-disorder,thm:sparse}.

\begin{theorem}\label{thm::noCLT-high-critical}Let $F_L=F_L(\varphi)$ be as in \eqref{ourfunc}, with Hermite rank $Q$, $F_L[Q]$ as in \eqref{deco Y} and  $\varphi$ non-odd if $\kappa$ is odd. Let $\widehat \kappa_L$ be as in \eqref{defkappahatL} and assume either \Cref{ass::CRI}, $\kappa'(1)>1$ and $\kappa'(b)\neq 0$, or \Cref{ass:cri-pot}.  If $Q<d/(2\nu)$, then
	\begin{align}\label{var non cent}
		\Var(F_L)\asymp v_L^{Q}\, \,
	\end{align}
	and we have the following convergence  in distribution
	\begin{equation*}
		\widetilde  F_L \overset{d}{\rightarrow} {\rm sgn}(\varphi_Q)\, \widetilde Z_Q(\widehat \beta)\,\,.
	\end{equation*}
where  $Z_Q(\widehat \beta)$ is the non-Gaussian  defined as  \eqref{defZQ}, with $\beta$ and $\widehat\beta$ as in \Cref{thm:high-disorder} (for the high-disorder) or \Cref{thm:sparse} (for the sparse).
\end{theorem}

\begin{proof}	Let us start from the variance. From the assumption on the symmetry of $\kappa$ and $\varphi$ and since $Q < d/(2\nu)$, one can use~\Cref{thm:variance chaos-iperuranio} and hence
\begin{equation}\label{vvvv}
\Var(F_L[Q])=v_L^{Q} \Big( \varphi_Q^2 D_q^2 +o(1)\Big) \,.
\end{equation}
with $D_q = \sqrt{Q! \beth_Q} \omega_d >0$. Now using the Wiener-Ito decomposition, since $|\widehat \kappa_L|\leq 1$ we have
\[
\Var(F_L-F_L[Q])\lesssim \sum_{q=Q+1}^\infty q!\varphi_q^2\int_{-1}^1 \widehat \kappa_L(u)^{Q}\,\di  \mu(u) = o(v_L^{Q})\,.
\]
where we use $\sum q! \varphi_q^2<\infty$. The previous bound  concludes the proof of \eqref{var non cent} and implies that
\[
\E\left[\left(\widetilde F_L- \frac{F_L[Q]}{\sqrt{\Var(F_L[Q])}}\right)^2\right]\rightarrow\,0\,.
\]
 By~\eqref{Iq-Hq} we have
$$ \frac{F_L[Q]}{\sqrt{\Var(F_L[Q])}}  = \frac{\varphi_q}{\sqrt{\Var(F_L[Q])}}\int_{\S^d} H_q(\widehat T_L(x)) \di x $$
and, using~\eqref{vvvv}, one can
conclude showing that
\[
\frac{\int_{\S^d}H_{Q}(\widehat  T_L(x))\di x }{v_L^{Q/2}} \overset{d}{=}I_{Q}\left(\mathcal{G}_{Q}\sqrt{\frac{g_L}{v_L}}^{\otimes Q}\right)
\overset{d}{\longrightarrow} I_{Q}\left(\mathcal{G}_{Q}\sqrt{g_\infty}^{\otimes Q}\right)\,.
\]
By Hermite isometry,  \Cref{welldefinitionZQ} and recalling the puntual convergence  $\frac{g_L}{v_L}\to g_\infty$ (cf. \eqref{puntuale}), we can show that
\begin{align*}
&I_{Q}\left(\mathcal{G}_{Q}\sqrt{\frac{g_L}{v_L}}^{\otimes Q}\right)\overset{L^2(\Omega)}{\longrightarrow}I_{Q}\left(\mathcal{G}_{Q}\sqrt{g_\infty}^{\otimes Q}\right)\quad \ses\quad \left\|\mathcal G_Q\Bigg(  \sqrt{\frac{g_L}{v_L}}^{\otimes{Q}} - \sqrt{g_\infty}^{\otimes Q}\Bigg)\right\|_{\mathcal{H}^{\otimes Q}} \to 0\\
	&  \quad \ses \quad \sum_{e = (e_1, \dots, e_Q)\in E^Q} \mathcal G_Q(e)^2\left| \prod_{i=1}^Q \sqrt{\frac{g_L(e_i)}{v_L}
    } -\prod_{i=1}^Q  \sqrt{ g_\infty (e_i)} \right|^2 \to 0 \\
	&  \quad \ses \quad \sum_{\ell_1, \dots, \ell_Q \in \N} \left| \prod_{i=1}^Q  \sqrt{ \frac{g_L(\ell_i,1)}{v_L}} -\prod_{i=1}^Q  \sqrt{ g_\infty (\ell_i,1)} \right|^2 \sum_{\substack{m_1, \dots, m_Q\\e_i = (\ell_i, m_i)\in E \; i=1, \dotsc,Q}} \mathcal G_Q(e_1, \dotsc, e_Q)^2 \to 0
\end{align*}
 where the last implication follows, since $g_L(\ell,m)$ is constant over $m$.

\noindent
We observe that for every $\ell_1, \dots, \ell_Q \in \N$ we have by addition formula  (cf. \eqref{defGell}), the definition of a generalized Gaunt integral (cf. \eqref{defGaunt}) and the orthogonality of spherical harmonics
\begin{align} \sum_{\substack{m_1, \dots, m_Q \\ e_i = (\ell_i, m_i)\in E \; i=1, \dotsc,Q}} \mathcal G_Q(e_1, \dotsc, e_Q)^2= & \sum_{m_1=1}^{n_{\ell_1,d}}\dotsb \sum_{m_Q=1}^{n_{\ell_Q,d}}  \int_{\S^d\times \S^d} \prod_{i=1}^Q Y_{\ell_i, m_i}(x)Y_{\ell_i, m_i}(y) \di x \di y   \nonumber\\
= &	\int_{\S^d\times\S^d} \prod_{i=1}^Q \sum_{m_i = 1}^{n_{\ell_i,d}} Y_{\ell_i,m_i}(x) Y_{\ell_i, m_i} (y) \di x \di y \nonumber\\
= & \int_{\S^d\times \S^d} \prod_{i=1}^Q \frac{ n_{\ell_i,d}}{\omega_d}G_{\ell_i,d}(\langle x, y\rangle ) \di x \di y\nonumber \\
= &\omega_d^{2-Q}\widehat{\mathcal G_Q}(\ell_1, \dots,\ell_p) \prod_{i=1}^Q n_{\ell_i,d} \label{additionmultiple}
	\end{align}
	where the second-last identity follows from the additional formula for spherical harmonics (cf.~\eqref{defGell}) and the last recalling the definition of ${\mathcal G_Q}$ (cf.~\eqref{eq:Ghat}). In particular,
we need to prove that
 $$S_Q  = \sum_{\ell_1, \dots, \ell_Q \in \N} \prod_{i=1}^Q  \left|\sqrt{ n_{\ell_i,d }\frac{g_L(\ell_i)}{v_L}} -\prod_{i=1}^Q  \sqrt{ n_{\ell_i,d}\, g_\infty (\ell_i)} \right|^2 \widehat{\mathcal G_Q}(\ell_1, \dots,\ell_Q) \to 0\,. $$
Since
$$ \left| \prod_{i=1}^Q  \sqrt{ n_{\ell_i,d }\frac{g_L(\ell_i)}{v_L}} -\prod_{i=1}^Q  \sqrt{ n_{\ell_i,d} g_\infty (\ell_i)} \right|^2  \leq 2 \prod_{i=1}^Q   n_{\ell_i,d }\frac{g_L(\ell_i)}{v_L} +2\prod_{i=1}^Q   n_{\ell_i,d} g_\infty (\ell_i)\,   $$
setting
\begin{align}
    &R_L:=\sum_{\ell_1, \dots, \ell_Q \in \N} \prod_{i=1}^Q  n_{\ell_i,d }\frac{g_L(\ell_i)}{v_L}\widehat{\mathcal G_Q}(\ell_1, \dots,\ell_Q)\\
    &R_\infty:=\sum_{\ell_1, \dots, \ell_Q \in \N} \prod_{i=1}^Q  n_{\ell_i,d }g_\infty(\ell_i)\widehat{\mathcal G_Q}(\ell_1, \dots,\ell_Q)\label{Rinf}
\end{align}
if we prove that $R_L \to R_\infty<\infty$, the claim follows by ~\eqref{puntuale} and generalized dominated convergence.
Since $\widehat{\beta}$ is the pointwise limit of $\widehat \kappa_L$, by \Cref{prop:real_analytic_limit} below also $\widehat \beta$ satisfies the hypothesis of~\Cref{lem:analytic conservation}. Thus, since
$$g_L(\ell)= \int_{-1}^1 \widehat \kappa_L(u) G_{\ell,d}(u)\di \mu(u), \qquad g_\infty(\ell) = \frac{1}{\beta_0} \int_{-1}^1 \widehat \beta(u) G_{\ell,d}(u) \di \mu(u)\,$$ using~\Cref{lem:analytic conservation} we obtain
$$ R_L = \frac 1 { v_L^Q} \int_{-1}^1\widehat \kappa_L(u)^Q\di \mu(u)\qquad R_\infty= \frac 1 {\beta_0^Q}\int_{-1}^1\widehat \beta(u)^Q \di \mu(u) = \beth_q\,.$$
We conclude using~\Cref{cor:boun-hatK-iper}.
\end{proof}

In the following two subsections, we prove ~\Cref{cor:boun-hatK-iper} and ~\eqref{puntuale} in both cases, high-disorder and sparse.

\subsection{Proof of~\texorpdfstring{\Cref{cor:boun-hatK-iper}}{Corollary~\ref{cor:boun-hatK-iper}} and \eqref{puntuale} in the high-disorder regime}\label{Prop:high}
The next lemma, whose proof is deferred to~\Cref{fixed-point:high-tecn}, establishes the existence and uniqueness of the interior fixed point $b$ in the high-disorder regime and describes the corresponding renormalized limiting function.

 \begin{lemma}\label{fixed-point:high}
Let $\kappa$ be such that $\kappa'(1)>1$. Assume that \Cref{ass::CRI} holds.
Then there exists a unique fixed point $b\in (-1,1)$, satisfying $b\ge0$ and $\kappa'(b)\in[0,1)$. Moreover, $\kappa'(b)=0$ is possible only if $b=0$. Assume that $\lambda:=\kappa'(b)\neq 0$ and set
\[
\nu = -\frac{\log(\kappa'(b))}{\log (\kappa'(1))}\,.
\]
Let $K=(-1,1)$ if $\kappa$ is either odd or even, and $K=[-1,1)$ otherwise. Then, for all $x\in K$, we have
\begin{equation}\label{betaL-high}
\beta_L(x) := \frac{ \kappa_L(x) - b}{\lambda^L} \to \beta(x).
\end{equation}
Moreover,
\begin{align}
\label{dominazione-high}
|\beta_L(x)|\lesssim (1-x^{2})^{-\nu}, &\qquad  x\in K\\
\label{bound-sup-high}
\beta_L(1-x \lambda^{n/\nu}) \geq c_b  \lambda^{-n} ,
&\qquad x\in\big( c_b\lambda^{1/\nu}, c_b\big),  \quad 0\leq  n \leq L
\end{align}
where $c_b = (1-b)/2$.
Moreover, we have  $0\leq \beta_L(x)\uparrow \beta(x)$ for $x\geq b$.
 \end{lemma}

Using the previous lemma, one can derive the following asymptotics for
\(\int_{-1}^1 \beta_L^q(u) \,\di \mu(u).
\)

\begin{proposition}\label{prop: moments of betaL}
Let $\beta_L$, $\beta$, $\lambda$ be as defined in \Cref{fixed-point:high}. Then, we have the following asymptotics for fixed $q\in\N$, $q\ge2$.

\smallskip
\noindent
If $q<\frac{d}{2\nu}$, then as $L\to\infty$
\begin{equation}
\label{eq: moment of betaL int case intro}
\int_{-1}^1 \beta_L^q(u)\,\di \mu(u)
=
\int_{-1}^1 \beta^q(u)\,\di \mu(u)\,(1+o(1))\,,
\end{equation}

\smallskip
\noindent
If $q\ge \frac{d}{2\nu}$, and $q$ is even whenever $\kappa$ is odd, then
\begin{align}\label{eq: moments of betaL non int case intro}
\int_{-1}^1 \beta_L^q(u)\,\di \mu(u)
\asymp
\begin{cases}
L & \textnormal{if } q=\frac{d}{2\nu},\\
\lambda^{(\frac{d}{2\nu}-q)L} & \textnormal{if } q>\frac{d}{2\nu}.
\end{cases}
\end{align}

\smallskip
\noindent
Moreover, we always have (with no restriction on $q, \kappa$)
\begin{align}\label{eq: moments of |betaL| intro}
\int_{-1}^1 |\beta_L(u)|^q\,\di \mu(u)
\asymp
\begin{cases}
1 & \textnormal{if } q<\frac{d}{2\nu},\\
L & \textnormal{if } q=\frac{d}{2\nu},\\
\lambda^{(\frac{d}{2\nu}-q)L} & \textnormal{if } q>\frac{d}{2\nu}\,.
\end{cases}
\end{align}
\end{proposition}

\begin{proof}

If $\kappa$ is odd and $q$ is odd, then $\beta_L^q$ is odd and therefore
\[
\int_{-1}^1 \beta_L^q(u)\,\di\mu(u)=0
\]
for every $L$.  We note that also $\beta$ is odd (indeed $b=0$) and hence the claim holds.

\medskip
\noindent
{\bfseries Case $\boldsymbol{q<\frac{d}{2\nu}}$ .}
The claim follows from dominated convergence.
Indeed, $\beta_L(u)\to\beta(u)$ pointwise on $(-1,1)$ and by~\eqref{dominazione-high}
\[
|\beta_L(u)|^q(1-u^2)^{\frac d2-1}
\lesssim
(1-u^2)^{-q\nu+\frac d2-1}.
\]
The right-hand side is integrable on $[-1,1]$ precisely when $q<\frac{d}{2\nu}$, which proves~\eqref{eq: moment of betaL int case intro}. \\

\medskip
\noindent
{\bfseries Case $\boldsymbol{q\ge\frac{d}{2\nu}}$.}
Let $r=1-c_b$, where $c_b$ is defined in~\Cref{fixed-point:high}.
The asymptotic behaviour of the integral is entirely driven by the singularity of $\beta$ at $u=1$.
Accordingly, we write
\[
\int_{-1}^1 \beta_L^q(u)\,\di\mu(u)
=
\int_{K_r}\beta_L^q(u)\,\di\mu(u)
+
\eta\int_r^1 \beta_L^q(u)\,\di\mu(u),
\]
where $K_r=(-r,r)$ and $\eta=2$ if $\kappa$ is even or odd, while $K_r=[-1,r)$ and $\eta=1$ otherwise.

Since $\beta_L$ is uniformly bounded on $K_r$, the first term is uniformly bounded in $L$; hence
\[
\int_{-1}^1 \beta_L^q(u)\,\di\mu(u)
\sim
\eta\int_r^1 \beta_L^q(u)\,\di\mu(u),
\]
provided that the latter diverges.

To estimate the dominant contribution, we decompose
\begin{equation}\label{I+II}
\int_r^1 \beta_L^q(u)\,\di\mu(u)
=
\int_r^{1-\epsilon_L}\beta_L^q(u)\,\di\mu(u)
+
\int_{1-\epsilon_L}^1\beta_L^q(u)\,\di\mu(u)
=
\mathrm{I} + \mathrm{II},
\end{equation}
where we set $1-r=(1-b)/2=c_b$ and  $\epsilon_L=(1-r)\lambda^{L/\nu}=c_b\lambda^{L/\nu}$.

\medskip
\noindent
{\itshape Upper bound for $\mathrm I$.}
Using~\eqref{dominazione-high},
\[
\mathrm{I}
\lesssim
\int_r^{1-\epsilon_L}(1-u)^{-q\nu}\,\di\mu(u).
\]
Since $\di\mu(u)\asymp (1-u)^{\frac d2-1}$ as $u\to1$,
\[
\mathrm{I}
\lesssim
\int_{\epsilon_L}^{1-r} x^{\frac d2-1-q\nu}\,\di x.
\]
Hence
\[
\mathrm{I}
\lesssim
\begin{cases}
\log(\epsilon_L^{-1})\asymp L
& \textnormal{if } q=\frac{d}{2\nu},\\
\epsilon_L^{\frac d2-q\nu}
\asymp
\lambda^{(\frac{d}{2\nu}-q)L}
& \textnormal{if } q>\frac{d}{2\nu}.
\end{cases}
\]

\medskip
\noindent
{\itshape Lower bound for ${\mathrm I}$.}
Partitioning the interval into geometric scales,
\[
\mathrm I
=
\sum_{i=0}^{L-1}
\int_{1-\epsilon_i}^{1-\epsilon_{i+1}}
\beta_L^q(u)\,\di\mu(u)\gtrsim
\sum_{i=0}^{L-1}
\lambda^{\frac{d}{2\nu}i}
\int_{c_b\lambda^{1/\nu}}^{c_b}
\beta_L^q(1-x\lambda^{i/\nu})
x^{\frac d2-1}\,\di x,
\]
where the last inequality follows from the change of variable $u = 1 - x\lambda^{i/\nu}$. Therefore, using~\eqref{bound-sup-high}, we obtain
\[
\mathrm I
\gtrsim
\sum_{i=0}^{L-1}\lambda^{(\frac{d}{2\nu}-q)i}
\asymp
\begin{cases}
L & \textnormal{if } q=\frac{d}{2\nu},\\
\lambda^{(\frac{d}{2\nu}-q)L} & \textnormal{if } q>\frac{d}{2\nu}.
\end{cases}
\]

\medskip
\noindent
{\itshape Estimate for $\mathrm{II}$.}
By definition of $\beta_L$ (see~\eqref{betaL-high}) and since $\kappa_L(u)\le1$, we have $\beta_L(u)\lesssim\lambda^{-L}$.
Hence
\[
\mathrm{II}
\lesssim
\lambda^{-qL}
\int_{1-\epsilon_L}^1 (1-u)^{\frac d2-1}\,\di u
\lesssim
\lambda^{(\frac{d}{2\nu}-q)L}
\lesssim \mathrm{I}.
\]

Combining~\eqref{I+II} with the previous bounds yields~\eqref{eq: moments of betaL non int case intro}.

\medskip
\noindent
{\bfseries Proof of~\eqref{eq: moments of |betaL| intro}.}
The same argument applies to $|\beta_L|^q$.
Indeed $|\beta_L|$ is uniformly bounded on $K_r$, while $\beta_L(u)\ge0$ for $u\ge r$.
\end{proof}
\begin{remark}\label{rem:superattractive}
The assumption $\kappa'(0)\neq0$ is imposed for technical reasons.
If $\kappa'(0)=0$, then $0$ becomes a super-attractive fixed point and the analysis requires a different renormalization scheme; we leave this case for future investigation.

The requirement $q$ even when $\kappa$ odd (which is relevant when $\beta\notin L^q(\mu)$) excludes the situation in which both $\kappa$ and $q$ are odd. In this case we trivially have
\begin{equation}
\label{eq:moment of betaL odd odd case}
\int_{-1}^1 \beta_L^q(u)\,\di \mu(u)=0
\qquad \forall\,L\in\N .
\end{equation}
An analogous statement holds for the variance of $F_L[q]$. Indeed, if both $\kappa$ and $q$ are odd,  then
\[
\Var(F_L[q])
=
q!a_q^2
\int_{(\S^d)^2}
\widehat\kappa_L(\langle x,y\rangle)^q
\,\di x\,\di y
=
0 .
\]
\end{remark}

\begin{remark}
We note that $\beta\in L^q(\mu)$ if and only if $q< \frac{d}{2\nu}$.

\smallskip \noindent
If $q<d/(2\nu)$,  by dominated convergence  we have $$\|\beta\|_{L^q(\mu)} = \lim_{L\to + \infty} \|\beta_L\|_{L^q(\mu)}<\infty\,.$$ Otherwise, for $q\geq d/(2\nu)$,we have $$
    \int_{-1}^1 |\beta(u)|^q\di \mu(u) \geq \int_{b}^1 |\beta(u)|^q \di \mu(u) = \lim_{L\to + \infty} \int_{b}^1 |\beta_L(u)|^q \di \mu(u) = + \infty $$
    where the identity follows by Beppo-Levi theorem's ($\beta_L$ is positive and increases on $[b,1)$.
\end{remark}

Using the previous proposition, we have the following corollary, that is ~\Cref{cor:boun-hatK-iper} in the high-disorder case.
\begin{corollary}
\label{cor:boun-hatK-high}
Let $\widehat \kappa_L$ be as in \eqref{defkappahatL}, $\widehat \beta$ as in \Cref{cor:boun-hatK-iper} and assume $\kappa'(1)>1$, $\lambda=\kappa'(b)\neq0$   and \Cref{ass::CRI}. Then, we have the following asymptotics for fixed $q\in\N$, $q\ge2$.

\noindent
If $q<\frac{d}{2\nu}$, as $L\to\infty$ we have
	\begin{equation*}
		\int_{-1}^1 \widehat \kappa_L(u)^q\,\di \mu(u)
		=\frac{\lambda^{Lq}}{(1-b)^q} \int_{-1}^1  \widehat \beta(u)^q\,\di \mu(u)\,(1+o(1))\,.
	\end{equation*}
{where $\,\int_{-1}^1  \widehat \beta^q\,\di \mu=0$  $\iff$ $\kappa$ and $q$ are both odd.}

    \noindent
If $q\ge \frac{d}{2\nu}$, and $q$ is even whenever $\kappa$ is odd, we have

$$ \int_{-1}^1 \widehat \kappa_L(u)^q \di \mu (u)\asymp \frac{\lambda^{Lq}}{(1-C_0(L))^q} \int_{-1}^1 \beta_L(u)^q \di \mu(u)\asymp
\begin{cases}
			\lambda^{Lq} L &\textnormal{if } q=\frac{d}{2\nu} \\
			\lambda^{\frac{d}{2\nu}L } &\textnormal{if } q>\frac{d}{2\nu}
		\end{cases}\,.
        $$
{Moreover, we always have (with no restrictions on $q,\kappa$)}
$$ \int_{-1}^1  |\widehat\kappa_L(u)|^q \di \mu (u)\lesssim
\begin{cases}
            \lambda^{Lq} &\textnormal{if } q<\frac{d}{2\nu} \\
			\lambda^{Lq} L &\textnormal{if } q=\frac{d}{2\nu} \\
			\lambda^{\frac{d}{2\nu}L } &\textnormal{if } q>\frac{d}{2\nu}
		\end{cases}\,.
        $$
\end{corollary}
\begin{proof} Recalling the definition of $C_0(L)$, we have
\begin{equation}\label{C0-b} \lim_{L\to + \infty}C_0(L) =\lim_{L\to + \infty} \int_{-1}^1 \kappa_L(u) \di \mu(u) = b\,.  \end{equation}
Indeed $\kappa_L(u) \to b$ and the convergence is dominated by $1$.

\noindent
Moreover, using the definition of  $\beta_L$ (cf..,~\eqref{betaL-high}), we obtain
\begin{equation}\label{widekappa}
\begin{aligned}
\widehat \kappa_L(u) = & \frac{\kappa_L(u) - C_0(L)}{1-C_0(L)} = \frac{ \lambda^{L}}{1-C_0(L)} \Bigg[ \beta_L(u)- \int_{-1}^1 \beta_L(v) \di \mu(v) \Bigg]
\end{aligned}
\end{equation}
and hence
\begin{equation}\label{cc} \int_{-1}^1 \widehat\kappa_L(u)^q \di \mu(u) = \frac{\lambda^{Lq}}{(1-C_0(L))^q} \sum_{k=0}^q \binom{q} k (-1)^{q-k} \int_{-1}^1 \beta_L(u)^k \di \mu (u) \Bigg( \int_{-1}^1 \beta_L(u) \di \mu(u)\Bigg)^{q-k}\,.\end{equation}

\medskip
\noindent
If $q< d/(2\nu)$, then $\beta\in L^k(\mu)$ also for $k<q$ and so the claim follows using~\Cref{prop: moments of betaL} and~\eqref{C0-b}. {Let us now prove that $\int \widehat \beta^q\di\mu=0$ $\iff$ $\kappa$ and $q$ are both odd. Since $\widehat\beta\in L^2(\mu)$ and we have the $L^2(\mu)$-decomposition
\begin{equation}
    \label{decobetahat}
    \widehat\beta(u)=\sum_{\ell\in\N}\,g_{\infty}(\ell)\,n_{\ell,d}\,G_{\ell,d}(u)\,.
\end{equation}
Moreover, by \Cref{fixed-point:high} and \eqref{widekappa}, the sequence
$\beta_L - \int \beta_L \,\di\mu$ converges pointwise to $\widehat{\beta}$ on $(-1,1)$.
By \ref{p0}--\ref{p3}, each $\beta_L$ is real analytic on $(-1,1)$ and satisfies
$\beta_L^{(k)}(0) \ge 0$ for all $k \in \N$. Hence, by \Cref{prop:real_analytic_limit},
the limit $\widehat{\beta}$ is also real analytic on $(-1,1)$ and satisfies
$\widehat{\beta}^{(k)}(0) \ge 0$ for all $k \in \N$.
\noindent
As a consequence, applying \Cref{lem:analytic conservation} and recalling \eqref{puntuale}, we obtain
\begin{equation}
\label{stringona}
\int_{-1}^1 \widehat{\beta}(u)^q \,\mu(\di u)
= \sum_{\ell_1, \dots, \ell_q \in \N} \prod_{i=1}^q n_{\ell_i,d}\, g_\infty(\ell_i)\,
\widehat{\mathcal G_q}(\ell_1, \dots, \ell_q),
\end{equation}
where $\widehat{\mathcal G_q}$ is given in~\eqref{eq:Ghat} and  is non-negative by~\eqref{additionmultiple}. Retaining only the diagonal terms, we obtain the lower bound
\[
\int_{-1}^1 \widehat{\beta}(u)^q \,\mu(\di u)
\ge C \sum_{\ell \in \N} n_{\ell,d}^q\, g_\infty(\ell)^q
\int_{-1}^1 G_{\ell,d}(u)^q \,\mu(\di u).
\]

We are now in a position to prove the claim. If $q$ is even, then
$\widehat{\beta}(u)^q \ge 0$ for all $u \in (-1,1)$ and $\widehat{\beta} \not\equiv 0$\footnote{Recall that $\beta(u)\to +\infty$ when $u\to 1^-$},
so that $\int \widehat{\beta}^q \,\di\mu > 0$.

Assume now that $q \ge 3$ is odd. We prove that
\[
\int \widehat{\beta}^q \,\di\mu = 0 \quad \Longleftrightarrow \quad \kappa \text{ is odd}.
\]
If $\kappa$ is odd, then $\kappa_L$ is odd and $b=0$, so that $\beta_L$ is odd for all $L$,
and therefore $\widehat{\beta}$ is odd. It follows that $\widehat{\beta}^q$ is odd, and hence
its integral vanishes.

Conversely, assume that $\kappa$ is not odd. Then $\widehat{\beta}$ is not odd.\footnote{If
$\widehat{\beta}$ were odd, then $\widehat{\beta}(-u) = -\widehat{\beta}(u)$ would contradict
the boundary behaviour implied by \Cref{fixed-point:high}.}
Using the parity property $G_{\ell,d}(-u) = (-1)^\ell G_{\ell,d}(u)$ together with
\eqref{decobetahat}, it follows that there exists at least one even index $\ell$ such that
$g_\infty(\ell) > 0$. Therefore, by \eqref{positivity of powers of Gell} and the lower bound
\eqref{stringona}, we conclude that $\int \widehat{\beta}^q \,\di\mu > 0$.
 }

\medskip
\noindent
If $q\ge d/(2\nu)$, then by~\Cref{prop: moments of betaL}, for every $1\le k\le q-1$ (WLOG  $k\le q-k$), we  get
$$ \int_{-1}^1 |\beta_L(u)|^k \di \mu(u) \cdot  \int_{-1}^1 |\beta_L(u)|^{q-k}\di \mu(u) \lesssim
		\begin{cases}
			\lambda^{(\frac{d}{2}-q) L}\,\, &\textnormal{if }d/(2\nu)<k\le {q-k}\\
			L^2\,\, &\textnormal{if }d/(2\nu)=k=q-k\\
			\lambda^{(\frac{d}{2\nu}-q+k )L} L\,\, &\textnormal{if }d/(2\nu)=k<q-k\\
			L\,\, &\textnormal{if }k<d/(2\nu)=q-k\\
			\lambda^{(d/(2\nu) -q+k)L}\,\, &\textnormal{if }k<d/(2\nu)<q-k\\
			1\,\, &\textnormal{if } k\le q-k <d/(2\nu)
            \end{cases}
$$
Thus, using Jensen's inequality, we have
$$
\Bigg| \int_{-1}^1 \beta_L(u)^k \di \mu(u) \Big( \int_{-1}^1 \beta_L(u)\di \mu(u)\Big)^{q-k}\Bigg| = o\Bigg( \int_{-1}^1 |\beta_L(u)|^q \di \mu(u)\Bigg)
$$
and
	\[
	\left|\int_{-1}^1 \beta_L(u)\di \mu(u)\right|^{q}\lesssim\begin{cases}
1\,\, &\textnormal{if } 1<d/(2\nu) \\
		L^q\,\, &\textnormal{if } 1=d/(2\nu) \\
		\lambda^{q(\frac{d}{2\nu}-1)L}\,\, &\textnormal{if }1>d/(2\nu)
	\end{cases}
	\qquad=o\left(\int_{-1}^1 |\beta_L(u)|^q\di \mu(u)\right).
	\]
Then, the dominating term into~\eqref{cc} comes from $k=q$ and hence the claim.  The proof that one has the same rates taking $|\widehat \kappa_L|$ instead of $\kappa_L$ follows by standard inequalities and in particular \eqref{eq: moments of |betaL| intro}.
\end{proof}
\begin{remark}

 The proof of~\eqref{puntuale} when $\beta\in L^2(\mu)$, i.e. when $2<d/2\nu$, mirrors the proof of the convergence of $\frac{g_L(\ell,m)}{\kappa'(1)^L}$ in the low-disorder case, but using \Cref{fixed-point:high} and dominated convergence.
\end{remark}

\subsection{Proof of~\texorpdfstring{\Cref{cor:boun-hatK-iper}}{the Corollary} and \eqref{puntuale} in the sparse regime}\label{Prop:sparse}
The proof of these cases is more similar to the proofs in the high disorder, but with different velocity and a different control for the integrals. First, we observe that~\Cref{ass:cri-pot} is satisfied for  activation functions commonly considered in the literature.

\begin{remark}\label{rem::relu}
Let $\kappa$ be the normalized kernel associated with the ReLU activation function
$\sigma(x)=0\vee x$. It is well known (see~\cite{cho2011analysisextensionarccosinekernels} for instance) that 
$$
\kappa(u)=\frac{(\pi-\arccos u)\,u+\sqrt{1-u^{2}}}{\pi}.
$$
An expansion as $u\to1^{-}$ yields
$$
1-\kappa(u)
= (1-u) - \frac{2\sqrt 2}{3\pi}(1-u)^{3/2}
  - \frac{1}{15\sqrt 2\,\pi}(1-u)^{5/2}
  + o\big((1-u)^{5/2}\big).
$$
Thus, the ReLU kernel satisfies~\Cref{ass:cri-pot} with $\gamma_2 > 2(\gamma_1-1)$
\end{remark}
\begin{remark}\label{rem::sparse-} Let $\kappa\in C^4([-1,1])$ be a normalized kernel associated with some activation functions. Then $\kappa$ satisfies~\Cref{ass:cri-pot} with $\gamma_2 = 2(\gamma_1-1)$. Indeed, by Taylor expansion, we have
$$ 1- \kappa(x) = (1-x) - c_1(1-x)^2  + c_2 (1-x)^3 - c_3 (1-x)^4 + o((1-x)^4)$$
where $c_i =\frac{ \kappa^{(i)}(1)}{i!}$.
\end{remark}
The next lemma is analogous to~\Cref{fixed-point:high} and the proof is deferred to~\Cref{fixed-point-sparse-tecn}
\begin{lemma}\label{fixed-point-sparse}
Let $\kappa$ satisfy~\Cref{ass:cri-pot} and set $\rho := (\gamma_1-1)^{-1}$.
There exist constants $\beta_0>0$ and $\beta_1$, depending only on $\kappa$,
such that the renormalized iterates
\begin{equation}\label{betaL}
\beta_L(x)
:=
L^{\rho+1}
\Big(
\kappa_L(x)
-1
+\beta_0 L^{-\rho}
-\beta_1 L^{-(\rho+1)} \log L
\Big),
\qquad x\in[-1,1),
\end{equation}
converge pointwise as $L\to\infty$ to a limit function  $\beta:[-1,1)\to\mathbb R$.

Moreover, setting $\nu = 1- \gamma_1$, there exist $L_1>0$ and constants $0<\delta<1$ and $C>0$ such that for all
$L\ge L_1$ one has
\begin{align}
\label{beta-upper}
|\beta_L(x)|
&\lesssim
(1-x^2)^{-\nu},
&& x\in (-1,1),\\[6pt]
\label{beta-lower}
\beta_L(x)
&\gtrsim
(1-x)^{-\nu} [ 1 - C(1-x)^{-\nu}L^{-1}]
&& x\in(1-\delta,1),
\end{align}
\end{lemma}
As in the high disorder, with this lemma we can prove the following proposition.
 \begin{proposition}\label{mom-sparso}With the same assumptions and notations of \Cref{fixed-point-sparse}, for fixed $q\in \mathbb N$ with $q\geq 2$, the following asymptotics hold.

 If $q< d/(2\nu)$, as $L\to \infty$ we have
 	\begin{equation}
\int_{-1}^1 \beta_L^q(u)\di \mu(u)=\,\int_{-1}^1  \beta^q(u)\di \mu(u)\,(1+o(1))\,\,
 	\end{equation}

If $q\ge d/(2\nu)$,
 	\begin{align}\label{eq: moments of betaL non int case intro sparso}
\int_{-1}^1 \beta_L^q(u)\di \mu(u)\,\asymp \,
 		\begin{cases}
 			\log L\,\, &\textnormal{if } q=d/(2\nu) \\
 			L^{q -d/(2\nu) }\,\, &\textnormal{if }q>d/(2\nu)
 		\end{cases}
 		\qquad\,.
 	\end{align}

Moreover, we have
\begin{equation}\label{eq: moments of |betaL| intro  sparso} \int_{-1}^1|\beta_L^q(u)|^q \di \mu(u) \,\asymp \,
 		\begin{cases}
        1 \,\, &\textnormal{if } q<d/(2\nu) \\
 			\log L\,\, &\textnormal{if } q=d/(2\nu) \\
 			L^{q -d/(2\nu) }\,\, &\textnormal{if }q>d/(2\nu)
 		\end{cases}
 		\qquad\,.
        \end{equation}
 \end{proposition}

\begin{proof}The proof closely follows that of the analogous proposition in the high-disorder case, with a few minor modifications.

\noindent
{\bfseries Case $\boldsymbol{q<\frac{d}{2\nu}}$.} The claim follows using the same argument as in the proof of the high-disorder.

\smallskip
\noindent
{\bfseries Case $\boldsymbol{q\geq\frac{d}{2\nu}}$.}
Let $C$ and $\delta$ are as in~\Cref{fixed-point-sparse}.
Following the proof in the high-disorder case, we decompose the integral as
\[
\int_{-1}^1 \beta_L^q(u)\,\di \mu(u)
=
\int_{-1}^{1-\delta}\beta_L^q(u)\,\di \mu(u)
+
\int_{1-\delta}^1 \beta_L^q(u)\,\di \mu(u)\,.
\]
Hence
\[
\int_{-1}^1 \beta_L^q(u)\,\di \mu(u)
\sim \eta
\int_{1-\delta}^1 \beta_L^q(u)\,\di \mu(u)\,,
\]
provided that the latter diverges, where $\eta=2$ if $\kappa$ is even and otherwise $\eta=1$\footnote{Recall that in the sparse cases, $\kappa$ can not be odd.}.
Arguing as in the high-disorder case, we further decompose
\begin{equation}\label{I+II-sparse}
\int_{1-\delta}^1 \beta_L^q(u)\,\di \mu(u)
=
\int_{1-\delta}^{1-\epsilon_L}\beta_L^q(u)\,\di \mu(u)
+
\int_{1-\epsilon_L}^1\beta_L^q(u)\,\di \mu(u)
:= \mathrm{I} + \mathrm{II}\,,
\end{equation}
where $\epsilon_L:=(\eps L)^{-1/\nu}$ for $\eps$ small enough that will be chosen later.

\smallskip
\noindent
{\itshape Upper bound for $\mathrm{I}$. }  the same computation used in the high-disorder case, we obtain
 	\begin{equation*}
 	\mathrm{I}\, \lesssim \,
 		\begin{cases}
 			\log(\epsilon_L^{-1})\asymp \log L\,\, &\textnormal{if } q=d/(2\nu) \,,\vspace{1mm}\\
 			\epsilon_L^{d/2-q\nu}\asymp L^{q -d/(2\nu)}\,\, &\textnormal{if }q>d/(2\nu)\,.
 		\end{cases}
        \end{equation*}
{\itshape Lower bound for $\mathrm{I}$.} Choosing $\eps \leq 1/ (2C)$ we have
$$1- C(1-x)^{-\nu} L^{-1} \geq 1-C \eps\geq \frac 1 2,\qquad x\in[ r, 1-\varepsilon_L]$$ and hence
\begin{equation*}
\begin{aligned}
\mathrm{I}\,\gtrsim\,&\int_{1-\delta}^{1-\varepsilon_L} (1-x)^{-\nu q} \Bigg[ 1- \frac{C(1-x)^{-\nu}}{L}\Bigg]^q \di \mu(u)  \gtrsim \int_{1-\delta}^{1-\varepsilon_L}  (1-x)^{-\nu q + d/2-1}\di u \,.
        \end{aligned}
        \end{equation*}
This allows us to derive the desired lower bound.

\smallskip
\noindent
{\itshape Estimate for $\mathrm{II}$} By definition of $\beta_L$ (cf..~\eqref{betaL}) and since $\kappa_L(u)< 1$ we have $\beta_L(u) \lesssim L $ and hence
 	\begin{equation*}
 		\mathrm{II}\lesssim L^q \int_{1-\epsilon_L}^1  (1-u)^{d/2-1}\,du\lesssim L^{q-d/(2\nu)} \lesssim \,\mathrm{I}\,.
 	\end{equation*}
 	\medskip
 	\noindent
 	\noindent
 	Combining~\eqref{I+II-sparse} with the previous bounds yields~\eqref{eq: moments of betaL non int case intro sparso}.

    \medskip
    \noindent {\bfseries Proof of~\eqref{eq: moments of |betaL| intro  sparso}.} Same as in the high-disorder cases.
 \end{proof}
\begin{remark}
 As in high-disorder cases, $\beta\in L^q(\mu)$ if and only if $q<d/(2\nu)$. Using the same argument as before, we have $\beta \in L^q(\mu)$ if $q< d/(2\nu)$. On the other hand, for every $x\in (1-\delta,1)$, from~\eqref{beta-lower}, we have
    $$ \beta(x) = \lim_{L\to + \infty} \beta_L(x)\gtrsim \lim_{L\to +\infty}\Big((1-x)^{-\nu} [ 1 - C(1-x)^{-\nu}L^{-1}] \Big) = (1-x)^{-\nu} \,.$$
Hence,
$$ \int_{-1}^1 |\beta(u)|^q\di \mu(u) \geq \int_{1-\delta}^1 |\beta(u)|^q\di \mu(u) \gtrsim \int_{-1}^1 (1-x)^{d/2-1-\nu} \di x $$
 and the claim follows since the integral on the right-hand side diverges.
\end{remark}
The proof of the following corollary, which establishes~\Cref{cor:boun-hatK-iper} in the sparse case,
partially follows the proof of the analogous result in the high-disorder case.
However, the argument is more technical, since $1-C_0(L)$ converges to $0$
rather than to a constant.

\begin{corollary}
\label{cor:boun-hatK-sparse}Let $\widehat \kappa_L$ be as in \eqref{defkappahatL} and assume \Cref{ass:cri-pot}. Then for $q\ge2$ we have the following.

\noindent
If $q<\frac{d}{2\nu}$, as $L\to\infty$ we have $\int_{-1}^1\widehat \beta ^q \di\mu >0$ always and
	\begin{equation*}
		\int_{-1}^1 \widehat \kappa_L(u)^q\,\di \mu(u)
		=\frac{L^{-q}}{\beta_0^q} \int_{-1}^1  \widehat \beta(u)^q\,\di \mu(u)\,(1+o(1))\, .
	\end{equation*}

    \noindent
If $q\ge \frac{d}{2\nu}$ we have

$$ \int_{-1}^1 \widehat \kappa_L(u)^q \di \mu (u) \asymp \frac{1}{(L\beta_0)^q} \int_{-1}^1 \beta_L(u)^q \di \mu(u)\asymp
\begin{cases}
			L^{-q} \log L &\textnormal{if } q=\frac{d}{2\nu}\,, \\
			L^{-\frac{d}{2\nu} } &\textnormal{if } q>\frac{d}{2\nu}\,.
		\end{cases}
        $$
{Moreover,}
$$ \int_{-1}^1  |\widehat\kappa_L(u)|^q \di \mu (u)\lesssim
\begin{cases}
L^{-q}  &\textnormal{if } q<\frac{d}{2\nu}\,, \\
			L^{-q} \log L &\textnormal{if } q=\frac{d}{2\nu}\,, \\
			L^{-\frac{d}{2\nu} } &\textnormal{if } q>\frac{d}{2\nu}\,.
		\end{cases}
        $$
\end{corollary}
\begin{proof} Using the definition of $C_0(L)$ (cf.~\eqref{C-orto}) and of $\beta_L$ (cf.~\eqref{betaL}), since $\mu$ is a probability density, we obtain
\begin{equation}\label{widekappa-sparse}
\begin{aligned}
\widehat \kappa_L(v) = & \frac{\kappa_L(v) - C_0(L)}{1-C_0(L)}=  f_L \Big[ \beta_L(v) - \int_{-1}^1 \beta_L(u) \di \mu(u) \Big]
\end{aligned}
\end{equation}
with
    $$ f_L = \frac{L^{-(\rho+1)}}{1-C_0(L)} = \frac{L^{-(\rho+1)}}{\beta_0 L^{-\rho} - \beta_1 L^{-(\rho+1)}\log L - L^{-(\rho+1)}\int_{-1}^1 \beta_L(u) \di \mu (u)}$$
and in particular,
\begin{equation}\label{cc-sparse} \int_{-1}^1 \widehat\kappa_L(u)^q \di \mu(u) = f_L^q \sum_{k=0}^q \binom{q} k (-1)^{q-k} \int_{-1}^1 \beta_L(u)^k \di \mu (u) \Bigg( \int_{-1}^1 \beta_L(u) \di \mu(u)\Bigg)^{q-k}\,.\end{equation}

\medskip
\noindent
If $\beta\in L^q(\mu)$, then $\beta\in L^k(\mu)$ also for $k<q$ and so, $f_L\sim (\beta_0 L)^{-1}$. The claim follows using~\Cref{mom-sparso}. {To prove that $\int_{-1}^1\widehat\beta ^q(u)\di\mu(u)$ is always strictly positive, one can mirror exactly the argument in the proof of \Cref{cor:boun-hatK-high} to show that $\int_{-1}^1\widehat\beta ^q(u)\di\mu(u)$ is $0$ only if $\kappa$ is odd, just using \Cref{fixed-point-sparse} and \eqref{widekappa-sparse} instead of \Cref{fixed-point:high} and \eqref{widekappa}. We conclude observing that $\kappa$ can never be odd in the sparse case (otherwise $0$ would be a fixed point). }

\medskip
\noindent
If $\beta \not \in L^q(\mu)$, i.e., $q\geq d/(2\nu)$, one can use  the same argument as in the proof of~\Cref{cor:boun-hatK-high} to prove that, for $k=1, \dots, q-1$, we have
$$
\Bigg| \int_{-1}^1 \beta_L(u)^k \di \mu(u) \Big( \int_{-1}^1 \beta_L(u)\di \mu(u)\Big)^{q-k}\Bigg| = o\Bigg( \int_{-1}^1 |\beta_L(u)|^q \di \mu(u)\Bigg)
$$
and, for $k=0$,
	\[
	\left|\int_{-1}^1 \beta_L(u)\di \mu(u)\right|^{q}=o\left(\int_{-1}^1 \beta_L^q(u)\di \mu(u)\right).
	\]
Then, the dominating term in~\eqref{cc} comes from $k=q$ and hence the claim. The proof of the upper bound when we consider $\int |\widehat \kappa_L|$ follows by standard inequalities and ~\Cref{mom-sparso}.
\end{proof}

\begin{remark}To conclude the proof, we need to prove~\eqref{puntuale} when $\beta\in L^2(\mu)$. Recall that
$$ g_L(\ell) = \frac{C_\ell(\ell)}{1-C_0(\ell)}\mathbf{1}_{\ell\ge 1} $$
Now, using the same computation as in~\Cref{cor:boun-hatK-sparse}, we have
\begin{equation}\label{pun0}1-C_0(L) \sim \beta_0 L^{-\rho}\,.\end{equation}
Moreover, since $\int G_{\ell,d}(u) \di \mu(u) =0 $ for every $\ell>0$, we have
\begin{align*}C_{L}(\ell) =&\int_{-1}^1 \kappa_L(u)  G_{\ell,d} \mu(u) = \int_{-1}^1 \Big( \kappa_L(u)-1 +\beta_0 L^{-\rho}- \beta_1 L^{-(\rho+1)}\log L)  G_{\ell,d}(u) \mu(u)\\
= & L^{-(\rho+1)} \int_{-1}^1 \beta_L(u) G_{\ell,d}(u) \di \mu(u) \,.
\end{align*}
Hence, by dominated convergence\footnote{We assume $\beta\in L^2(\mu)$, hence $(1-x)^{-\nu}\in L^2$. Then the domination comes from~\eqref{beta-upper}}, we have
\begin{equation}\label{pun1} L^{\rho+1} C_L(\ell) \sim \mathbf{1}_{\ell\ge 1}\int_{-1}^1 \beta(u) G_{\ell,d}(u)\di \mu(u)=\int_{-1}^1 \widehat \beta(u) G_{\ell,d}(u)\di \mu(u)\,,
\end{equation}
where the last equality follows by definition of $\widehat\beta$ (since $\beta\in L^2(\mu)\subseteq L^1(\mu)$). So,~\eqref{puntuale} follows combining~\eqref{pun0},~\eqref{pun1} and the definition of $g_L(\ell)$

\end{remark}

\section{Auxiliary results}\label{sec: auxiliary}
In this section, we prove the technical results exploited in the previous sections.

\subsection{Proof of Properties~\ref{p0}--\ref{p3}}\label{sec::prop}
\begin{enumerate}[label=(P\arabic*), start=0]
\item We begin with the first property. We can write the Hermite expansion
\[
\sigma(Z)=\sum_{q=0}^{\infty} J_q(\sigma)H_q(Z),
\]
and using the definition of $\kappa$ (cf.~\eqref{eq:KL1}) together with the diagram formula, we obtain
\[
\kappa(u)
=C \sum_{q=0}^{\infty} J_q(\sigma)^2
\mathbb{E}\!\left[
H_q(Z_1)H_q\!\big(uZ_1+\sqrt{1-u^2}\,Z_2\big)
\right]
= \sum_{q=0}^{\infty} a_q(\sigma)^2 q! \, u^q
\]
where $a_q = J_q(\sigma) \sqrt{C}$. Since $\sigma\in\mathbb D^{1,2}$ we have
\[
\sum_{q=1}^{\infty} q\,a_q(\sigma)^2 q!<\infty,
\]
which implies that the derivative of the above power series has radius of convergence at least~$1$.
    \item The second property follows trivially from the power series representation by applying the triangle inequality and recalling that $\kappa(1)=1$.

\item
We prove that $\kappa\in C^1([-1,1])$.
Since $\kappa$ admits a power series expansion around $0$ with radius of convergence at least $1$, it is infinitely differentiable in $(-1,1)$ and
\[
\kappa'(u)=\sum_{q=1}^{\infty} q\,a_q(\sigma)^2 q!\,u^{q-1}\,.
\]
Since the coefficients are non-negative, Beppo Levi's theorem yields
\[
\lim_{u\to1^-}\kappa'(u)
=
\sum_{q=1}^{\infty} q\,a_q(\sigma)^2 q!\,.
\]
Since $\sigma\in\mathbb D^{1,2}$ we have
\[
\sum_{q=1}^\infty q\,a_q(\sigma)^2 q!<\infty,
\]
which ensures that the derivative extends continuously to $u=1$.
   \item
Arguing as above, for $u\in(-1,1)$ we obtain
\[
\kappa^{(s)}(u)
=
\sum_{q=s}^{\infty}
\frac{(q!)^2}{(q-s)!}\,a_q(\sigma)^2\,u^{q-s}\,.
\]
Hence all terms of the series are non-negative for $u\in(0,1)$.
\end{enumerate}
For the converse implication, assume that $\kappa$ admits the representation~\eqref{power-series} and define
\[
\sigma(x)=\sum_{q=0}^{\infty} \frac{a_q(\sigma)}{\sqrt{C}} H_q(x).
\]
Then the kernel of a shallow infinite-width random neural network with activation function $\sigma$ coincides with $\kappa$.
Arguing as in the proof of~\ref{p2}, condition $\kappa\in C^1([-1,1])$ implies $\sigma\in\mathbb D^{1,2}$. Indeed, let  $a_q(\sigma)$ be the coefficients of $\sigma$ with respect to the Hermite basis: the diagram formula yields $a_q = J_q(\sigma)\sqrt{C}$;
differentiating the power series term by term, we observe that $\kappa'(1)<\infty$ if and only if $\sigma \in \mathbb D^{1,2}$.

\smallskip
\noindent

\subsection{Proof of~\texorpdfstring{\Cref{fixed-point:high}}{the Theorem}: fixed-point analysis in the high-disorder regime}\label{fixed-point:high-tecn}

 The existence of a unique fixed point $b$, the fact that $b$ is necessarily $\ge0$ and the pointwise converge of $\kappa_L(x)\to b$ was proved in~\cite[Section 5]{dilillo2026largedeviationprinciplesfunctional}. Moreover, if $\kappa$ is neither even nor odd, then the convergence also holds at $-1$. The fact that $\kappa'(b)\in[0,1)$ follows by Rolle's theorem applied to $x\mapsto \kappa(x)-x$, $x=b$ and $x=1$ and the fact that $\kappa'$ is strictly increasing.
 Moreover, if $\kappa'(b)=0$, since $\kappa$ is not constant, \eqref{power-series} yields $b=0$.

 \smallskip
 \noindent

\noindent
{\bfseries Convergence of $\boldsymbol{\beta_L}$.}
 Since $\kappa\in C^\infty((-1,1))$, for every $M>\kappa''(b)/2>0$ and $\varepsilon>0$ there exists $\delta>0$ such that, for every $u\in (b-\delta,b+\delta)=:U_\delta$, we have
\begin{align}
\label{e1}
\kappa(u) &= b+\lambda(u-b)+r(u), \qquad r(u)\le M(u-b)^2,\\
\label{e2}
\kappa'(u) &\le \lambda+\varepsilon .
\end{align}

Define
$$ K_\eta = \begin{cases} [-1,1-\eta] & \text{ if } \kappa(-1)  \neq \pm 1  \\
[-1+\eta,1-\eta]  & \text{ otherwise}
\end{cases}
$$
From the compactness of $K_\eta$, it is trivial to see that for every $\delta>0$, there exists $L_0 = L_0(\delta, \eta, \kappa)$ such that,
$$ \kappa_L(x) \in (b-\delta, b+\delta), \qquad \forall x\in K_\eta, \qquad  L\geq L_0\, \footnote{For every $\ell$, set $$E_\ell = \{ x\in K_\eta\; |\; |\kappa_\ell(x) -b |< \eps\}
$$
We note that $E_{\ell}\subseteq E_{\ell+1}$, indeed if $x\in E_\ell$, using~\eqref{e1} we have
$$ \kappa_{\ell+1}(x) - b \leq \eps ( \lambda + M \eps) $$
and hence for $\eps$ small enough we have $x\in E_{\ell+1}$.

Since $\kappa$ is continuous and $\kappa_L(x)\to b$ for every $x\in K_\eta$. $E_\ell$ is an open cover of $E_\eta$ and hence
$$ K_\eta = E_{\ell_1} \cup \dots E_{\ell_k}\,.$$
The claim follows taking $L_0 = \max_{s=1, \dots k} \ell_k$\,.
}$$

From now on, $x\in K_\eta$ and  $\ell\ge L_0$. From~\eqref{e1} we obtain
\begin{equation}\label{e3}
\big|\beta_{\ell+1}(x)-\beta_\ell(x)\big|
\le
M\frac{|\kappa_\ell(x)-b|^2}{\lambda^{\ell+1}} .
\end{equation}

Now let $s>L_0$. By the mean value theorem, there exists $\xi$ in the open interval with endpoints $\kappa_{s-1}(x)$ and $b$ such that
\[
\kappa_s(x)-b
=
\kappa'(\xi)\big(\kappa_{s-1}(x)-b\big)
\le
(\lambda+\varepsilon)\big(\kappa_{s-1}(x)-b\big),
\]
where the last inequality follows from~\eqref{e2}. Iterating the previous bound yields
\[
\kappa_s(x)-b
\le
(\lambda+\varepsilon)^{\,s-L_0}\big(\kappa_{L_0}(x)-b\big)
\le
\delta(\lambda+\varepsilon)^{\,s-L_0},
\]
where the last inequality follows from the definition of $L_0$.

Substituting this estimate into~\eqref{e3}, we obtain
\[
\big|\beta_{\ell+1}(x)-\beta_\ell(x)\big|
\le
\frac{M\delta^2}{\lambda(\lambda+\varepsilon)^{2L_0}}\, q^\ell,
\qquad \forall \ell\ge L_0,
\]
where we set
\[
q:=\frac{(\lambda+\varepsilon)^2}{\lambda}\,.
\]

For $\varepsilon>0$ small enough we have $q<1$, and therefore $(\beta_L(x))_{L\ge0}$ is a Cauchy sequence (uniform in $x\in K_\eta$). This proves the desired convergence.

    \medskip
    \noindent
{\bfseries Positive and monotone convergence of $\boldsymbol{\beta_L}$ on $\boldsymbol{[b,1)}$.}
Let $x\in (b,1)$. By the mean value theorem,
\[
\beta_{L+1}(x)=\beta_L(x)\frac{\kappa'(\xi)}{\lambda},
\]
for some $\xi$ in the interval with endpoints $\kappa_L(x)$ and $b$. Since both $\kappa$ and $\kappa'$ are increasing on $[0,1]$ (see~\ref{p3}), we have
\[
\kappa'(\xi)\ge \kappa'(b)=\lambda .
\]
Since $\beta_0(x)\ge 0$, it follows that $\beta_L(x)\ge 0$ for every $L$.

\medskip
\noindent
{\bfseries Proof of~\eqref{bound-sup-high}.}
By a Taylor expansion at $1$ with second-order Lagrange remainder, there exists
$\xi_L\in(1-c\lambda^{L/\nu},1)\subset(0,1)$ such that
\[
\kappa_L(1-c\lambda^{L/\nu})
=
\kappa_L(1)
-
\kappa_L'(1)c\lambda^{L/\nu}
+
\kappa_L''(\xi_L)c^2\lambda^{2L/\nu}
\ge
\kappa_L(1)-\kappa_L'(1)c\lambda^{L/\nu},
\]
where the last inequality follows since $\kappa''(\xi_L)\ge0$ (see~\ref{p3}).
Since $\kappa_L(1)=1$ and $\kappa_L'(1)=\kappa'(1)^L=\lambda^{-L/\nu}$, we obtain
\begin{equation}\label{piccolo}
\kappa_L(1-c\lambda^{L/\nu})\ge 1-c .
\end{equation}
Now let $r\ge b$ and $x\in\big(\lambda^{1/\nu}(1-r),\,1-r\big)$. Then
\[
1-x\lambda^{L/\nu}\ge b .
\]
Since $\beta_L(x)$ is increasing in $L$ for $x\ge b$, we obtain
\[
\beta_L(1-x\lambda^{n/\nu})
\ge
\beta_n(1-x\lambda^{n/\nu})
=
\frac{\kappa_n(1-x\lambda^{n/\nu})-b}{\lambda^n}
\ge
\lambda^{-n}(1-x-b)
\ge
\lambda^{-n}(r-b).
\]

The desired bound follows by taking $r=(b+1)/2$.

   \medskip
    \noindent
{\bfseries Technical inequality}
To prove~\eqref{dominazione-high} we need to prove the following inequality
 \begin{equation}\label{pp} \kappa_L(x)\leq \kappa_L(1-c\lambda^{L/\nu}) \leq 1 - c(1-\varepsilon), \qquad \forall x\in [0, 1-c\lambda^{L/\nu}], \qquad \forall c\in (0, \delta_\varepsilon)\,.
\end{equation}
 By a Taylor expansion at  $1$  with first-order Lagrange remainder, there exists $\xi_L\in (1-c\lambda^{L/\nu},1)$ such that
\begin{equation}\label{bb00} \kappa_L(1-c\lambda^{L/\nu}) = \kappa_L(1) - \kappa_L'(\xi_L) c\lambda^{L/\nu} \leq 1 - c\lambda^{L/\nu} \kappa_L'(1-c\lambda^{L/\nu})
\end{equation}
 where the last inequality follows since $\kappa'$ increases in $[0,1]$.

\noindent
 	So, to find a good upper bound for $\kappa_L$, we need a good lower bound on $\kappa_L'$. Using~\eqref{piccolo}, with $L$ replaced by $L-i$ and $c$ replaced by $c\lambda^{i/\nu}$
    we have,
   $$ \kappa_{L-i}(1-c\lambda^{L/\nu})\ge 1-c\lambda^{i/\nu}\,.$$
From~\Cref{ass::CRI} we obtain $\kappa'(1)-\kappa'(x)\le c_\alpha (1-x)^{\alpha}$ for $\alpha=\gamma-1>0$ and some $c_\alpha>0$. Thus, using the definition of $\nu$, we have $\kappa'(1) = \lambda^{-1/\nu}$ and
 	\begin{align*}
 	\kappa_L'(1-c\lambda^{L/\nu})&=\prod_{i= 1}^L\kappa'(\kappa_{L-i}(1-c\lambda^{L/\nu}))\geq \prod_{i=1}^L\kappa'(1-c\lambda^{i/\nu})\\
    &\ge \prod_{i=1}^L(\lambda^{-1/\nu} -c_\alpha\,c^\alpha\lambda^{\alpha i/\nu})=\lambda^{-L/\nu}\prod_{i=1}^L\big(1-c_\alpha\,c^\alpha\lambda^{(\alpha i+1)/\nu}\big)\,.
 	\end{align*}
Taking   $c\in(0,c_\alpha^{-1/\alpha})$, we can use the classical Weierstrass product inequality (see e.g. \cite{DH10}). Thus
\begin{align*}
 	\lambda^{L/\nu}\kappa_L'(1-c\lambda^{L/\nu})&\geq 1 - c_\alpha\,c^\alpha  \lambda^{1/\nu}\sum_{i=1}^L \Big( \lambda^{\alpha/\nu}\Big)^i \geq 1-  C c^\alpha \,
 	\end{align*}
    where $C$ depends only on $\alpha, \lambda, \nu$. So, for  $\epsilon>0$ arbitrary small, there exists $\, \delta_\epsilon>0$  depending  only on $C$ such that for all $c\in(0,\delta_\epsilon)$
 	we have
    $$ \lambda^{L/\nu}\kappa_L'(1-c\lambda^{L/\nu}) \geq 1-\epsilon\,.$$
Combining the previous bound with~\eqref{bb00}, it follows that
$$ \kappa_L(1-c\lambda^{L/\nu}) \leq 1 - c(1-\varepsilon) $$
and since $\kappa_L$ is increasing, we obtain~\eqref{pp}

\medskip
    \noindent
    {\bfseries Proof of~\eqref{dominazione-high}} We note that proving~\eqref{dominazione-high} is equivalent to showing that
\[
|\beta_L(x)| \lesssim (1-x^2)^{-\nu}, \qquad x\in [b,1].
\]
Indeed, by uniform convergence, if $\kappa$ is neither even nor odd, then $\beta_L$ is uniformly bounded on $[-1,b]$. On the other hand, if $\kappa$ has a parity, then in order to control $|\beta_L(x)|$ on $[-1,1]$ it is enough to control it on $[0,1]$, and we again use uniform convergence on $[0,b]$.

Since $0\leq \beta_L(x)\uparrow\beta(x)$, to conclude the proof we need to prove that $$\beta(x) \lesssim (1-x)^{-\nu}, \qquad x\in [b,1)\,.$$
Using the definition of $\beta$, we have
 	\begin{equation}\label{semigruppo} \beta(x) = \lambda^{-n} \beta(\kappa_n(x))\,.
 	\end{equation}
 Let us fix $c\in(0,\delta_\epsilon)$. Then, for any $x\in(1-c,1)$, there exists a unique $L$ such that
 such that $x\in[1-c\lambda^{(L-1)/\nu},1-c\lambda^{L/\nu}]$. Thus  $(1-x)^{\nu} \leq c^{\nu}\lambda^{L-1}$ and in particular, using~\eqref{semigruppo}, we obtain
 $$ \beta(x) (1-x)^{\nu} \leq \frac{c^\nu}{\lambda}  \beta(\kappa_L(x)) \leq  \frac{c^\nu}\lambda \beta(1-c(1-\varepsilon)) $$
 where the last inequality follows from~\eqref{pp} and using that
 $\beta$ increases on $(0,1)$ (is a limit of increasing functions).
 \qed
\subsection{Proof of~\texorpdfstring{\Cref{lem: classic for contractions}}{the Lemma}: bound for contractions}
\label{sec: classic for contractions-tecn}
By the inequality $|a^rb^{q-r}|\le a^q+b^q$, we get the bound
	\[
\int_{\S^d}\left(\int_{(\S^d)^3}|f(\langle x,y \rangle)|^{r_1+r_2}|f(\langle z,w \rangle|^{r_1}|f(\langle y,w\rangle|^{r_2}\di y\di z\di w\right)\,\di x\,.
	\]
	Moreover, for fixed $x\in \S^d$, we can take $A_x\in {\rm SO}(d)$ such that $A_x(e_1)=x$. With this choice and the change of variable $y'=A_x(y)$, $z'=A_x(z)$, $w'=A_x(w)$, we get
	\begin{align*}
		&\int_{(\S^d)^3}|f(\langle x,y \rangle)|^{r_1+r_2}|f(\langle z,w \rangle|^{r_1}|f(\langle y,w\rangle|^{r_2}\di y\di z\di w
		\\
		&=\int_{(\S^d)^3}|f(\langle e_1,y' \rangle)|^{r_1+r_2}|f(\langle z',w' \rangle|^{r_1}|f(\langle y',w'\rangle|^{r_2}\di y'\di z'\di w'
		\\
		&=\int_{\S^d}|f(\langle e_1,y' \rangle)|^{r_1+r_2}\left(\int_{\S^d}\left(\int_{S^d}|f(\langle z',w' \rangle)|^{r_1}\,\di z'\right)|f(\langle y',w'\rangle|^{r_2}\di w'\right)\di y'
		\\
		&=\int_{\S^d}|f(\langle e_1,y' \rangle)|^{r_1+r_2}\left(\int_{\S^d}\left(\int_{\S^d}|f(\langle z',w' \rangle)|^{r_1}\,\di z'\right)|f(\langle y',w'\rangle|^{r_2}\di w'\right)\di y'
		\\
		&=\int_{\S^d}|f(\langle e_1,y' \rangle)|^{r_1+r_2}\left(\int_{\S^d}\left(\int_{\S^d}|f(\langle z'',e_1 \rangle)|^{r_1}\,\di z''\right)|f(\langle e_1,w''\rangle|^{r_2}\di w''\right)\di y'.
	\end{align*}
	where the last equality follows again by the invariance of the integrands under the action of $SO(d)$, with the changes of variables $z''=A_{w'}(x)$, $w''=A_{y'}(w')$. Thus, we can conclude the proof combining the previous two facts and a standard change of variable.
\qed

\subsection{Proof of~\texorpdfstring{\Cref{lem_tailcontrol}}{the Lemma}: tail control in non-critical CLT}\label{sec:tailcontrol}
For the proof in the sparse regime, we need a bound analogous to~\eqref{pp}, stated in the following lemma
\begin{lemma}
{Let \Cref{ass:cri-pot} hold. Then for some $\delta\in(0,1)$ and $L$ large enough
\begin{equation}
\label{bho-piccolo-sparso}\kappa_L(1-xL^{-\rho}) \leq 1 -\frac{x}{2} L^{-\rho}, \qquad x\in (0, \delta) \,.
\end{equation}}
\end{lemma}
\begin{proof}
Let $y\in (0,1)$, by Lagrange theorem, there exists $\xi \in (y,1)$ such that
$$ \kappa_L(y) = \kappa_L(1) - (1-y) \kappa_L'(\xi)$$
and hence
$$\kappa_L(y) \leq 1  - (1-y) \kappa_L'(y)\,.$$
In particular, for every $x<1$ we have
$$\kappa_L(1-xL^{-\rho}) \leq 1  - xL^{-\rho} \kappa_L'(1-xL^{-\rho})\,.$$
Now, since $\kappa(u) > u$ and $\kappa'$ is increasing in $[0,1]$, we have
$$ \kappa_L'(1-xL^{-\rho}) = \prod_{s=1}^L \kappa'(\kappa_s(1-xL^{-\rho})) \geq \Big( \kappa'(1-xL^{-\rho})\Big)^{L}\,.$$
Using~\Cref{ass:cri-pot}, there exists constants $\beth>0, \delta\in(0,1)$ such that for $x\in(0,\delta)$
$$ \kappa'(1-xL^{-\rho}) \geq 1 - \beth(xL^{-\rho})^{1/\rho} =  1 - \frac{ \beth x^{1/\rho} }L \,.$$
Thus, for $x\in(0,\delta)$ and $L$ large enough, we can conclude observing
 $$ \Big(\kappa'(1-xL^{-\rho})\Big)^L\ge \Big(1 - \frac{ \beth x^{1/\rho}}{L}\Big)^L \ge \frac12\,.$$
\end{proof}

\medskip
We prove that~\Cref{lem_tailcontrol} holds both in the sparse and in the high-disorder regimes.
\begin{proof}[Proof of~\Cref{lem_tailcontrol}]
    We observe that by arguing as in \cite[Proof of Lemma 3.1]{MRZ25}, and using~\Cref{cor:boun-hatK-iper}, we obtain
    \[
     \E\left[\left(\widetilde F_L-\widetilde F_{L,N}\right)^2\right]\lesssim \,\frac{\,\|\varphi\|^2}{v_L^{d/(2\nu)}}\,\int_{(\S^d)^2}|\widehat \kappa_L(\langle x,y \rangle)|^N\,\di x \di y\lesssim\,\frac{\int |\widehat \kappa_L(u)|^N\,\di \mu(u)}{v_L^{d/(2\nu)}}\,.
    \]
    It remains to control uniformly $\int |\widehat \kappa_L(u)|^N\,\di \mu(u)$ as $N\rightarrow\infty$.

    \smallskip
    \noindent {\bfseries High-disorder.} Using~\eqref{pp} (with $c$ replaced by $x$) and the bound $1-C_0(L)\le1$,
there exists $\delta_{1/2}>0$ such that for all $x\in(0,\delta_{1/2})$,
\[
\widehat\kappa_L(1-x\lambda^{L/\nu})
\le
\frac{1-\frac{x}{2}-C_0(L)}{1-C_0(L)}
=
1-\frac{x}{2(1-C_0(L))}
\le
1-\frac{x}{2}\,.
\]

     \smallskip
    \noindent {\bfseries Sparse.} Using~\eqref{bho-piccolo-sparso} and since $L^\rho(1-C_0(L)) \to \beta_0 $, for $L$ large enough and $x<\delta_{1/2}$
    $$  \widehat\kappa_L(1-xL^{-\rho}) \le
\frac{1-\frac x2 L^{-\rho}-C_0(L)}{1-C_0(L)}
= 1 - \frac{x/2}{\beta_0/2}\,.$$

\medskip
\noindent
Summarizing in both cases, recalling that $\nu^{-1} = {\rho}$, there exists $\delta_{1/2}$ and a constant $c>0$ such that, for $x\in (0, \delta_{1/2})$
$$  \widehat\kappa_L(1-xv_L^{\rho}) \leq 1 - \frac{x}{c}\,.$$
    Therefore, choosing $\eps<1 \wedge\delta_{1/2}$ and arbitrary small, by a change of variable $u=1-xv_L^{\rho}$ and recalling that $\nu = \rho^{-1}$, we get
	\begin{align*}
	\int_{1-\eps v_L^{\rho}}^1 |\widehat \kappa_L(u)|^N\,\di \mu(u) &\lesssim v_L^{d/(2\nu)} \int_0^\varepsilon |\widehat \kappa_L(1-xv_L^{\rho})|^N x^{d/2-1}  \di x  \lesssim v_L^{d/(2\nu)}\int_{0}^\eps (1-x/c)^N\,x^{d/2-1}\,\di x\\
    &\lesssim v_L^{d/(2\nu)}\int_{0}^1 (1-z)^N\,z^{d/2-1}\,\di z
	\end{align*}
   where the last inequality follows by setting $z=x/c$. Using the classical asymptotics of the Beta function,
\[
\int_0^1 (1-z)^N z^{d/2-1}dz \asymp N^{-d/2}\,.
\]
Hence, for $L,N$ sufficiently large, we get
    \[
    \frac{\int_{1-2v_L}^1 |\widehat \kappa_L(u)|^N\,\di \mu(u)}{v_L^{d/(2\nu)}}\lesssim N^{-d/2}\,.
    \]
    On the other hand, since $\kappa_L(x)$ is increasing for $x\in[0,1]$, we have
    $$\widehat  \kappa_L(u)\le \widehat  \kappa_L(1-\eps v_L^{\rho}), \qquad u \in (0, 1-\varepsilon v_L^{\rho})\, .$$
    Thus,
    \[
     \int_{0}^{1-\eps v_L^{\rho}}\widehat \kappa_L(u)^{N}\,\di \mu(u)\le (1-\eps/c)^{N-Q}\int_{0}^{1-\eps v_L^{\rho}}\widehat \kappa_L(u)^{Q}\,\di \mu(u)\lesssim (1-\eps/2)^{N-Q}v_L^{d/(2\nu)}
    \]
    where the last inequality follows from~\Cref{cor:boun-hatK-iper}.
    \end{proof}

\subsection{Proof of~\texorpdfstring{\Cref{lem:analytic conservation}}{the Lemma}: justification of the exchange of sum and integral}
\label{lem:analytic conservation-tecn}
First of all, note that since $\langle \theta_1,\theta_2\rangle ^k$ is a covariance kernel on $\S^d$, by Schoenberg's theorem~\cite{schoenberg} $u\mapsto u^k$ can be expressed as follows
\begin{equation}
    \label{eq:monomial expansion}
    u^k=\sum_{\ell=0}^k\,a_{k,\ell}\,n_{\ell}\,G_{\ell}(u)\,\quad\quad \quad a_{k,\ell}=\int_{-1}^1u^{k}G_{\ell}(u)\di \mu(u)\ge0\,.
\end{equation}
Moreover, $f$ is real analytic with radius at least $1$ and $f^{(k)}(0)\ge0$ for $k\ge k_0$, so  by Fubini-Tonelli\footnote{Recall that $
|G_\ell|\le 1$, $f^{(k)}(0)\ge0$ for $k\ge\bar k$ and $f(u)\in L^p(\mu)\subseteq L^1(\mu)$ implies $f(|u|)\in L^1(\mu)$, so we have $$\sum_{k_i=0}^\infty \frac{|f^{(k_i)}(0)|}{k_i!}\int_{-1}^1|u|^{k_i}|G_{\ell_i}(u)|\di \mu(u)\lesssim \int_{-1}^1f(|u|)\di \mu(u)<\infty.$$}
\[
\int_{-1}^1f(u)G_{\ell_i}(u)\di \mu(u)=\sum_{k_i=0}^\infty \frac{f^{(k_i)}(0)}{k_i!}\int_{-1}^1u^{k_i}G_{\ell_i}(u)\di \mu(u)=\sum_{k_i=0}^{\infty} \frac{f^{(k_i)}(0)}{k_i!}a_{k_i,\ell_i}\mathbbm{1}_{\ell_i\le k_i}\,.
\]
Thus, exchanging the sums in $\ell_1,\dots,\ell_p$ and $k_1,\dots,k_p$ (i.e. applying again Fubini-Tonelli), we get
\begin{align*}
    &\sum_{\ell_1, \dots, \ell_p \in \N} \prod_{i=1}^p   n_{\ell_i,d }\,\left(\int_{-1}^1f(u)G_{\ell_i}(u)\di \mu(u)\right)\widehat{\mathcal G_p}(\ell_1, \dots,\ell_p)
    \\
    &=\sum_{k_1, \dots, k_p \in \N}\,\prod_{i=1}^p\frac{f^{(k_i)}(0)}{k_i!}\sum_{\ell_1=0}^{k_1}\dots\sum_{\ell_p=0}^{k_p} \prod_{i=1}^p   n_{\ell_i,d }\,a_{k_i,\ell_i}\widehat{\mathcal G_p}(\ell_1, \dots,\ell_p)
    \\
    &=\sum_{k_1, \dots, k_p \in \N}\,\prod_{i=1}^p\frac{f^{(k_i)}(0)}{k_i!}\int_{-1}^1\left(\sum_{\ell_1=0}^{k_1}\dots\sum_{\ell_p=0}^{k_p} \prod_{i=1}^p   n_{\ell_i,d }\,a_{k_i,\ell_i}G_{\ell_i}(u)\,\right)\di \mu(u)
    \\
    &=\sum_{k_1, \dots, k_p \in \N}\,\prod_{i=1}^p\frac{f^{(k_i)}(0)}{k_i!}\int_{-1}^1u^{k_1}\dots u^{k_p}\di \mu(u)\,\,,
\end{align*}
where  the second-last equality follows by definition \eqref{eq:Ghat} and the last equality by \eqref{eq:monomial expansion}.
Thus, the proof is concluded by observing that again by Fubini-Tonelli we have
\[
\sum_{k_1, \dots, k_p \in \N}\,\prod_{i=1}^p\frac{f^{(k_i)}(0)}{k_i!}\int_{-1}^1u^{k_1}\dots u^{k_p}\di \mu(u)=\int\left(\sum_{k=0}^\infty \frac{f^{(k)}(0)}{k!}u^k\right)^p\di \mu(u)=\int_{-1}^1f^p(u)\di \mu(u)\,.
\]
\qed

\subsection{Proof of~\texorpdfstring{\Cref{fixed-point-sparse}}{the Theorem}: fixed-point analysis in the sparse regime.}\label{fixed-point-sparse-tecn}
To prove~\Cref{fixed-point-sparse} we need the following auxiliary lemmas. We will sometimes use the notation
\[
\eps_\ell(x):=(1-\kappa_\ell(x))\,\,.
\]
In~\cite{nostro,dilillo2026largedeviationprinciplesfunctional}, the authors proved that
\begin{equation}\label{base-convergenza}\lim_{\ell\to + \infty} \varepsilon_\ell(x) = 0, \qquad x\in [-1,1]\,
\end{equation}
In~\Cref{fixed-point-sparse} we need a quantitative version of this convergence.

\begin{lemma}\label{lem52}
Let $\kappa$ be  not even and satisfy~\Cref{ass:cri-pot}. Setting $ A_0 = c_1(\gamma_1-1)$ and
$A_1 = \frac{\gamma_1\,c_1}{2} +\frac{c_2}{c_1} \mathbf 1_{\gamma_2 = 2\gamma_1-1}$\footnote{In this scenario, no assumption is made on the sign of $A_1$. For instance, if $\kappa$ is the kernel associated with the Gaussian activation function in the sparse case (see~\cite{nostro2}), then $A_1 = 0$.}. Then
the renormalized iterates
\begin{equation}\label{SL-def} S_L(x) := (1-\kappa_L(x))^{1-\gamma_1} - A_0 L - A_1 \log L,
\qquad x\in[-1,1),
\end{equation}
converge pointwise as $L \to \infty$ to a limit function $S:[-1,1)\to \mathbb R$. Moreover, there exists a positive constant $D$  independent from $L$ and $x$ such that
\begin{align}\label{upp=SL}S_L(x)& \lesssim (1-x)^{1-\gamma_1}
\\
\label{low=SL}S_L(x) &\gtrsim (1-x)^{1-\gamma_1} -D
\end{align}

\end{lemma}
\begin{remark}
    The analysis when $\kappa$ is even is analogous. Indeed  the proof of~\Cref{lem52} works also for $\kappa$ even, if one restrict $x\in [0,1)$. Since also $S_L$ is even we simply get
    \begin{align}\label{upp=SL-even}S_L(x)& \lesssim (1-x^2)^{1-\gamma_1}
\\
\label{low=SL-even}S_L(x) &\gtrsim (1-x^2)^{1-\gamma_1} -D
\end{align}
\end{remark}

\begin{proof}[Proof of the case $\gamma_2>2\gamma_1-1$]
First of all, by~\Cref{ass:cri-pot}, there exist $D,\delta_1>0$ such that
\begin{align}\nonumber 1 - \kappa(x)  = (1-x) - c_1(1-x)^{\gamma_1}- r(1-x)\,,\\
\label{dbound}|r(1-x)|\leq D (1-x)^{\gamma_2} \quad \forall\,\,x\in (1-\delta_1, 1]\,.
\end{align}
By Taylor expansions of $t\mapsto t^{1-\gamma_1}$, for some $\xi_\ell(x)$ in the interval with endpoints $\varepsilon_{\ell+1}(x) $ and $\varepsilon_\ell(x)$, we have
\begin{align}
\nonumber\varepsilon_{\ell+1}(x)^{1-\gamma_1} =&\varepsilon_{\ell}(x)^{1-\gamma_1}  + (\gamma_1-1)\varepsilon_{\ell}(x)^{-\gamma_1}\Big( c_1 \varepsilon_\ell(x)^{\gamma_1} + r(\varepsilon_\ell(x))\Big) \\
\nonumber&+ \frac{\gamma_1(\gamma_1-1)} 2 \varepsilon_\ell(x)^{-\gamma_1-1} \Big(  c_1 \varepsilon_\ell(x)^{\gamma_1} + r(\varepsilon_\ell(x))\Big)^2\\
\nonumber&  + \frac{(\gamma_1+1)\gamma_1(\gamma_1-1)}{3!} \xi_{\ell}(x)^{-\gamma_1-2} \Big(c_1 \varepsilon_\ell(x)^{\gamma_1} + r(\varepsilon_\ell(x))\Big)^3\\
\nonumber= & \varepsilon_{\ell}(x)^{1-\gamma_1}  + (\gamma_1-1)c_1 + \frac{\gamma_1(\gamma_1-1)c_1^2} 2 \varepsilon_\ell(x)^{\gamma_1-1}  + R_\ell(x)\\
\label{esatta}= & \varepsilon_{\ell}(x)^{1-\gamma_1}  + A_0 + A_0 A_1\varepsilon_\ell(x)^{\gamma_1-1}  + R_\ell(x)
\end{align}
where  $R_\ell(x)$ contains all the other terms in the expansions, i.e.
\begin{align*} R_\ell(x) =& (\gamma_1-1) \varepsilon_\ell(x)^{-\gamma_1} r(\varepsilon_\ell(x))
+ \frac{\gamma_1(\gamma_1-1)} 2 \varepsilon_\ell(x)^{-\gamma_1-1} \Big(  2c_1 \varepsilon_\ell(x)^{\gamma_1}r(\varepsilon_\ell(x)) + r(\varepsilon_\ell(x))^2\Big)\\
\nonumber&  + \frac{(\gamma_1+1)\gamma_1(\gamma_1-1)}{3!
} \xi_{\ell}(x)^{-\gamma_1-2} \Big(c_1 \varepsilon_\ell(x)^{\gamma_1} + r(\varepsilon_\ell(x))\Big)^3\,.
\end{align*}
Let $\delta<\delta_1$ (that will be chosen later) and let $L_\delta>0$ such that $\kappa_\ell(x)
\in (1-\delta, 1]$ for all $\ell\geq L_\delta$ and for all $x\in[-1,1]$\footnote{To prove that such a $L_\delta$ exists, one can reason as at the beginning of \Cref{fixed-point:high-tecn}.}. \textbf{From now on}, let $\ell\geq L_\delta$. For $\delta$ small enough,
$\varepsilon_\ell/2\leq \xi_\ell \leq {\varepsilon_\ell}  $
\footnote{Since $\kappa'(x)\geq 0$,  $\kappa(x)>x$. Thus $\varepsilon_\ell(x)\geq \varepsilon_{\ell+1}(x)$. Moreover, by \eqref{dbound} for $\delta$ small enough we have $|\varepsilon_\ell-\varepsilon_{\ell+1}|\le C \varepsilon_\ell(x)^{\gamma_1-1}\varepsilon_\ell(x)\le C\delta^{\gamma_1-1} \varepsilon_\ell(x)<\varepsilon_\ell(x)/2$}
and so
\begin{equation}\label{R-bound}|R_\ell(x)| \leq C\varepsilon_\ell(x)^\alpha, \qquad \alpha = \min\{\gamma_2-1, 2 (\gamma_1-1)\}> \gamma_1-1\,.
\end{equation}
Note now that by a telescopic argument, for any $L\ge L_0\ge L_\delta$, denoting $H_L=\sum_{\ell=0}^{L-1}\frac{1}{\ell+1}$, we have
\begin{align}\nonumber S_L(x) =& (1-x)^{1-\gamma_1} +  \sum_{\ell=0}^{L-1}\Big( \varepsilon_{\ell+1}(x)^{1-\gamma_1} - \varepsilon_{\ell}(x)^{1-\gamma_1} -A_0 - \frac{A_1}{\ell+1}\Big) + A_1(H_L - \log L) \\
\nonumber = &A_1(H_L - \log L )   - A_0 L_0 - A_1 H_{L_0}  \\
\nonumber & + \varepsilon_{L_0}(x)^{1-\gamma_1}+ \sum_{\ell=L_0}^{L-1}\Big( \varepsilon_{\ell+1}(x)^{1-\gamma_1} - \varepsilon_{\ell}(x)^{1-\gamma_1} -A_0 - \frac{A_1}{\ell+1}\Big)
\\
\nonumber = &A_1(H_L - \log L )   - A_0 L_0 - A_1 H_{L_0}   +\sum_{\ell=L_0}^{L-1}\Big( \varepsilon_{\ell+1}(x)^{1-\gamma_1} - \varepsilon_{\ell}(x)^{1-\gamma_1} -A_0 - A_1A_0 \varepsilon_\ell(x)^{\gamma_1-1}\Big)   \\
\nonumber & + \varepsilon_{L_0}(x)^{1-\gamma_1}+  A_1 \sum_{\ell=L_0}^{L-1}\Big( A_0\varepsilon_\ell(x)^{\gamma_1-1} - \frac{1}{\ell+1}\Big)
\\
\nonumber = &  A_1(H_L - \log L )   - A_0 L_0 - A_1 H_{L_0}   +\sum_{\ell=L_0}^{L-1} R_\ell(x)  + \varepsilon_{L_0}(x)^{1-\gamma_1}+  A_1 \sum_{\ell=L_0}^{L-1}\Big( A_0\varepsilon_\ell(x)^{\gamma_1-1} - \frac{1}{\ell+1}\Big)
\end{align}
Let us prove an upper bound for $S_L(x)$.
Since $A_0A_1> 0$, for $\delta$ small enough, by \eqref{R-bound} we have $$A_0A_1\varepsilon_\ell(x)^{\gamma_1-1} + R_\ell(x)\geq 0\,.$$
Thus, by \eqref{esatta} we have
$$ \varepsilon_\ell(x)^{1-\gamma_1} = \varepsilon_{L_0}(x)^{1-\gamma_1}+ \sum_{k=L_0}^{\ell-1} \varepsilon_{k+1}(x)^{1-\gamma_1} - \varepsilon_{k}(x)^{1-\gamma_1} \geq \varepsilon_{L_0}(x)^{1-\gamma_1} + A_0(\ell-L_0)\,.$$
and
\begin{equation}\label{eps-bound}
    \varepsilon_\ell(x) \leq \frac{1}{\Big(\varepsilon_{L_0}(x)^{1-\gamma_1} + A_0(\ell-L_0) \Big)^{\frac{1}{\gamma_1-1}} } \leq A_0^{-\frac{1}{\gamma_1-1}} (\ell-L_0)^{-\frac{1}{\gamma_1-1}}
\end{equation}
where the last inequality follows since $\varepsilon_{L_0}(x)\geq 0$. Using \eqref{R-bound}-\eqref{eps-bound} with $L_0=L_\delta$, we get
\begin{equation}
    |R_\ell(x)| \leq C (\ell-L_\delta)^{-\frac{\alpha}{\gamma_1-1}}\,\,,\label{boundforRwithdelta}
\end{equation}
$$ A_0 \varepsilon_{\ell}(x)^{\gamma_1-1} - \frac{1}{\ell+1} \leq \frac{1}{\ell-L_\delta} - \frac{1}{\ell+1} = \frac{L_\delta+1}{(\ell+1)(\ell-L_\delta)} \leq  C (\ell -L_\delta)^2\,.$$ Using the previous decomposition of $S_L$, the previous bounds and  since $H_L-\log L \leq 1 $, we obtain
$$ S_L(x) \leq A_1 - A_0 L_\delta - A_1 H_{L_\delta} +  C \zeta\Bigg(\frac{\alpha}{\gamma-1}\Bigg) + C \zeta(2)  + \varepsilon_{L_\delta}(x)^{1-\gamma_1} \le\,C (1-x)^{1-\gamma_1}$$
where the last inequality follows by $x\mapsto (1-x)(1-\kappa_{L_\delta}(x))^{-1}$ positive and continuous function.

\medskip
\noindent
Let us move to the lower bound.  Using~\eqref{esatta}, \eqref{eps-bound} and \eqref{boundforRwithdelta}, for $\ell\ge L_0\ge L_\delta$ we obtain
$$ \varepsilon_{\ell+1}(x)^{1-\gamma_1} -\varepsilon_\ell (x)^{1-\gamma_1} \leq A_0 + \frac{A_0A_1}{\varepsilon_{L_0}(x)^{1-\gamma_1} + A_0 (\ell-L_0)} + C \ell^{-\frac{\alpha}{\gamma_1-1}}\,,$$
where $C$ does not depend on $L_0$. Hence, using again a  telescopic argument, we have
\begin{align}\nonumber
 \varepsilon_L(x)^{1-\gamma_1} &\leq \varepsilon_{L_0}(x)^{1-\gamma_1} + A_0(L-L_0) + A_1 \log(L + A_0^{-1}\varepsilon_{L_0}(x)^{1-\gamma_1}) +C\\
 &\leq \varepsilon_{L_0}(x)^{1-\gamma_1} + A_0L + 2A_1 \log(L + A_0^{-1}\varepsilon_{L_0}(x)^{1-\gamma_1})\,,
\end{align}
where the last inequality holds for $L\ge L_0$ and $L_0$ large enough. Then, for $\ell\ge L_0\ge L_\delta$ and $\delta>0$ small enough we get
\begin{align*}
 A_0 \varepsilon_\ell(x)^{\gamma_1-1} - \frac{1}{\ell+1} &
    \geq \frac{A_0 (\ell+1) - \Big(\varepsilon_{L_0}(x)^{1-\gamma_1} + A_0\ell + 2A_1 \log(\ell + A_0^{-1}\varepsilon_{L_0}(x)^{1-\gamma_1})\Big)}{(\ell+1)\Big( \varepsilon_{L_0}(x)^{1-\gamma_1} + A_0\ell + 2A_1 \log(\ell + A_0^{-1}\varepsilon_{L_0}(x)^{1-\gamma_1}) \Big)}
    \\
    & \geq -\frac{\varepsilon_{L_0}(x)^{1-\gamma_1} +  2A_1 \log(\ell + A_0^{-1}\varepsilon_{L_0}(x)^{1-\gamma_1}) }{A_0\ell^2}\\
    &=  -\frac{2 A_1 \log(\ell) + \varepsilon_{L_0}(x)^{1-\gamma_1} +  2A_1 \log \Big( 1  + \frac{\varepsilon_{L_0}(x)^{1-\gamma_1}}{A_0 \ell}\Big) }{A_0\ell^2}\\
    &\geq  -\frac{2  A_1 \log(\ell)}{A_0\ell^2} - \varepsilon_{L_0}(x)^{1-\gamma_1} \frac{ 1 + \frac{2A_1}{A_0\ell}}{A_0\ell^2}\\
    &\geq  -\frac{2 A_1 \log(\ell)}{A_0\ell^2} - \varepsilon_{L_0}(x)^{1-\gamma_1} \frac{ 1 +2A_1A_0^{-1}}{A_0\ell^2}\,.
\end{align*}
where the second-last inequality follows by $\log(1+a)\le a$. Using the decomposition of $S_L$, since $H_L-\log L\geq 0 $, we obtain
$$ S_L(x)\geq - A_0 L_0 -A_1H_{L_0} - C \zeta\Big(\frac{\alpha}{\gamma-1}\Big) - \sum_{\ell=L_0}^\infty \frac{2A_1^2\log (\ell)}{A_0\ell^2} + \varepsilon_{L_0}(x)^{1-\gamma_1} \Bigg( 1 - \frac{A_1(A_0+2A_1)}{A_0^2} \sum_{\ell=L_0}^\infty \frac{1}{\ell^2}\Bigg)\,.$$
Now, since
$$ \sum_{\ell=L_0}^\infty \frac{1}{\ell^2}\leq \frac1{L_0 -1}, $$
we obtain
$$ S_L(x) \geq  -D + \frac{A_0^2(L_0-1) -A_1(A_0+2A_1)}{A_0^2 (L_0-1)}\varepsilon_{L_0}(x)^{1-\gamma_1} $$
and choosing $L_0$ large enough we get the lower bound \eqref{low=SL}. To conclude, we only have to prove that $S_L(x)$ converges pointwise to $S(x)$ for every $x\in[-1,1)$. But this immediately follows by the decomposition of $S_L$, the fact that $H_L-\log(L)$ converges to the Euler-Mascheroni constant $\gamma$ and the absolute convergence (for fixed $x$) of the two series in the decomposition of $S_L$ proved in the previous steps. In particular, the limit function $S$ is well defined as
\[
S(x):=A_1\gamma   - A_0 L_0 - A_1 H_{L_0}   +\sum_{\ell=L_0}^{\infty} R_\ell(x)  + \varepsilon_{L_0}(x)^{1-\gamma_1}+  A_1 \sum_{\ell=L_0}^{\infty}\Big( A_0\varepsilon_\ell(x)^{\gamma_1-1} - \frac{1}{\ell+1}\Big)\,.
\]
\end{proof}
\begin{proof}[Proof of the case $\gamma_2=2\gamma_1-1$]Using the same computations as in the previous cases and recalling that $\gamma_1-1=\gamma_2-\gamma_1$,  we have
\begin{align}
\label{risolutiva}\varepsilon_{\ell+1}(x)^{1-\gamma_1} = \varepsilon_{\ell}(x)^{1-\gamma_1}  + A_0 + A_0 A_1\varepsilon_\ell(x)^{\gamma_1-1}  + R_\ell(x)\,.
\end{align}
Moreover, one can analogously prove that $\exists$ $L_
\delta$,  $C$ and $
\alpha>\gamma_1-1$ such that for any $\ell\geq L_\delta$ we have
\begin{equation}\label{R-bound2}|R_\ell(x)|\leq  C\varepsilon_\ell(x)^\alpha  \,.
\end{equation}
We split the proof in three cases.

\smallskip \noindent
{$\boldsymbol{A_1> 0}.$}  The proof proceeds exactly as in the previous lemma.

\smallskip \noindent
{$\boldsymbol{A_1= 0}.$} Combining~\eqref{risolutiva} with~\eqref{R-bound2} we obtain, for $k\geq L_\delta$
$$ \varepsilon_{k+1}(x)^{1-\gamma_1} - \varepsilon_{k}(x)^{1-\gamma_1} \geq A_0 - C \varepsilon_k(x)^\alpha \geq A_0 - C \delta^\alpha $$
where the last inequality follows from the definition of $L_\delta$ and since $\kappa(x)>x$. A telescopic argument yields, for $\ell\geq L_\delta$,
$$ \varepsilon_\ell(x)^{1-\gamma_1} \geq \varepsilon_{L_\delta}(x)^{1-\gamma_1} + (A_0 - C\delta^\alpha)(L- L_\delta)\gtrsim(L-L_\delta) $$
and hence
$$ |R_\ell(x)|\leq C (L-L_\delta)^{-\frac{\alpha}{\gamma_1-1}}\,.$$
Using the decomposition of $S_L$ in the previous lemma and recalling that $A_1=0$ we obtain
$$ S_L(x) = \varepsilon_{L_\delta}(x)^{1-\gamma_1} -A_0 L_0 + \sum_{\ell=L_0}^{L-1} R_\ell(x) $$
and hence the claim, indeed
\begin{equation}\label{eq--11} \Big|\sum_{\ell=L_0}^{L-1} R_\ell(x)\Big| \leq C \zeta\Big(\frac{\alpha}{\gamma_1-1}\Big)\,.
\end{equation}
with $\alpha> \gamma_1-1$.

\smallskip \noindent
{$\boldsymbol{A_1< 0}.$} For $\delta$ small enough and $\ell\geq L_\delta$, we have $$A_0 A_1\varepsilon_{\ell}(x)^{\gamma_1-1} + R_\ell(x) \geq -C \varepsilon_\ell(x)^{\gamma_1-1} \geq - C\delta^{\gamma_1-1}$$
where $C>0$ is a constant independent from $\ell$ and hence we obtain again~\eqref{eq--11}. Moreover, a similar telescopic argument yields
$$ \varepsilon_{\ell}(x)^{1-\gamma_1} \leq \varepsilon_{L_\delta}(x)^{1-\gamma_1} + A_0 (\ell- L_\delta) \leq A_0 \ell + \varepsilon_{L_\delta}(x)^{1-\gamma_1} \,.$$
Then
$$ A_0 \varepsilon_\ell(x)^{\gamma_1-1} - \frac{1}{\ell+1} \geq -\frac{ \varepsilon_{L_\delta}(x)^{1-\gamma_1} }{(\ell+1)(A_0\ell + \varepsilon_{L_\delta}(x)^{1-\gamma_1})} \geq -\frac{\varepsilon_{L_\delta}(x)^{1-\gamma_1}}{A_0\ell^2}$$
and since $A_1<0$ and $H_L-\log L >0$ we have
\begin{align*}S_L(x) = &  A_1(H_L - \log L )   - A_0 L_0 - A_1 H_{L_0}   +\sum_{\ell=L_0}^{L-1} R_\ell(x)  + \varepsilon_{L_0}(x)^{1-\gamma_1}+  A_1 \sum_{\ell=L_0}^{L-1}\Big( A_0\varepsilon_\ell(x)^{\gamma_1-1} - \frac{1}{\ell+1}\Big) \\
&\leq - A_0 L_0 - A_1H_{L_0} + C \zeta\Big(\frac{\alpha}{\gamma_1-1}\Big) + \varepsilon_{L_0}(x)^{1-\gamma_1} \Bigg( 1 - \frac{A_1}{A_0} \sum_{\ell=L_0}^L \ell^{-2}\Bigg) \,.
\end{align*}

Let us move to the lowe bound. From the definition of $L_\delta$, we have $\kappa_{L_\delta}(x) \geq 1-\delta$ and hence for $\delta$ small enough, $\kappa_{L_\delta}(x)\geq \kappa(0)>0$. Thus, since $\kappa$ is increasing in $[0,1]$, we have
$$\varepsilon_{L}(x)^{1-\gamma_1} \geq \varepsilon_{L-L_\delta}(0)^{1-\gamma_1}\,.$$
Using~\eqref{risolutiva}, since $\varepsilon_L(0) \to 0$, we have
$$ \varepsilon_{L+1}(0)^{1-\gamma_1} - \varepsilon_{L}(0)^{1-\gamma_1} = A_0 + o(1)$$
and hence from Stoltz-Cesaro theorem we have
$$ \varepsilon_L(0)^{1-\gamma_1} \sim A_0 L\,.$$
Using again Stoltz-Cesaro, since
$$ \varepsilon_{L+1}(0)^{1-\gamma_1} -A_0(L+1) - \Big( \varepsilon_{L}(0)^{1-\gamma_1} -A_0 L\Big) = A_0A_1 \varepsilon_\ell(0)^{\gamma_1-1}  + R_L(0) \sim \frac{A_1}L\,,$$
we have
$$\frac{\varepsilon_{L}(0)^{1-\gamma_1} -A_0 L}{\log L} \to A_1  $$ and in particular for $L$ large enough, since $A_1<0$ we have
$$\varepsilon_L(x)^{1-\gamma_1}\geq  \varepsilon_{L-L_\delta}(0)^{1-\gamma_1} \geq  A_0 (L-L_\delta) +  2A_1\log (L- L_\delta)\,. $$
Thus
$$A_0 \varepsilon_\ell(x)^{\gamma_1-1}  - \frac{1}{\ell+1}\leq  \frac{A_0(\ell+1) - A_0(\ell-L_\delta) - 2A_1\log(\ell-L_\delta)}{(\ell+1)\Big( A_0 (\ell-L_\delta) + 2A_1 \log(L-L_\delta)) \Big)} \leq
-\frac{2A_1 \log(\ell) }{A_0\ell  (\ell-L_\delta) }
$$
and hence the claim.

\end{proof}

\begin{proof}[Proof of~\Cref{fixed-point-sparse}] Setting $\beta_0 = A_0^{-\rho}$ and $\beta_1 = \beta_0\frac{A_1\rho}{A_0}$. Recalling the definition of $S_L$ we have
\begin{align} \nonumber
&1-\kappa_L(x)
  = \Big( A_0 L  + A_1 \log L + S_L(x)\Big)^{-\rho}= \beta_0 L^{-\rho} \Big( 1 + \frac{A_1}{A_0} \frac{\log L} L + \frac{S_L(x)}{A_0 L}\Big)^{-\rho}\\
&\quad    \label{taynew} = \beta_0 L^{-\rho} \Bigg( 1 - \frac{A_1\rho}{A_0} \frac{\log L} L - \frac{S_L(x)\rho}{A_0 L} + \frac{\rho(\rho+1)\big( A_1\log L + S_L(x)\big)^2}{2A_0^2 L^2(1+\xi_{L,x})^{\rho+2}}  \Bigg)
\end{align}
where the last identity holds for $\xi_{L,x}$  in the interval with endpoints $0$ and $y_{L,x}:=\frac{A_1 \log L + S_L(x) }{A_0 L}$.
Using the definition of $\beta_L$ (cf..~\eqref{betaL}) we obtain
\begin{equation}\label{beta_L-equi} \beta_L(x) =  \frac{\beta_0\rho}{A_0} S_L(x)  - \frac{\rho(\rho+1) \beta_0}{2A_0} \frac{(A_1\log L + S_L(x))^2}{L(1+\xi_{L,x})^{\rho+2}}  \,
\end{equation}
and hence, since $S_L(x)\to S(x)$,  we have the pointwise convergence of $\beta$ for every fixed $x\in [-1,1)$. To conclude the proof we need to prove the upper and lower bound for $\beta_L$. From~\eqref{upp=SL},  we obtain
$$ y_{L,x} \geq \frac{A_1 \log L - D}{A_0L} \geq -\frac{1}2 $$
for $L\geq L_1$ and in particular,
$$ \xi_{L,x}  \geq \min\{ 0, y_{L,x}\}\geq -\frac 1 2 \,.$$
Plugging this into~\eqref{beta_L-equi}   we obtain,
$$ \beta_L(x) \gtrsim  S_L(x) \Bigg( 1 -2CA_1\frac{\log L}L - C\frac{S_L(x)}L\Bigg) -C_1A_1^2 \frac{(\log L)^2}{L}$$
$$ \beta_L(x) \lesssim   S_L(x) $$
for some positive constant $C$. Using~\eqref{upp=SL} and~\eqref{low=SL} (or~\eqref{upp=SL-even} and~\eqref{low=SL-even}  when $\kappa$ is even) , provided that $x$ is chosen sufficiently small so $S_L(x)>0$ (possible from ~\eqref{low=SL}--\eqref{low=SL-even}) we obtain~\eqref{beta-lower}  and
$$ \beta_L(x)\lesssim (1-x^2)^{1-\gamma_1} \,.$$
Moreover, there exists  $L_0$ such that $\kappa_{L_0}(x)\in [0,1]$. Since $\kappa$ is increasing on $[0,1]$, also $\beta_L$ is increasing. Then
$$ \beta_L(x)  = \beta_{L-L_0}(\beta_{L_0}(x))\geq \beta_{L-L_0}(0) \to  \beta(0)$$
and hence~\eqref{beta-upper}.

\end{proof}

\subsection{Background results on holomorphic functions and real-analytic limits}\label{sec:holo}

The goal of this section is to provide a brief and self-contained introduction to some standard results on holomorphic and real-analytic functions. The purpose is twofold: to recall, in elementary terms, what holomorphic functions are and why they are rigid; to show how this rigidity implies that, under suitable \emph{structured assumptions}, a point-wise limit of real analytic functions remains real analytic (whereas this is false in general). For all the missing details, see, e.g., the classical references \cite{Conway78,Rudin1987}.

\begin{definition}
	A function $f:(-1,1)\to\mathbb{R}$ is \emph{real analytic} if, for every $x_0\in(-1,1)$, it can be written in a neighborhood of $x_0$ as a (absolutely) convergent\footnote{Recall that convergence of a power series with radius $r$ automatically implies absolute convergence for $|x-x_0|<r$.} power series
	\[
	f(x)=\sum_{n\ge 0} c_n (x-x_0)^n \qquad \text{for } |x-x_0|<r,
	\]
	for some $r>0$. The largest such $r$ is called the \emph{radius of convergence} at $x_0$. \\
\end{definition}

\begin{definition}
Let $D \subset \mathbb{C}$ be open.
A function $F : D \to \mathbb{C}$ is \emph{holomorphic} (or \emph{complex analytic})
if for every $z_0 \in D$ there exists $r>0$ and a sequence of complex coefficients $(a_n)_{n\ge 0}$
such that
\[
F(z) = \sum_{n=0}^\infty a_n (z-z_0)^n
\qquad \text{for all } |z-z_0| < r.
\]

\end{definition}
\begin{remark}[Power series and holomorphic functions]
	It is standard for every power series
	\begin{equation}
	    \label{eq:power series}
        F(z)=\sum_{n\ge 0} a_n z^n\,,
	\end{equation}
	with radius of convergence $r>0$, is holomorphic on the disk $\{z\in\C:|z|<r\}$.
	Vice versa, if $F$ is holomorphic on $\{z\in\C:|z|<r\}$, then \eqref{eq:power series} holds with radius at least $r$ (see, e.g. \cite[p.72]{Conway78}).
\end{remark}
\begin{remark}[Real analytic and holomorphic extension]\label{radius}
    If is a real analytic function $f$ with radius of convergence $r$ at $x_0=0$, the same series defines a holomorphic extension to the unit disk
	\[
	\mathbb{D}:=\{z\in\mathbb{C}:|z|<1\}\,.
	\]
\end{remark}

Holomorphic (complex analytic) functions enjoy strong rigidity properties that have no analog
in real analysis.
In particular, boundedness on compact sets and convergence on small subsets
often force global convergence and preservation of analyticity. These principles will be the
key tools used to prove the convergence result for real analytic functions. The concept of boundedness that we need to introduce is the following.

\begin{definition}[Locally uniformly bounded family]
	A family $\mathcal{F}$ of functions $F : D \to \mathbb{C}$ is said to be \emph{locally uniformly bounded}
	if for every compact $K \subset D$ there is a constant $C_K < \infty$ such that
	\[
	\sup_{F \in \mathcal{F}} \sup_{z \in K} |F(z)| \le C_K.
	\]
\end{definition}

Vitali's theorem shows that pointwise convergence of locally uniformly bounded holomorphic functions on a set with
an accumulation point forces uniform convergence.

\begin{theorem}[Vitali]
	Let $(F_n)_{n\in \N}$ be a sequence of holomorphic functions on a open set $D \subset \mathbb{C}$ that is locally
	uniformly bounded. If $F_n$ converges pointwise on a set $E \subset D$ having an accumulation point
	in $D$, then $F_n$ converges uniformly on compact subsets of $D$ to a holomorphic function.
\end{theorem}

We are ready to prove the result that will be applied in our proofs. It allows us to show that, under suitable assumptions, a sequence of real analytic functions with a certain radius of convergence converges to a real analytic function with the same radius of convergence.

\begin{proposition}
\label{prop:real_analytic_limit}
Let $(f_n)_{n\in\mathbb N}$ be a sequence of real-analytic functions on $(-1,1)$. Assume that
\begin{enumerate}
\item $f_n(x)\to f(x)$ pointwise for every $x\in(-1,1)$;
\item The radius of convergence of $f_n$ at $0$ is at least $1$;
\item For every $n$ we have $f_n^{(k)}(0)\ge0$ $\forall\,k\ge k_0$, with $k_0$ independent of $n$.
\end{enumerate}
Then $f$ is real analytic on $(-1,1)$, with $f^{(k)}(0)\ge0$ $\forall k\ge k_0$ and the radius of convergence of $f$ at $0$ is at least $1$.
\end{proposition}
\begin{proof}
By 2. $f_n$ has a power series expansion and a holomorphic extension $F_n$ on $\mathbb D$ of the form
\[
        f_n(x)=\sum_{k=0}^\infty f_n^{(k)}(0)\,\frac{x^k}{k!}\,,\quad\quad |x|<1\,,\quad\quad F_n(z):=\sum_{k=0}^\infty f_n^{(k)}(0)\,\frac{z^k}{k!},\qquad |z|<1\,.
\]
Fix now $r\in(0,1)$. By assumption 3., setting $c_r:=\sum_{k=0}^{k_0}|f^{(k)}_n(0)| \frac{r^k}{k!}$, we have
\[
\sup_n\,\sup_{|z|<r}|F_n(z)|\le \sum_{k=0}^{k_0}(|f^{(k)}_n(0)|-f^{(k)}_n(0))+f_n(r)<\infty
\]
Since $f_n(r)\to f(r)$ by assumption 1., the sequence $(f_n(r))_n$ is bounded by a constant $c_r$ independent of $n$ and so $(F_n)_n$ is a locally uniformly bounded family of holomorphic functions on $\mathbb D$. Moreover, again by assumption 1. $(F_n)_n$ converges pointwise to $F$ on the set $(-1,1)\subset\mathbb D$, which contains many accumulation points. By Vitali's theorem, we have then that $F_n$ converges uniformly on the compact sets of $\mathbb D$ to a unique holomorphic function $F$ on $\mathbb D$ with radius at least $1$ at $0$ (see Remark \ref{radius}). Thus, by Cauchy's integral formula also all the derivatives converge uniformly con compact sets and also the coefficients $f^{(k)}_n(0)$. In particular, $F(z)=f(z)$ for $z\in(-1,1)$ and so $f$ is real analytic with radius at least~$1$, with $f^{(k)}(0)\ge0$ $\forall k\ge k_0$.
\end{proof}

\begin{remark}
	The argument above yields real analyticity of the limit function in the open interval $(-1,1)$,
	but provides no information on the analyticity at the boundary points $\pm1$, where
	singularities of the limit may occur.
	For this reason, an analogous statement on the closed interval $[-1,1]$ is false without stronger
	assumptions.
\end{remark}
\section*{Acknowledgements}
The authors are associated with INdAM (Istituto Nazionale di Alta Matematica “Francesco Severi”) and the GNAMPA group.
This work was partially supported by the MUR Excellence Department Project MatMod@TOV, awarded to the Department of Mathematics, University of Rome Tor Vergata (CUP E83C18000100006). We also acknowledge financial support from the MUR 2022 PRIN project GRAFIA (project code 202284Z9E4) and the University of Rome Tor Vergata Research Project METRO (CUP E83C25000630005).

\bibliographystyle{abbrvurl}
 \bibliography{ref}


\end{document}